\documentclass[a4paper
]{amsart}

\usepackage{mathrsfs}
\usepackage{todonotes}
\usepackage{enumitem}
\usepackage{array}
\usepackage{amssymb}
\usepackage{mathtools}
\usepackage{faktor}
\usepackage{xcolor}
\usepackage{tikz}
\usetikzlibrary{arrows}
\usetikzlibrary{decorations.pathmorphing}
\tikzset{snake it/.style={decorate, decoration=snake}}
\usepackage{tikz-cd}
\usepackage{amsthm}
\usepackage{thmtools}
\usepackage[breakable]{tcolorbox}
\usepackage[utf8]{inputenc}
\usepackage{csquotes}
\usepackage[english]{babel}
\usepackage{slashed}

\usepackage{newunicodechar}
\usepackage[style=alphabetic,backend=biber,maxnames=5,maxalphanames=5,doi=true,isbn=false,url=false]{biblatex}

\usepackage[hypertexnames=false]{hyperref} 
\hypersetup{urlcolor=blue, citecolor=blue, linkcolor=blue, colorlinks=true}
\usepackage[capitalise,nameinlink]{cleveref}
\crefname{enumi}{Theorem}{Main theorems}
\Crefname{enumi}{Theorem}{Main theorems}
\bibliography{literature.bib}
\renewbibmacro{in:}{}
\renewbibmacro*{issue+date}{%
  \ifboolexpr{not test {\iffieldundef{year}} or not test {\iffieldundef{issue}}}
    {\printtext[parens]{%
       \iffieldundef{issue}
         {\usebibmacro{date}}
         {\printfield{issue}%
          \setunit*{\addspace}%
          \usebibmacro{date}}}}
    {}%
  \newunit}

\newtheorem*{maintheorem}{Theorem}

\newtheorem{theorem}{Theorem}[section]
\newtheorem{conjecture}[theorem]{Conjecture}
\newtheorem{lemma}[theorem]{Lemma}
\newtheorem{proposition}[theorem]{Proposition}
\newtheorem{corollary}[theorem]{Corollary}

\theoremstyle{definition}
\newtheorem{remark}[theorem]{Remark}

\newtheorem{definition}[theorem]{Definition} 
 
\newtheorem{construction}[theorem]{Construction} 
 
\newtheorem*{convention}{Convention}

\numberwithin{equation}{section}

\definecolor{grau1}{RGB}{100, 100, 100}
\definecolor{grau2}{RGB}{150, 150, 150}
\definecolor{grau3}{RGB}{200, 200, 200}

\newcommand{\oldcomment}[1]{
\begin{tcolorbox}[breakable]
	Old version in comments
\end{tcolorbox}
}

\newcommand{\proofcomment}[1]{
\begin{tcolorbox}[breakable]
	Proof in comments
\end{tcolorbox}
}

\DeclareMathAlphabet{\mathpzc}{OT1}{pzc}{m}{it}

\let\oldtocsection=\tocsection
\let\oldtocsubsection=\tocsubsection
\let\oldtocsubsubsection=\tocsubsubsection
\renewcommand{\tocsection}[2]{\hspace{0em}\oldtocsection{#1}{#2}}
\renewcommand{\tocsubsection}[2]{\hspace{1em}\oldtocsubsection{#1}{#2}}
\renewcommand{\tocsubsubsection}[2]{\hspace{2em}\oldtocsubsubsection{#1}{#2}}
\newcommand{\nocontentsline}[3]{}
\newcommand{\tocless}[2]{\bgroup\let\addcontentsline=\nocontentsline#1{#2}\egroup}

\setenumerate[1]{leftmargin=*,labelindent=0pt,label=(\roman*),}
\newcommand{\II}{\mathrm{II}}

\newcommand{\eps}{\varepsilon}
\renewcommand{\phi}{\varphi}

\newcommand{\im}{\mathrm{im}} 

\newcommand{\scal}{\mathrm{scal}}

\newcommand{\vol}{\mathrm{vol}}
\newcommand{\diam}{\mathrm{diam}}
\newcommand{\dvol}{\mathrm{dvol}}

\DeclareMathOperator{\id}{\mathrm{id}}

\newcommand{\pr}{\mathrm{pr}}

\newcommand{\supp}{\mathrm{supp}}

\newcommand{\embeds}{\hookrightarrow}

\DeclarePairedDelimiter{\scpr}{\langle}{\rangle}

\newcommand{\calN}{\mathcal{N}}

\newcommand{\calS}{\mathcal{S}}

\newcommand{\bbN}{\mathbb{N}}
\newcommand{\bbH}{\mathbb{H}}

\newcommand{\bbR}{\mathbb{R}}

\newcommand{\dt}{\mathrm{d}t}
\newcommand{\dtau}{\mathrm{d}\tau}
\newcommand{\dx}{\mathrm{d}x}

\newcommand{\tr}{\operatorname{tr}}
\newcommand{\Ric}{\mathrm{Ric}}

\begin{document}

\author[Georg Frenck]{Georg Frenck}
\email{\href{mailto:georg.frenck@math.uni-augsburg.de}{georg.frenck@math.uni-augsburg.de}}
\urladdr{\href{http://frenck.net/Math}{frenck.net/math}\vspace{-.8em}}
\address{Universität Augsburg, Universitätsstr.~14, 86159 Augsburg, Germany}

\author[Bernhard Hanke]{Bernhard Hanke}
\email{\href{mailto:hanke@math.uni-augsburg.de}{hanke@math.uni-augsburg.de}}
\urladdr{\href{https://www.uni-augsburg.de/de/fakultaet/mntf/math/prof/diff/team/bernhard-hanke/}{math.uni-augsburg.de/hanke}\vspace{-.8em}}

\address{Universität Augsburg, Universitätsstr.~14, 86159 Augsburg, Germany}

\author[Sven Hirsch]{Sven Hirsch}
\email{\href{mailto:}{sven.hirsch@columbia.edu}}
\urladdr{\href{https://svenhirsch.com/}{svenhirsch.com}\vspace{-.8em}}
\address{Columbia University, 2990 Broadway, New York NY 10027, USA}

% \subjclass[2020]{53C21; 53C23; 53C27; 57R65}
\keywords{Fill-ins, scalar curvature, total mean curvature, surgery; positive mass theorem}

\title[Surgery and total mean curvature]{Surgery and total mean curvature}

\begin{abstract} 
We prove Gromov's conjecture on the total mean curvature of fill-ins in various cases.
Our methods are based on surgery to reduce the statement to fill-ins of spheres, which can be treated by instances of the positive mass theorem. 
For spin fill-ins, where we permit the mean curvature to take negative values, we build on a classical surgery result of Lawson--Michelsohn and a recent positive mass theorem with creases by Kazaras--Khuri--Lin.
For non-spin fill-ins of spin manifolds, where we assume the mean curvature to be non-negative, we develop a novel quantitative surgery process to reduce the general situation to a result of Shi--Wang--Wei.
We also treat the case of fill-ins of non-spin manifolds, provided there is a fixed positive lower bound on the mean curvature.
\end{abstract}

\maketitle
\tableofcontents

\section{Introduction}\label{sec:introduction}
Let $(M,g_M)$ be a closed Riemannian manifold, possibly disconnected, with scalar curvature $\scal_{g_M} \colon M \to \mathbb{R}$.
A \emph{fill-in} of $(M,g_M)$ is a compact 
Riemannian manifold $(\Omega,g_\Omega)$ such that $(\partial\Omega,g_{\partial\Omega})$ is isometric to $(M,g_M)$, where $g_{\partial\Omega}$ denotes the metric induced by $g_\Omega$ on $\partial\Omega$.
We write $H_{g_\Omega} \colon \partial \Omega \to \mathbb{R}$ for the  mean curvature of $\partial \Omega \subset (\Omega,g_\Omega)$ with respect to the exterior normal.
By convention, this means  that the mean curvature of the closed unit ball in flat Euclidean space $\mathbb{R}^n$ is equal to $n-1$.

\medskip

Inspired by the positive mass theorem in general relativity, Gromov conjectured that for fixed $(M, g_M)$, the {\em total mean curvature} of $(\Omega, g_{\Omega})$ cannot be arbitrarily large under a uniform lower scalar curvature bound of $(\Omega, g_{\Omega})$: 

\begin{conjecture}[{\cite[p.~232]{Gromov2023FourLectures}}] \label{conj:gromov}
    Let $(M, g_M)$ be a closed Riemannian manifold.
    Then for every $\sigma \in \bbR$, there exists a constant $\Lambda(M,g_M, \sigma) \in \mathbb{R}_+$ such that for every fill-in $(\Omega, g_{\Omega})$ of $(M, g_M)$ satisfying $\scal_{g_\Omega}  \ge \sigma$, one has
    \[
    \int_{\partial \Omega} H_{g_\Omega} {\rm dvol}_{g_{\partial \Omega}} \le \Lambda(M,g_M,\sigma).
    \]
\end{conjecture}

The following result verifies this conjecture for various cases, with a constant that also depends on a lower mean curvature bound. 
 
\begin{maintheorem} \label{thm:Main}
    Let $(M,g_M)$ be a closed Riemannian manifold of dimension $n$ and let $\sigma,\kappa \in \bbR$.
    Then there exists a constant $\Lambda = \Lambda(M,g_M,\sigma,\kappa) > 0$ with the following property: 
    If  $(\Omega,g_\Omega)$ is a fill-in of $(M,g_M)$ such that: 
    \begin{itemize}
        \item $H_{g_\Omega}\ge\kappa$,
        \item $\scal(g_\Omega)\ge\sigma$,
    \end{itemize}
    and such that one of the following conditions holds:
    \begin{enumerate}[label=\textbf{(\Alph*)}]
        \item \label{thm:mainC} $\kappa>0$ and $n\le 6$,
        \item \label{thm:mainB} $\kappa=0$, $n\le 6$ and $M$ admits a spin structure,
        \item \label{thm:mainA} $\Omega$ admits a spin structure,
    \end{enumerate}
    \medskip
    then
    \[
      \int_{M} H_{g_\Omega} {\rm dvol}_{g_{\partial \Omega}} \leq \Lambda(M,g_M, \sigma, \kappa).
    \]
\end{maintheorem}

The dependence of $\Lambda$ on the Riemannian metric $g_M$ is unavoidable, even if $\vol(M)$ and $\diam(M)$ are uniformly bounded, see \cref{fig:weird-pimple} and \cref{sec:bad-constant}.
The explicit dependence of $\Lambda$ in terms of $\sigma$ and $\kappa$ is worked out in \eqref{constant dependency general}.

\begin{figure}[ht]
    \begin{tikzpicture}
        \node at (0,0) {\includegraphics[width=.95\textwidth]{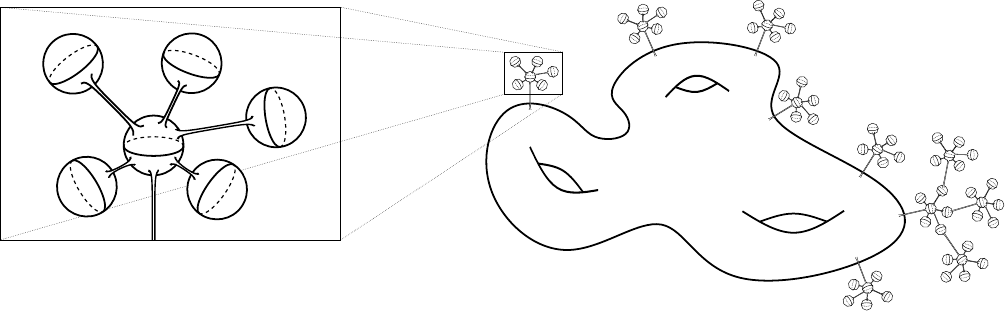}};
    \end{tikzpicture}
    \caption{A Riemannian manifold with large $\Lambda(M,g_M,0,0)$.}\label{fig:weird-pimple}
\end{figure}

Assuming $\kappa \geq 0$, this result has been proven before in the following cases: By Shi--Wang--Wei \cite{Shi-Wang-Wei} if $M$ is diffeomorphic to the $n$-sphere and for orientable fill-ins, based upon Shi--Tam's proof for the positivity of the Brown--York mass \cite{shi-tam-manifolds-with-boundary} and the hyperboloidal positive mass theorem (PMT) \cite{Wang2001AHmass, CJL, ACG2008}; and by Wang \cite{Wang2024FillInsTori} for tori $T^n$ and fill-ins $D^2 \times S^{n-1}$, using Brendle--Hung's \cite{BrendleHung2024SystolicHorowitzMyers} work on the Horowitz--Myers PMT.

\medskip

Christian Bär in \cite{Baer2026UpperBound} has presented  a different proof of \cref{thm:mainA} using spinorial methods.
\cref{thm:mainA} and Bär's result provide the first upper bounds on the total mean curvature to treat fill-ins whose mean curvature assumes negative values.

\begin{convention} Let  $(\Omega, g_{\Omega})$ be a fill-in of $(M,g_M)$ as in \cref{thm:mainC}, \cref{thm:mainB} or \cref{thm:mainA}. 
Then the orientation double cover  $(\overline \Omega, g_{\overline \Omega})$ of $(\Omega, g_{\Omega})$ is a fill-in of the orientation double cover of $(M,g_M)$ and the assumptions in the respective theorem are still satisfied.
Furthermore, $\int_{\partial \overline \Omega} H_{g_{\overline \Omega}} \dvol_{g_{\partial \overline \Omega}} = 2 \int_{\partial  \Omega} H_{g_\Omega} \dvol_{g_{\partial \Omega}}$.
Therefore, in our proofs, we can restrict to oriented  $(M, g_M)$ and $(\Omega, g_\Omega)$ with  an orientation preserving Riemannian isometry $(\partial\Omega, g_{\partial\Omega}) \cong (M, g_M)$.
These assumptions will be used throughout the rest of the paper.
\end{convention}

\subsection{Proof outline}
Our approach is different from the previous ones.
By topological arguments, we can assume that we can pass from $M$ to the sphere $S^n$ through surgeries of codimension $k\ge 2$. 
At the same time, we pass from a fill-in of $M$ to a fill-in of $S^n$ via the corresponding handle attachments. 
Working with handles of sufficiently small width and a pigeonhole argument, this ensures that the total mean curvature of the original fill-in of $M$ transfers to the total mean curvature of a fill-in of $S^n$, up to a controlled error. 

\medskip

Relying on the classical surgery result of Lawson--Michelsohn \cite{Lawson_Michelsoh} for positive mean curvature, this strategy can be carried out as stated when restricting to fill-ins $\Omega$ of $M$ with $H_{g_\Omega} \geq \kappa$ for some fixed $\kappa > 0$.
Using the result of Shi--Wang--Wei \cite{Shi-Wang-Wei}, this leads to a proof of \cref{thm:mainC}, which we present in \cref{sec:fixed lower H bound}.

\medskip

The same strategy also applies to \cref{thm:mainA}, when working with a positive mass theorem for manifolds with creases due to \cite{Lin, KazarasKhuriLin2025}.
The surgery part is even simpler in this case, as the result of Lawson--Michelsohn is not needed. 

\medskip 

From a technical point of view, the proof of \cref{thm:mainB} is the most challenging task.
The majority of this paper is devoted to addressing this issue.
Using surgery, we again reduce the general case to $M = S^n$ treated by Shi--Wang--Wei.
However, if the pointwise lower mean curvature bound of different fill-ins $\Omega$ is not uniformly bounded below by a positive constant, the handles have to be chosen thinner and thinner to ensure that the mean curvature for the resulting fill-in of $S^n$ is nonnegative, which is required in \cite{Shi-Wang-Wei}.
This leads to infinitely many metrics on the resulting $S^n$ and consequently, Shi--Wang--Wei's result does not yield a uniform total mean curvature upper bound.

\medskip

This difficulty is overcome by the following strategy (see \cref{sec:useful-observation}): Shi--Wang--Wei's upper bound can be chosen uniformly for a family of metrics on $S^n$, provided that each member of the family can be connected to a fixed metric on $S^n$ along a monotonically increasing path of Riemannian metrics with lower scalar curvature bounds that are independent of the given metric in the family (see \cref{prop:criterion-hssw} for the precise statement). 
We will develop a quantitative surgery construction, see \cref{subsec:quantitative},  that provides the required control of the surgery handles.
Proceeding by induction on the number of surgery steps finishes the proof for \cref{thm:mainB}. 

\begin{remark} \cref{thm:mainC} and \cref{thm:mainB} hold for any $n$ for which the hyperboloidal PMT holds for oriented manifolds in dimension $n+1$.
\end{remark}

\subsection{Quantitative surgery}
\label{subsec:quantitative}

Our main tool for establishing \cref{thm:mainB} is a new quantitative surgery procedure.
This encompasses and strengthens the classical surgery result of Gromov--Lawson \cite{GromovLawson1980} and has connections to the extrinsic surgery result of Lawson--Michelsohn \cite{Lawson_Michelsoh}.

\medskip

To explain this construction, let $\calS$ be a compact submanifold in $(M,g_M)$ of codimension $k$, and let $\varrho > 0$ be a small parameter. 
A {\em Gromov--Lawson handle} of size $\varrho$ is a hypersurface $\Sigma_\varrho$ contained in $(M\times[0,1], dt^2 + g_M)$  which coincides with $M=M\times\{0\}$ outside the $2\sqrt{\varrho}$-tubular neighbourhood of $\calS$ in $M$, becomes cylindrical of width $\varrho^4$ at height $2 \varrho$ and is invariant under the fibrewise $\mathrm{O}(k)$-symmetry of the normal bundle of $\calS$ in $M$, see \cref{fig:handle_intro}.

\medskip

Let $\scal_\varrho$ be the scalar curvature of $\Sigma_{\varrho}$ and let $H_{\varrho}$ be the mean curvature of $\Sigma_\varrho$.
Furthermore, let $s$ be the minimum of the scalar curvature of $(M,g_M)$ on $\calS$. 
The crucial properties of our construction are:
\begin{itemize}
\item if $k \geq 2$, then $(H_\varrho)_-$ converges to zero as $\varrho\to 0$, 
\item if $k \geq 3$, then $(\scal_\varrho- s)_-$ converges to zero as $\varrho \to 0$.
\end{itemize}
By performing a conformal transformation arbitrarily close to the identity, we can ensure that the mean curvature is positive for sufficiently small $\varrho$.

\medskip

The detailed construction in the special case where $\calS$ is a point in flat Euclidean space $\mathbb{R}^n$ is carried out in \cref{sec:euclidean-handles}.
For general $M$ and $\calS$, error terms arise, which will be controlled in \cref{sec:handles}.

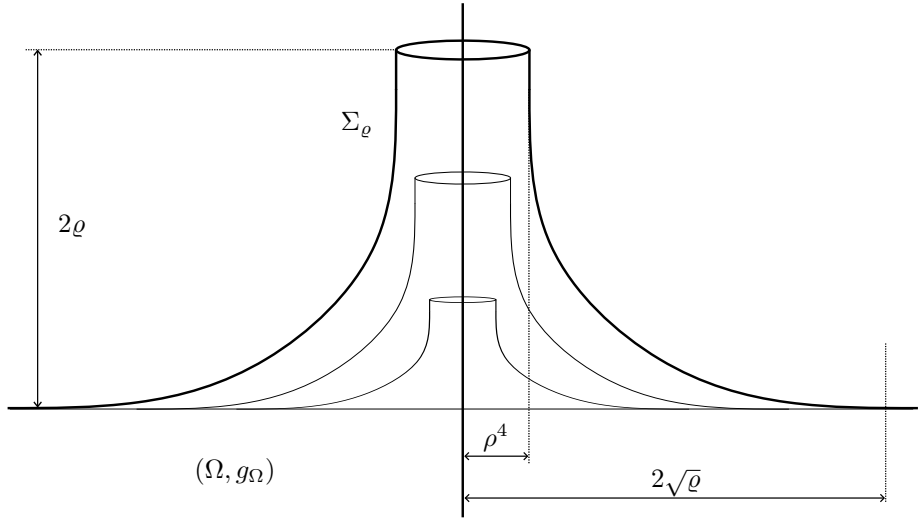
\begin{figure}[ht]
    \begin{tikzpicture}
        % \grid
        \node at (0,0) {\includegraphics[width=.95\textwidth]{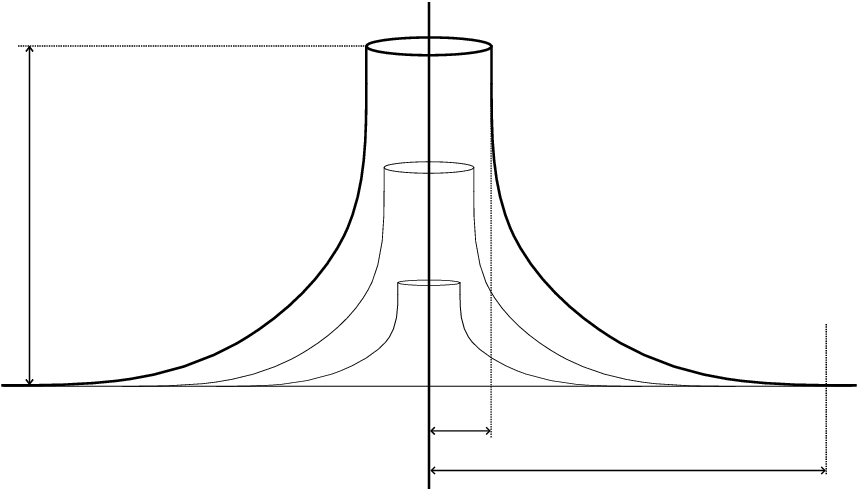}};
        \node at (0.435,-2.35) {$\rho^4$};
        \node at (2.8,-2.95) {$2 \sqrt{\varrho}$};
        \node at (-5.15,0.45) {$2 \varrho$};
        \node at (-3,-2.8) {$(\Omega,g_\Omega)$};
        \node at (-1.4,1.8) {$\Sigma_\varrho$};
       \end{tikzpicture}
    \caption{A cross-section of a family of Gromov--Lawson handles $\Sigma_\varrho$ attached to $\Omega $ along $M=\partial \Omega$. The origin corresponds to the intersection of the cross-section with $\calS$.}\label{fig:handle_intro}
\end{figure}

\medskip

Our construction moreover provides precise control between handles of different sizes $\varrho$.  
In particular, the heuristic observation (see Figure \ref{fig:handle_intro}) that the metrics on $\Sigma_\varrho$ are monotonically increasing in $\varrho$ can be made rigorous, while maintaining a uniform lower bound on the scalar curvature of $\Sigma_\varrho$.
The last property requires working with  handles of codimension at least $3$, which accounts for the spin assumption for $M$ in \cref{thm:mainB}.

\medskip 

Our computations are fully explicit, and the handles are constructed from first principles, yielding an independent approach to  \cite[Theorem A]{GromovLawson1980} with explicit geometric bounds.
Consequently, our results also encompass variants of the conclusions of several earlier works, including  
\cite{sweeney2025examples, sweeney2025new, BasilioDodziukSormani2017, Dodziuk2020, BasilioSormani2021, BasilioKazarasSormani2020}; compare \cref{rem:GL}.

\medskip

Our results also allows us to define a new quasi-local mass in general relativity with several particularly appealing properties. 
This will be discussed in \cref{sec:quasi-local}.

\medskip

\noindent \textbf{Acknowledgements:} 
BH and SH wish to thank Simon Brendle for useful discussions and his interest in this work. 
BH and SH are grateful to the University of Augsburg, to Columbia University, to the Lonavala Geometry Festival, and to the Chennai Mathematical Institute for their hospitality and ideal working conditions. 
BH acknowledges financial support by Columbia University and the DFG-SPP 2026 ``Geometry at Infinity''.
SH was supported in part by an AMS–Simons Travel Grant (award no. 210022).

\section{Surgery for fill-ins with fixed lower mean curvature bounds} \label{sec:fixed lower H bound}
Applying a surgery procedure, we will reduce the general case of \cref{thm:mainC} to (oriented) fill-ins of the $n$-sphere. 
Since this case has been treated in \cite{Shi-Wang-Wei}, this finishes the proof of \cref{thm:mainC}.
A similar argument applies to \cref{thm:mainA}, where the case of spheres will be postponed to \cref{sec:spin_fill_in}.

\begin{proof}[Proof of \cref{thm:mainC}]
Let $(M,g_M)$ and  $\sigma \in \bbR, \kappa >0$ be as in \cref{thm:mainC}.
Put $n := \dim M$.
By taking products with standard round spheres, we can assume without loss of generality that $n \geq 5$. 
 Let $B$ be a fixed compact oriented manifold such that $\partial B = M$ and so that $B$ induces the reverse orientation on $M$. 

\medskip

Removing a small open disc from the interior of $B$, we obtain a compact oriented bordism $W$ from $M$ to $S^{n}$. 
Annihilating $\pi_0(W)$ and $\pi_1(W)$  by  surgeries in the interior of $W$ (using $n \geq 5$), we can assume that  the inclusion $S^n \hookrightarrow \partial W \subset W$ is $1$-connected. 
Since $n \geq 5$, this implies that $(W, M)$ has a relative handle decomposition with handles of codimension at least $2$.

\medskip

Let $\ell \geq 0$ be the number of handles in this handle decomposition of $W$.
The proof of \cref{thm:mainC} is by induction on $\ell$. 
The case $\ell = 0$ is covered by \cite[Theorem 4.1 and Remark 4.2]{Shi-Wang-Wei}.
Now assume by induction that  \cref{thm:mainC} holds for $\ell \geq 0$ and that the number of handles in $W$ is equal to $\ell + 1$. 

\medskip

Let $(\Omega, g_{\Omega})$ be a fill-in of $(M, g_M)$ satisfying $\scal_{g_{\Omega}} \geq \sigma$ and $H_{g_{\Omega}} \geq \kappa$.
Let $\omega \colon (\partial \Omega, g_{\partial \Omega}) \cong (M, g_M)$ be an (orientation preserving) Riemannian isometry.
For simplicity of exposition, we will assume that $\partial \Omega = M$ and $\omega = \id$.
In the general case, the isometry $\omega$ has to be inserted when gluing subsets of $B$ to $\Omega$. 
This implies that $\Omega \cup_M B$ is oriented.

\medskip 

Let $\psi \colon S^{n-k} \times D^{k} \hookrightarrow M$, $k \geq 2$, be the surgery datum for the first handle $T$ in $W$ and let  $\widetilde M = \partial (\Omega \cup T)$ be obtained from $M$ by surgery along $\psi$.
Note that the induction hypothesis applies to $\widetilde M$.
There exist smooth embeddings $\psi_i \colon S^{n-k} \times D^{k} \hookrightarrow M$, $i = 1,2$, that are both isotopic to $\psi$ and have disjoint images, see \cref{fig:displacements-of-psi}.
With this ``doubling construction'' we can estimate the total mean curvature after the surgery by a pigeonhole argument, see \eqref{eq:pigeonhole} below. 

\medskip

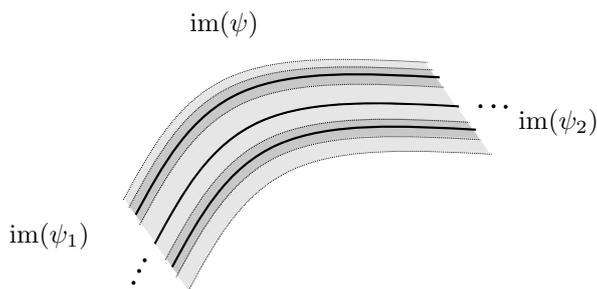
\begin{figure}[ht]
  \begin{tikzpicture}
    % \grid
    \node at (0,0) {\includegraphics[width=.4\textwidth]{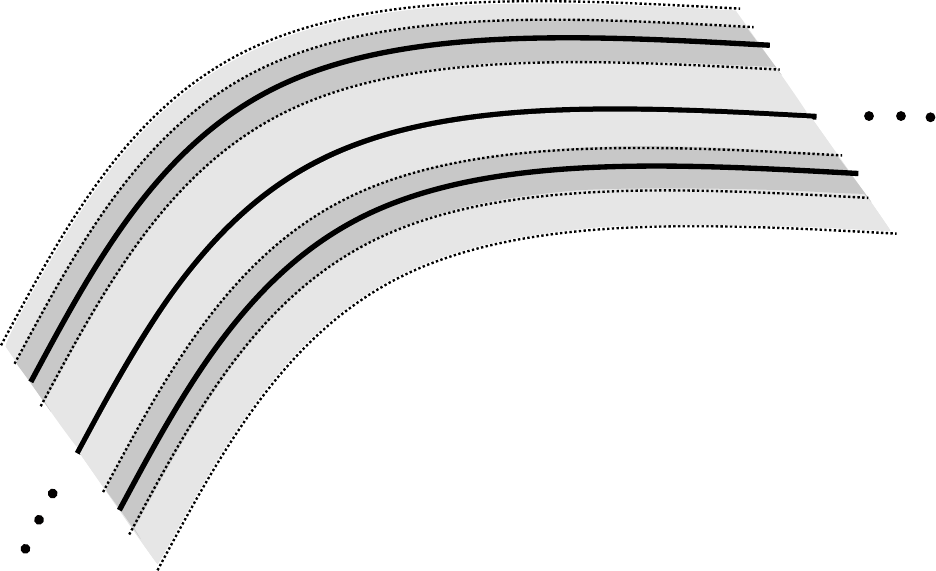}};
    \node at (-1.2,2) {$\im(\psi)$};
    \node at (-3.5,-0.8) {$\im(\psi_1)$};
    \node at (3.2,0.7) {$\im(\psi_2)$};
  \end{tikzpicture}
  \caption{The disjoint embeddings $\psi_1$ and $\psi_2$ which are isotopic to $\psi$ as small displacements.}
  \label{fig:displacements-of-psi}
\end{figure}
Choose a smooth collar  $M \times [0,1) \hookrightarrow W$ of $M \subset \partial W$ in $W$.
With respect to this collar, we identify $M \times [0,1)$ with a smooth submanifold of $W$.
Let $t$ be the canonical coordinate on $[0,1)$ and choose a Riemannian metric $h$ on $W$ which restricts to the metric 
\begin{equation} \label{collarmetric}
   dt^2 + \left(1+ \tfrac{2t}{n} \kappa\right) g_M \text{ on } M \times [0,1) \subset W,
\end{equation}
Note that $(W,h)$ only depends on $(M,g_M)$, but not on the fill-in $(\Omega, g_\Omega)$.
The second fundamental form of $(W,h)$ along $M \times \{0\}$ with respect to the exterior normal is equal to $- \tfrac{\kappa}{n} \cdot g_M$.

\medskip

From now on, let  $i \in \{1,2\}$. 
By \cite[Theorem 3.1]{Lawson_Michelsoh}, we can find a submanifold $T_i\subset W$ which intersects $M$ in $\im (\psi_i)$, which is isotopic to the handle attached along the surgery datum $\psi_i$ and such that the mean curvature of 
\begin{equation} \label{eq:attach_T_i}
   \widetilde M_i := \partial \left( \Omega\cup T_i\right)
\end{equation}
is bounded below by some constant $0 < \widetilde \kappa  < \kappa$, for $i = 1,2$, see \cref{fig:handle-attachment}. 
Here $\partial \left( \Omega\cup T_i\right)$ is a smooth submanifold of $W$ and the mean curvature is computed with respect to the metric $h$ and the normal pointing inside $W$.

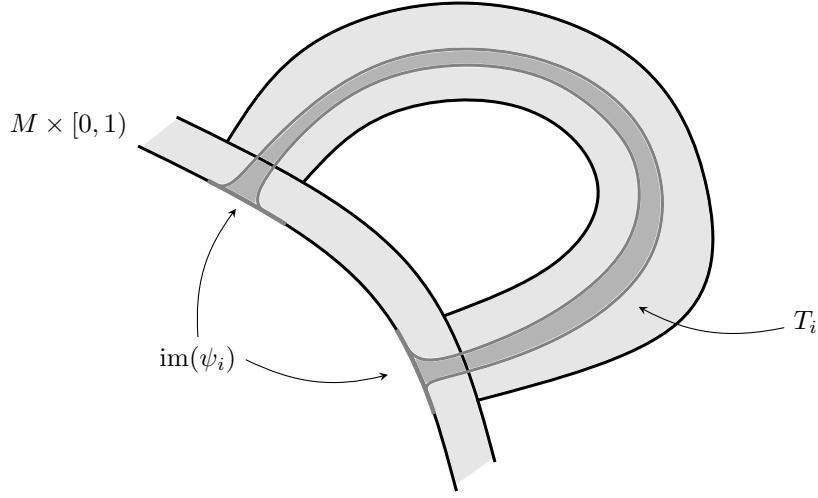
\begin{figure}[ht]
  \begin{tikzpicture}
    % \grid
    \node at (0,0) {\includegraphics[width=.6\textwidth]{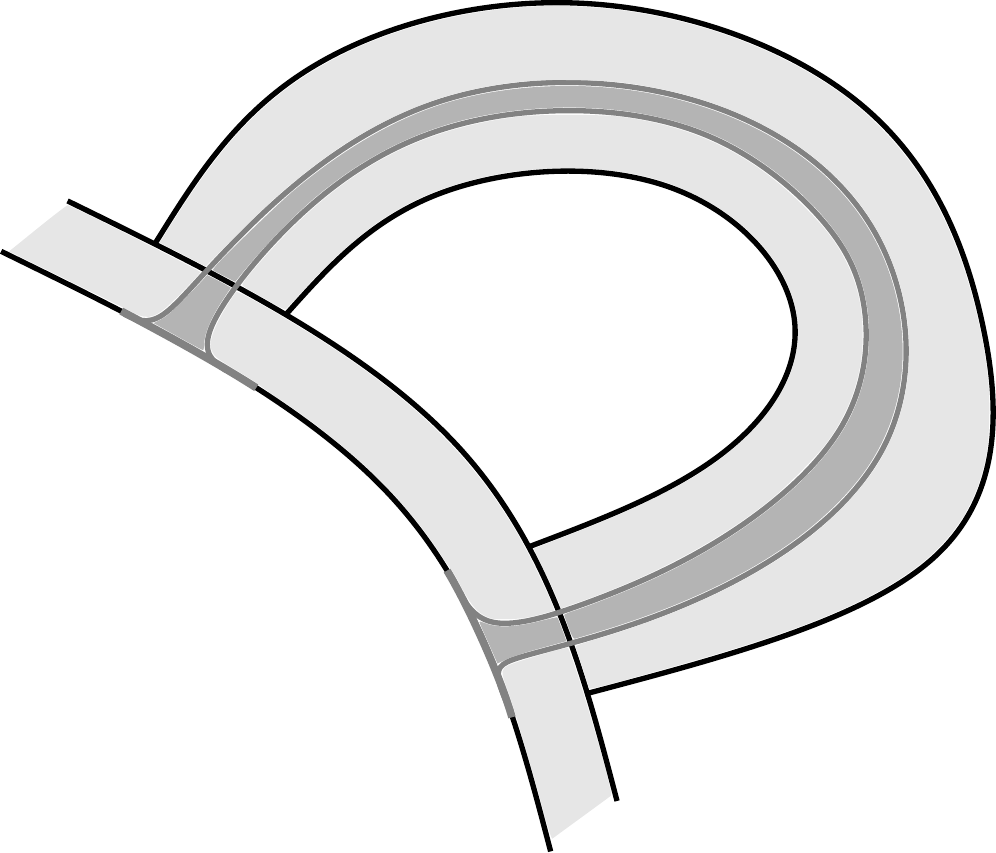}};
    \node at (-4.7,1.6) {$M\times[0,1)$};
    \node (0) at (-3,-1.5) {$\im (\psi_i)$};
    \draw[-stealth, bend left=20] (0.north) to (-2.5,0.5);
    \draw[-stealth, bend right=20] (0.east) to (-0.5,-1.7);
    \node (1) at (5,-1) {$T_i$};
    \draw[-stealth, bend left=20] (1) to (2.8,-0.8);
  \end{tikzpicture}
  \caption{The bordism $W$ consisting of one handle and the embedded handle}\label{fig:handle-attachment}
\end{figure}

\medskip 

Let $g_{\widetilde M_i}$ be the metric on $\widetilde M_i$ induced by $h$.
The Riemannian manifold $(\widetilde M_i, g_{\widetilde M_i})$ is independent of the fill-in $(\Omega, g_{\Omega})$. 
Furthermore, since $\widetilde M_i$ is diffeomorphic to $\widetilde M$,  the induction hypothesis applies to $\widetilde M_i$.

\medskip

Set $ \widetilde \Omega_i := \Omega \cup T_i \subset \Omega \cup W$.
We have $\partial \widetilde \Omega_i = \widetilde M_i$.
Let $Z_i$ be open neighborhoods of $\supp(\psi_i) \subset M$ such that $Z_1 \cap Z_2 = \emptyset$.
Further below, we will construct Riemannian metrics $g_{\widetilde \Omega_i}$ on $\widetilde\Omega_i$ such that
\begin{enumerate}
  \item \label{zero} the metric on $\widetilde M_i$ induced by $g_{\widetilde \Omega_i}$ coincides with $g_{\widetilde M_i}$,
  \item \label{one}  $H_{g_{\widetilde \Omega_i}} = H_{g_\Omega}$ on $\partial \Omega \setminus Z_i$,
  \item \label{two} $H_{g_{\widetilde \Omega_i}}\ge \widetilde \kappa$ everywhere,
  \item  \label{three} $\scal(g_{\widetilde \Omega_i}) > {\rm min} \{ \sigma  , \min \scal_h \}-1$.
\end{enumerate}
Using $Z_1\cap Z_2=\emptyset$, properties \ref{one}, \ref{two} and \ref{three}, and $\widetilde \kappa > 0$, we obtain
\begin{align} \label{eq:pigeonhole}
  \begin{split}
    \int_{\partial  \Omega}H_{g_\Omega} {\rm dvol}_{g_{\partial \Omega}} \le{}& \int_{\partial  \Omega\setminus Z_1}H_{g_\Omega} {\rm dvol}_{g_{\partial \Omega}} +\int_{\partial  \Omega\setminus Z_2}H_{g_\Omega} {\rm dvol}_{g_{\partial \Omega}} \\
      \le{}& \int_{\partial  \widetilde \Omega_1}H_{g_{\widetilde \Omega_1}} {\rm dvol}_{g_{\partial \widetilde \Omega_1}} +\int_{\partial  \widetilde \Omega_2}H_{g_{\widetilde \Omega_2}} {\rm dvol}_{g_{\partial \widetilde \Omega_2}} \\
      \le{}& 2 \max \left\{ \int_{\partial \widetilde \Omega_1} H_{g_{\widetilde \Omega_1}} {\rm dvol}_{g_{\partial \widetilde \Omega_1}} , \int_{\partial \widetilde \Omega_2} H_{g_{\widetilde \Omega_2}} {\rm dvol}_{g_{\partial \widetilde \Omega_2}} \right\} . 
  \end{split}
\end{align}

By the inductive assumption, for $i \in \{1,2\}$, \cref{thm:mainC} holds for  the manifold $(\widetilde M_i, \widetilde g_i)$, the lower scalar curvature bound  ${\rm min} \{ \sigma , \scal_h \}-1$ and the lower mean curvature bound $\widetilde \kappa > 0$.
Since  all of these data are independent of the fill-in $(\Omega, g_{\Omega})$, this concludes the induction step.

\medskip

It remains to construct the metric $g_{\widetilde \Omega_i}$ on $\widetilde\Omega_i$. 
Let $U_i, V_i$, $i = 1,2$, be open neighborhoods of $\im(\psi_i) \subset M$ such that $\overline U_i\subset V_i$, $\overline V_i\subset Z_i$.
By \cite[Theorem 3.7]{BaerHanke2}, we can smoothly deform the metric $g_{\Omega}$ on $\Omega$ to a new metric on $\Omega$, called $g_i$ for short, with the following properties (see \cref{fig:deformations}): 
\begin{itemize}
  \item The metrics on $\partial \Omega$ induced by $g_i$ coincide with $g_M$,
  \item $H_{g_i} = H_{g_\Omega}$ on $\partial \Omega \setminus Z_i$, 
  \item $H_{g_i}\ge \kappa$ everywhere,
  \item On $V_i$, the second fundamental form of $g_i$ is given by $\II_{g_i} = \tfrac{\kappa}{n} g_M$, 
  \item $\scal_{g_i} > \sigma -1$.
\end{itemize}
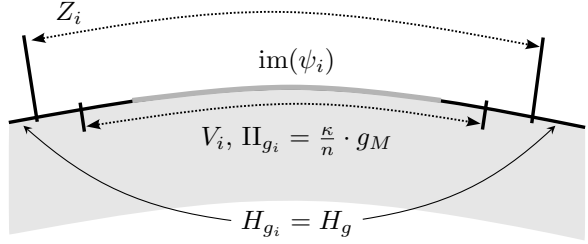
\begin{figure}[ht]
  \begin{tikzpicture}
    \node at (0,0) {\includegraphics[width=0.6\textwidth]{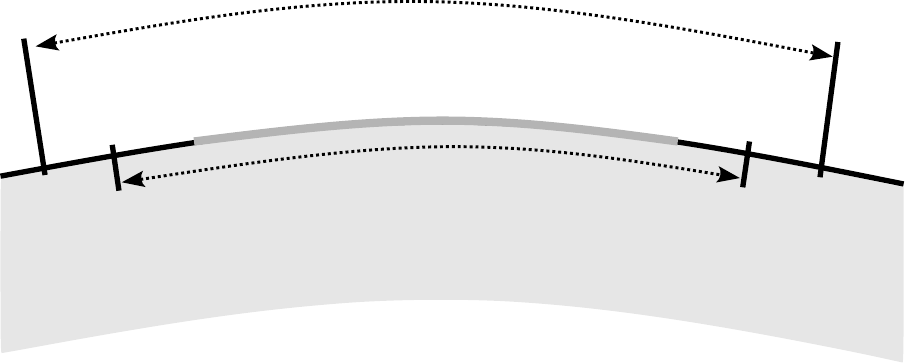}};
    \node at (0,0.85) {$\im(\psi_i)$};
    \node at (0,-.2) {$V_i$, $\II_{g_i} = \tfrac{\kappa}{n} \cdot g_M$};
    \node at (-3,1.5) {$ Z_i$};
    \node (0) at (0,-1.3) {$H_{g_i} = H_g$};
    \draw[-stealth, bend left=17] (0) to (-3.6,0);
    \draw[-stealth, bend right=17] (0) to (3.4,0);
  \end{tikzpicture}
  \caption{Deformation of $g_\Omega$ near $\partial \Omega$ (grey region)}
  \label{fig:deformations}
\end{figure}

Let $\delta > 0$ and let $\mathscr{C}_i := \partial \Omega \times (-\delta, 0] \subset \Omega$ be an open collar  of $\partial \Omega$ in $\Omega$ induced by the normal exponential map along $\partial \Omega$ for $g_i$.
Put $\mathscr{V}_i := V_i \times (-\delta, 0] \subset \mathscr{C}_i$.
Equip the topological manifold
\[ 
    \mathscr{V}_i \cup_{V_i \times 0} (V_i \times [0,1)) = V_i \times (-\delta, 1) \subset \Omega \cup W
\]
with the product smooth structure. 
We have 
\[
   \widetilde \Omega_i \subset \Omega \cup (U_i \times (-{\delta},1)) \cup {\rm int}(W). 
 \]
By construction, the given smooth structures on mutual intersections are compatible with each other. 
Hence, we obtain an induced smooth structure on $\widetilde \Omega_i$.

\medskip

Extend the restriction $h|_{V_i \times [0,1)}$ to a smooth metric $h_i'$ on $V_i \times (-\delta, 1]$ in an arbitrary way. 
The  $1$-jets of the metrics $g_i$ and $h_i'$ along $V_i \times \{0\} \subset  V_i \times (-\delta, 1)$ coincide by the construction of $g_i$.
Let $\chi_i \colon \partial \Omega \to [0,1]$ be a smooth cut-off function that is equal to $1$ on $U_i$ and equal to $0$ on $\partial \Omega \setminus V_i$. 
Consider the smooth path of symmetric $(0,2)$-tensor fields $\psi_i \colon [0,1] \to C^{\infty}(T^*\mathscr{C}_i  \otimes T^*\mathscr{C}_i)$,
\[
  \psi_i(t)(x,r) := (1-\chi_i(x)\cdot t) g_i + \chi_i(x) \cdot t \cdot h_i'. 
 \]
Since $\chi_i(x) = 0$ for $x \in \partial \Omega \setminus V_i$, this is well-defined.
Possibly after passing to a smaller $\delta$ (hence to a smaller $\mathscr{C}_i$), the path $\psi_i$ consists of Riemannian metrics on $\mathscr{C}_i$, see \cref{fig:deformation-before-local-flex}.
\begin{figure}[ht]
  \begin{tikzpicture}
    \node at (0,0) {\includegraphics[width=.8\textwidth]{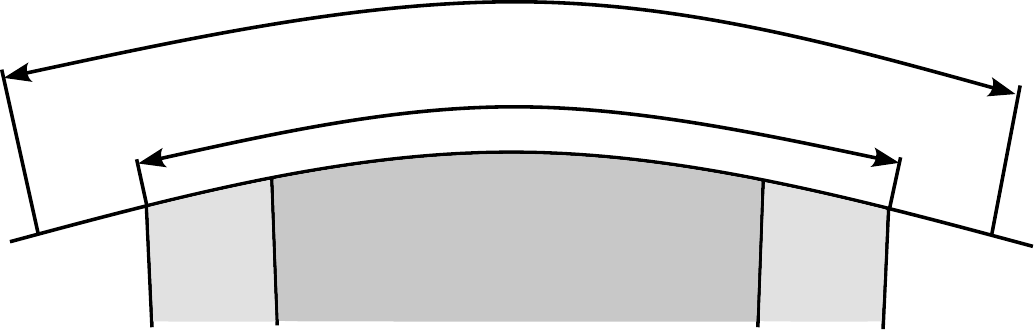}};
    \node at (-5,-1.3){$g_i$};
    \node at (5,-1.3){$g_i$};
    \node at (-4,1.9){$Z_i$};
    \node at (-2,1){$V_i$};
    \node at (0,-0.5){$U_i\times(-\delta,0]$};
    \node at (0,-1.1){$h_i'$};
    \node (0) at (0,-2){interpolation region};
    \draw[-stealth, bend right=20] (0.east) to (2.8,-1.6);
    \draw[-stealth, bend left=20] (0.west) to (-2.8,-1.6);
  \end{tikzpicture}
  \caption{The local deformation $\psi_i$ for $t=1$.}\label{fig:deformation-before-local-flex}
\end{figure}
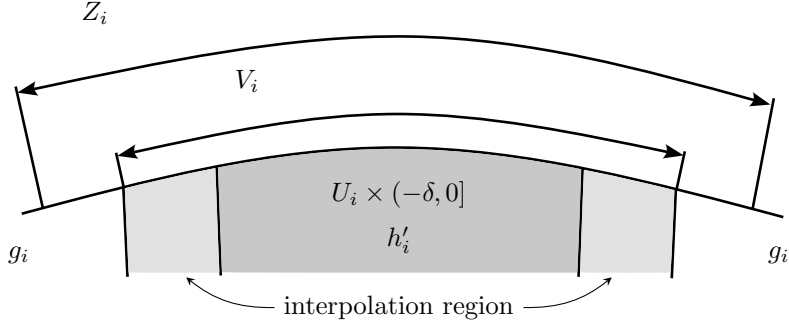

\medskip

By the choice of $\chi_i$, the $1$-jet of $\psi_i(t)$ along $\partial \Omega$ is independent of $t$.
In particular, the scalar curvature of $\psi_i(t)$ along $\partial \Omega$ is given by 
\[
  \scal_{\psi_i(t)} = (1-\chi_i(x) \,  t) \cdot \scal_{g_i} + \chi_i(x)\,  t \cdot \scal_{h'_i} > \min \{ \sigma , \min \scal_{h} \} -1 . 
\]
By continuity, possibly after passing to a smaller $\delta$, we have, on $\mathscr{C}_i$,
\[
  \scal_{\psi_i(t)} > \min \{ \sigma , \min \scal_{h} \} -1, \quad t \in [0,1] . 
\]
We apply the local flexibility lemma \cite[Theorem 1.2]{BaerHanke1}  to  the local deformation $\psi_i$ on the open neighborhood $\mathscr{C}_i$ of $\partial \Omega$ in $\Omega$ and to the second order open partial differential relation of being a Riemannian metric   with $\scal > \min \{ \sigma, \min \scal_{h}  \} -1 $ on symmetric $(0,2)$-tensor fields in $C^{\infty}(T^* \Omega \otimes T^*\Omega)$. 
This results in  a smooth path $\Psi_i \colon [0,1] \to C^{\infty}( T^* \Omega \otimes T^* \Omega)$ of Riemannian metrics on $\Omega$ starting at $g_i$, and there exists some $0 < \delta' < \delta$, such that the restriction of $\Psi_i(t)$ to $\partial \Omega \times (-\delta',0]$ coincides with $\psi_i(t)$, see \cref{fig:deformation-after-local-flex}. 

\begin{figure}[ht]
  \begin{tikzpicture}
    \node at (0,0) {\includegraphics[width=.8\textwidth]{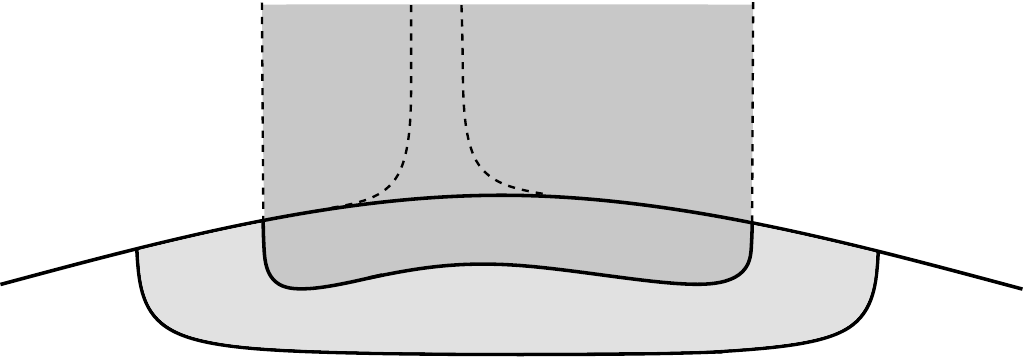}};
    % \grid
    \node at (1,0.4) {$g_{\widetilde \Omega_i} =  h_i' = h$};
    \node at (1,-.6) {$g_{\widetilde\Omega_i} = h_i'$};
    \node at (4,-2) {$g_{\widetilde \Omega_i} = g_i$};
    
    \node at (1,1.2) {$U_i\times[0,1)$};
    \node at (-3.6,-2) {$\Omega$};
    \node (0) at (-3.7,1.2) {$T_i\, \cap\, U_i\times[0,1)$};
    \draw[-stealth, bend right=20] (0) to (-1.1,0.4);
  \end{tikzpicture}
  \caption{The end of the global deformation $\Psi_i(1)$ obtained by applying the local flexibility lemma to the local deformation $\psi_i$.}\label{fig:deformation-after-local-flex}
\end{figure}
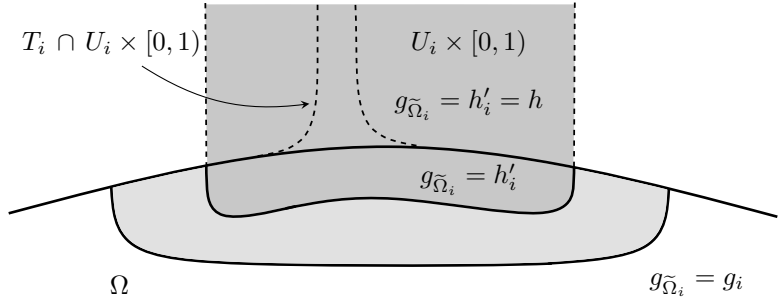

The metric $\Psi_i(1)$ on $\widetilde \Omega_i \cap \Omega = \Omega$, the metric $h_i'$ on $\widetilde \Omega_i \cap ( U_i \times (-\delta',1))$ and the metric $h$ on $\widetilde \Omega_i \cap W$ are compatible with each other.
Hence, the union of these metrics defines a smooth metric $g_{\widetilde \Omega_i}$ on $\widetilde \Omega_i$.
By construction, the following holds: 
\begin{itemize}
    \item the metric on $\partial \widetilde \Omega_i$ induced by $g_{\widetilde \Omega_i}$ coincides with $g_{\widetilde M_i}$,
    \item  $H_{g_{\widetilde \Omega_i}} = H_{g_i} = H_{g_\Omega}$ on $\Omega \setminus Z_i$,
    \item  $H_{g_{\widetilde \Omega_i}} \geq \kappa$ on $Z_i \setminus U_i$,
    \item $H_{g_{\widetilde \Omega_i}} \geq \widetilde \kappa$ everywhere else on $\widetilde M_i$, by the choice of $T_i$,
    \item $\scal_{g_{\widetilde \Omega_i}} > {\rm min} \{ \sigma, \min \scal_{h}  \} - 1$.
\end{itemize}
Therefore, $g_{\widetilde \Omega_i}$ possesses all the required properties.
This finishes the proof of \cref{thm:mainC}.
\end{proof}

In preparation for Section \ref{sec:spin_fill_in}, we prove the following result by a similar argument.

\begin{lemma} \label{lem:prep_spin}
Assume that \cref{thm:mainA} holds in the case when $M$ is diffeomorphic to the $n$-sphere.
Then \cref{thm:mainA} holds in general.
\end{lemma} 

\begin{proof}
Let $(M,g_M)$ and $\sigma, \kappa \in \bbR$ be as in \cref{thm:mainA}.
Put $n := \dim M$.
As before, we can assume $n \geq 5$.
Since $M$ carries only finitely many spin structures, we can assume without loss of generality that $M$ is equipped with a fixed spin structure and we are only considering spin fill-ins $\Omega$ of $M$ with isometries $ \partial \Omega \approx M$ that are spin structure preserving. 
Let $B$ be a spin manifold with $\partial B = M$ such that $B$ induces the inverse spin structure on $M$.

\medskip 

Removing a small open disc from the interior of $B$, we obtain a compact spin bordism $W$ from $M$ to $S^{n}$. 
Annihilating $\pi_0(W)$ and $\pi_1(W)$ by surgeries in the interior of $W$ (using $n \geq 5$), we can assume that  the inclusion $S^n \hookrightarrow \partial W \subset W$ is $1$-connected. 
Since $n \geq 5$, this implies that $(W, M)$ has a relative handle decomposition with handles of codimension at least $2$. 

\medskip 

Let $\ell \geq 0$ be the number of handles in this handle decomposition. 
The proof of \cref{thm:mainA} is by induction on $\ell$. 
The case $\ell = 0$ is covered by the assumption of \cref{lem:prep_spin}.
Assume by induction that \cref{thm:mainA} holds for $\ell \geq 0$ and that the number of handles of $W$ is equal to $\ell +1$.

\medskip 

Let $(\Omega, g_{\Omega})$ be a spin fill-in of $(M,g_M)$ satisfying $\scal_{g_{\Omega}} \geq \sigma$, $H_{g_{\Omega}} \geq \kappa$.
Let $\omega \colon (\partial \Omega, g_{\partial \Omega}) \cong (M, g_M)$ be a spin preserving Riemannian isometry.
For simplicity of exposition, we will assume from now on that $\partial \Omega = M$ and $\omega = \id$.
This implies that $\Omega \cup_M B$ is a spin manifold.

\medskip

We now proceed as in the proof of \cref{thm:mainC}, but  with one difference: For the construction of the surgery handles $T_i \subset W$ before  \eqref{eq:attach_T_i}, intersecting $M$ in $\mathrm{im}(\psi_i)$, we do no longer refer to \cite{Lawson_Michelsoh}. 
This implies that the lower bound $\widetilde \kappa$  for the mean curvatures of $\widetilde M_i$, $i = 1,2$, in \eqref{eq:attach_T_i} may become negative. 
Note that, as before,  $\widetilde \kappa$ does not depend on the fill-in $(\Omega, g_{\Omega})$ and that we can assume $\widetilde \kappa < \kappa$.

\medskip 

Set $\widetilde \Omega_i := \Omega \cup T_i$. 
Note that $\widetilde \Omega_i$ carries an induced spin structure.
By the construction below \cref{eq:pigeonhole}, which does not require $\widetilde \kappa > 0$, we obtain metrics $g_{\widetilde\Omega_i}$ on $\widetilde \Omega_i$ with properties \ref{zero}, \ref{one}, \ref{two} and \ref{three}.
Using that $Z_1 \cap Z_2 = \emptyset$ and that $H_{g_{ \Omega}}, H_{g_{\widetilde \Omega_i}} \geq \widetilde \kappa$, we hence obtain
\begin{align} \label{eq:pigeonhole2}
  \begin{split}
    &\int_{\partial  \Omega} (H_{g_\Omega}- \widetilde \kappa)  {\rm dvol}_{g_{\partial \Omega}}\\
      &\qquad\le{} \int_{\partial  \Omega\setminus Z_1} (H_{g_\Omega}-\widetilde \kappa) {\rm dvol}_{g_{\partial \Omega}} +\int_{\partial  \Omega\setminus Z_2}(H_{g_\Omega}-\widetilde \kappa) {\rm dvol}_{g_{\partial \Omega}} \\
      &\qquad\le{} \int_{\partial  \widetilde \Omega_1}(H_{g_{\widetilde \Omega_1}}-\widetilde \kappa) {\rm dvol}_{g_{\partial \widetilde \Omega_1}} +\int_{\partial  \widetilde \Omega_2}(H_{g_{\widetilde \Omega_2}}-\widetilde \kappa) {\rm dvol}_{g_{\partial \widetilde \Omega_2}} \\
      &\qquad\le{} 2 \max \left\{ \int_{\partial \widetilde \Omega_1} (H_{g_{\widetilde \Omega_1}}-\widetilde \kappa) {\rm dvol}_{g_{\partial \widetilde \Omega_1}} , \int_{\partial \widetilde \Omega_2} (H_{g_{\widetilde \Omega_2}}-\widetilde \kappa) {\rm dvol}_{g_{\partial \widetilde \Omega_2}} \right\}.
  \end{split}
\end{align}
Here we take advantage of the fact that all integrands are non-negative.

\medskip

Since $\widetilde \kappa$ and the metrics on $\partial \Omega$ and $\partial \widetilde \Omega_i$, $i=1,2$, are independent of $(\Omega, g_{\Omega})$, we obtain a constant $C$, which is independent of $(\Omega, g_{\Omega})$,  such that 
\[
    \int_{\partial  \Omega} H_{g_\Omega} {\rm dvol}_{g_{\partial \Omega}} \le{}
    2 \max \left\{ \int_{\partial \widetilde \Omega_1} H_{g_{\widetilde \Omega_1}} {\rm dvol}_{g_{\partial \widetilde \Omega_1}} , \int_{\partial \widetilde \Omega_2} H_{g_{\widetilde \Omega_2}} {\rm dvol}_{g_{\partial \widetilde \Omega_2}} \right\} + C. 
\]
This concludes the induction step  for \cref{thm:mainA}.
\end{proof}

\section{Spin fill-ins and creased initial data sets} \label{sec:spin_fill_in}

This section concludes the proof of \cref{thm:mainA}. 
According to  \cref{lem:prep_spin}, the only remaining case is when $M$ is diffeomorphic to the $n$-sphere.
Let $(\Omega_0,g_0)$ be a spin fill-in of $(M,g_M)$ such that
\begin{itemize}
\item $\scal(g_0)\ge \sigma_0$ for some $\sigma_0 < 0$, 
\item $H_{g_0}\ge \kappa$ along $M$, w.r.t.  the outward-pointing normal, for some $\kappa<0$.
\end{itemize}

In this section we adopt the sign-convention commonly used in general relativity, whereby the mean curvatures, denoted by $H(M\subseteq W)$, are computed with respect to the \enquote{infinity-pointing normal}, that is if $W\colon M\leadsto N$ is a cobordism from $M$ to $N$, then $H(M\subseteq W)$ is computed with respect to the interior unit normal, whereas $H(N\subseteq W)$ is computed with respect to the exterior unit normal.

\medskip

Now, let $M,N$ be diffeomorphic to the sphere and consider the positive Lipschitz continuous function $h$ on $M$ given by
\begin{align} \label{def:h}
    h_0=|\kappa|\sinh(1)+\frac13 H(M\subseteq \Omega_0)_+ >0.
\end{align}
After possibly Replacing $h_0$ by a smooth approximation we may without loss of generality assume that $h_0$ is smooth.
According to the proof of \cite[Lemma 3.1 and Theorem 4.1]{Shi-Wang-Wei}, there exists a Riemannian metric $g_1$ on $M\times[0,1]$ such that the cobordism $(\Omega_1=M\times[0,1],g_1)$ connects $(M\approx S^n,g_M)$ to $N:=S^n_{r_0}$, equipped with the round metric $g_{\circ,{r_0}}$ of radius ${r_0}\gg 1$ independent of $h$, with the following properties: 
\begin{enumerate}
    \item \label{um} $ H(M\subseteq \Omega_1)=h_0$ and $H(N\subseteq \Omega_1)>0$ for the mean curvatures of $M,N$, 
    \item \label{dois} $\scal(g_1)=\sigma_1$ for some constant $\sigma_1 < 0$ independent of the function $h_0$,
    \item \label{tres} the total mean curvatures are related by $\int_M H(M\subseteq \Omega_1)\le \int_{N} H(N\subseteq \Omega_1)$.
\end{enumerate}
We note that the final inequality together with $H(M\subseteq\Omega_1)
\ge\tfrac{1}{3}H(M\subseteq\Omega_0)$ implies
\begin{align}\label{total mean curvature monotonicity}
    \int_M H(M\subseteq\Omega_0)\le 3\int_N H(N\subseteq\Omega_2).
\end{align}
We define constants $c_i$, $i=0,1,2$ via
\begin{align*}
    c_0=\max\left\{4|\kappa|, \sqrt{\frac{n}{n+1}|\sigma_0|}, (\cosh 1)^{-1}\sqrt{\frac{n}{n+1}|\sigma_1|}\right\}, \qquad c_1=c_2=\cosh(1)c_0,
\end{align*}
and we set
\begin{align*}
    \sigma_2=-c_2^2\frac{n+1}{n},
\end{align*}
which again is independent of $h_0$. 
We observe that
\begin{align*}
    c_0\ge 4|\kappa|,\qquad |\sigma_2|>\max\{|\sigma_0|,|\sigma_1|\}, \qquad c_i^2\frac{n+1}{n}\ge -\sigma_i\qquad\text{for } i=0,1,2.
\end{align*}\
As in the proof of \cite[Lemma 3.2 and Theorem 4.1]{Shi-Wang-Wei}, there exists an asymptotic hyperboloidal extension $(\Omega_2,g_2)$ of $N$ with $ H(N\subseteq \Omega_1)=H(N\subseteq \Omega_2)$ and $\scal(g_2)=\sigma_2$.
Note that $\sigma_2=\sigma_1$ in \cite{Shi-Wang-Wei}, while in our setting $\sigma_2<\sigma_1$ to compensate for the mean curvature jump at $M$. 
Consider now the manifold $(\Omega,g):=(\Omega_0,g_0)\cup_M(\Omega_1,g_1)\cup_N(\Omega_2,g_2)$ with the above choices for $h_0$ and $\sigma_2$ which has two corners at $M$ and $N$.
The situation is illustrated in \cref{figure H<0}.

\begin{figure}[ht]
  \scalebox{0.65}{
    \begin{tikzpicture}
      \node at (0,0) {\includegraphics[width=.7\textwidth]{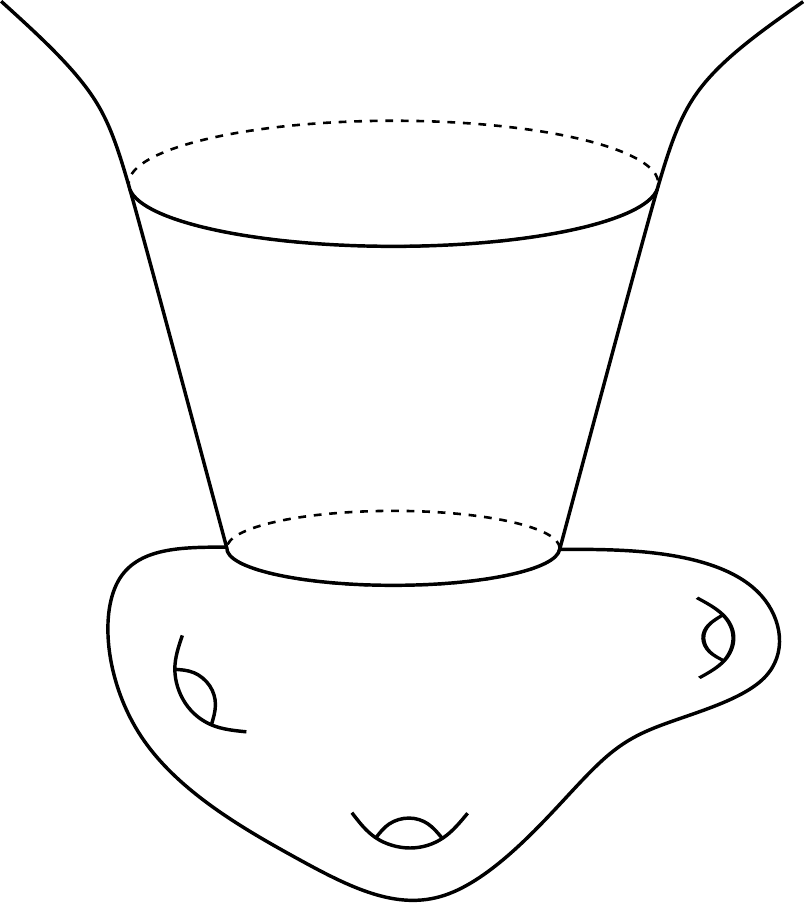}};
      % \grid;
      \node at (0,4.5) {$(\Omega_2,g_2)$};
      \node at (0,1) {$(\Omega_1,g_1)$};
      \node at (0,-3) {$(\Omega_0,g_0)$ };
      \node at (5.1,3.8) {$\begin{matrix}\text{asymptotically}\\ \text{hyperboloidal extension}\end{matrix}$};
      \node at (4.5,1) {cylinder $S^n\times[0,1]$};
      \node at (4,-4) {$\begin{matrix}
        \text{fill-in with}\\ H_{g_0}|_{M}<0 \text{ somewhere}
      \end{matrix}$};
      \node at (0,-1.1) {$(M\approx S^n,g_M)$};
      \node at (0,3) {$(N,g_N)= (S^n_{r_0}, g_{\circ,{r_0}})$};
      \node (0) at (-4.8,-.5) {$\left.\begin{matrix}
        H(M\subset\Omega_0)\ge\kappa<0\\
        H(M\subset\Omega_1)=h_0>0
      \end{matrix}\right\}$};
      \draw[-stealth, bend left=20] (0.east) to (-2.2,-0.9);
    \end{tikzpicture}
  }
  \caption{
  At  $M$  the mean curvature jumps from negative to positive. 
  Treating $\Omega$ as an initial data set $(\Omega,g,k)$, this jump will have a favorable sign if $k$ is chosen appropriately, with an appropriate jump at $M$, which requires the scalar curvature of the asymptotic hyperboloidal extension to be sufficiently negative.}
  \label{figure H<0}
\end{figure}
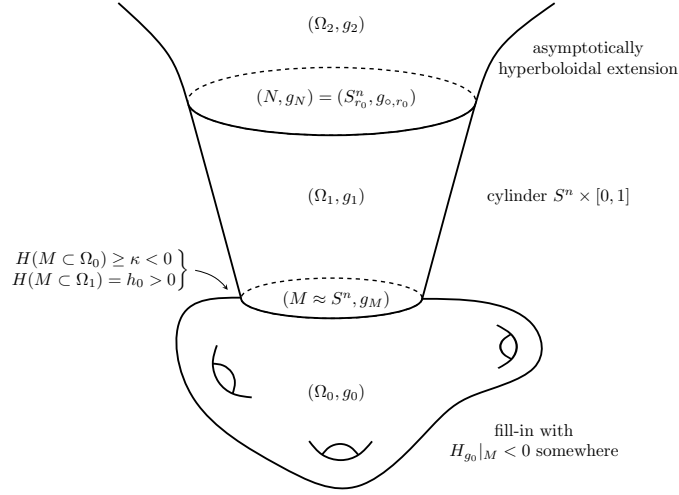

\medskip

Similarly to \cite[Theorem 4.1 and p.~232--233]{Shi-Wang-Wei} and \cite{shi-tam-manifolds-with-boundary}, the functional
\begin{align}\label{mean curvature - mass inequality}
    m(r)=\frac1{n\omega_{n}}\int_{\Sigma_r}\big(H(\Sigma_r\subseteq \mathbb H^{n+1}_{\sigma_2})-H(\Sigma_r\subseteq \Omega_2)\big)\cosh(r\sqrt{-K})%\xrightarrow{r\to\infty} m,
\end{align}
is monotonically increasing, and $m(r)\to m$ for $r\to\infty$.
Here $K=\tfrac{\sigma_2}{n(n+1)}$ ($=1$ in \cite{Shi-Wang-Wei}), and $H(\Sigma_r\subseteq \mathbb H^{n+1}_{\sigma_2})=n\sqrt{-K}\coth(r\sqrt{-K})$ is the mean curvature of a sphere in hyperbolic space $\mathbb H^{n+1}_{\sigma_2}$ with constant scalar curvature $\sigma_2$. 
Moreover, $\Sigma_r$ denotes the $(r-r_0)$-distance spheres to $N$ with respect to the hyperbolic background metric, e.g. we $\Sigma_{r_0}=N=S^n_{r_0}$.
Finally, the hyperboloidal mass $m$ will be discussed in \eqref{mass def}. % in more detail.

\medskip

In view of \eqref{total mean curvature monotonicity} and \eqref{mean curvature - mass inequality}, it suffices to prove $m\ge0$ in order to establish \cref{thm:mainA}.
However, since $\kappa<0$ and $h_0>0$, there may be an unfavorable jump of the mean curvature at $M$, which corresponds to \enquote{infinitely negative} scalar curvature, as depicted in \cref{figure H<0}.
This can be remedied by introducing a symmetric $(0,2)$-tensor $k$ on $\Omega$ which also jumps at $M$.
With a suitable choice of $k$, the \emph{dominant energy condition} for the resulting initial data set $(\Omega,g,k)$ is satisfied weakly across $M$.
Hence, the hyperboloidal positive mass theorem can be applied in this more general context.

\medskip

With the above choice of constants $c_0,c_1,c_2$, define the symmetric $(0,2)$-tensor $k_i=\tfrac {c_i}ng_i$ on $\Omega_i$ and let $k=k_0\cup k_1\cup k_2$.
Consequently, we obtain an initial data set $(\Omega,g,k)$ satisfying the dominant energy condition
\begin{align*}
    \mu-|J|=\frac12\left(\scal_i+c^2_i\frac{n+1}n \right)\ge \frac12\left(\scal_i-\sigma_i \right)\ge0\qquad \text{for $i=0,1,2$},
\end{align*}
where the energy density $\mu$ and momentum density $J$ are given by
\begin{align*}
    \mu=\frac12\left(\scal+\tr_g(k)^2-|k|^2\right),\qquad J=\operatorname{div}(k-\tr_g(k)g).
\end{align*}
It remains to check whether the dominant energy condition of $(\Omega,g,k)$ holds weakly across $M$.

\medskip

Given $\vartheta\in\mathbb R$, define the hyperbolic rotation matrix
\begin{align*}
    F_\vartheta=
    \begin{pmatrix}
        \cosh \vartheta & -\sinh \vartheta\\
        -\sinh \vartheta &\cosh\vartheta
    \end{pmatrix}.
\end{align*}
Consider the \emph{spacetime mean curvature vectors}
\begin{align*}
    \vec H(M\subseteq \Omega_i)=\begin{pmatrix}
        H(M\subseteq \Omega_i)\\ \tr_{M}(k_i)
    \end{pmatrix}=\begin{pmatrix}
        H(M\subseteq \Omega_i)\\ c_i
    \end{pmatrix}.
\end{align*}
Using the definition of $c_1$, we compute for $\vartheta=-1$
\begin{align*}
    &F_{\vartheta}      \vec H(M\subseteq \Omega_0)- \vec H(M\subseteq \Omega_1)\\
    &\qquad=
    \begin{pmatrix}
       \cosh(-1) H(M\subseteq \Omega_0)-\sinh(-1) c_0\\
       \cosh(-1)c_0-\sinh(-1)H(M\subseteq \Omega_0)
    \end{pmatrix}
    -
    \begin{pmatrix}
        H(M\subseteq \Omega_1)\\ c_1
    \end{pmatrix}\\
    &\qquad=
    \begin{pmatrix}
        \cosh(-1) H(M\subseteq \Omega_0)-\sinh(-1) c_0-H(M\subseteq \Omega_1)\\
        -\sinh(-1)H(M\subseteq \Omega_0)
    \end{pmatrix}.
\end{align*}
Let $x\in M$ where $H(M\subseteq \Omega_0)>0$. Using $c_0\ge4|\kappa|$ and $2\sinh(1)>\cosh(1)$, we estimate in this case at $x$
\begin{align*}
     &\cosh(-1) H(M\subseteq \Omega_0)-\sinh(-1) c_0-H(M\subseteq\Omega_1)\\
     &\qquad\ge \cosh(1) H(M\subseteq \Omega_0)+(4-1)\sinh(1)|\kappa|\\
     &\qquad\ge\sinh(1)|H(M\subseteq \Omega_0)|.
\end{align*}
Next, let $x\in M$ where $H(M\subseteq\Omega_0)\ge 0$.
Using that $\cosh(1)-\sinh(1) = e^{-1} > \tfrac13$, we estimate in this case at $x$
\begin{align*}
     &\cosh(-1) H(M\subseteq \Omega_0)-\sinh(-1) c_0-H(M\subseteq\Omega_1)\\
     &\qquad\ge \cosh(1) H(M\subseteq \Omega_0)+(4-1)\sinh(1)|\kappa|-\frac13H(M\subseteq \Omega_0)\\
     &\qquad\ge \sinh(1)|H(M\subseteq \Omega_0)|.
\end{align*}
This demonstrates that inequality (1.7) in \cite{KazarasKhuriLin2025} holds at $M$, i.e. the dominant energy condition is weakly satisfied across $M$.
We point out, using the notation of \cite{KazarasKhuriLin2025}, that $\nu_+,\tau_+$ correspond to $\binom{1}{0},\binom{0}{1}$, and that the terms $\beta_\pm=k(\nu,\cdot)|_{TM}$, $\beta^\Delta=\beta_+-\beta_--d\vartheta$ vanish, where $\nu$ denotes the infinity-pointing normal to $M$.

\medskip

It remains to justify that the hyperboloidal positive mass theorem (PMT) can indeed now be applied, which is needed to resolve Gromov’s total mean curvature conjecture in the spin case.
To this end, we first recall the standard hyperboloidal PMT for initial data sets.
In what follows, let $b$ denote the standard $(n+1)$-dimensional hyperbolic metric with scalar curvature $-n(n+1)$.
We write $\mathcal S$ for the spinor bundle of the spin manifold $\Omega$, and set $\overline{\mathcal S}=\mathcal S\oplus \mathcal S$ for the corresponding spacetime spinor bundle, cf. \cite[Definition 2.6]{HirschZhang2025}, which admits a Clifford action of Minkowski space $\mathbb R^{n+1,1}=\operatorname{span}\{e_0,e_1,\dots,e_{n+1}\}$.

\begin{theorem}[{\cite{CH2003} and \cite[Proposition 2.9]{HirschJangZhang2025}}]
    Let $(\Omega^{n+1},g,k)$ be an asymptotically hyperboloidal spin initial data set, that is,
    $g\to g_{\mathbb H^n}$ and $k\to g$ at infinity\footnote{We refer to \cite[Definition 2.2]{HirschJangZhang2025} for the precise asymptotics.},
    and suppose that the dominant energy condition $\mu\ge|J|$ holds.
    Then for each Killing spinor $\psi^\infty$ on $\mathbb H^{n+1}$ there exists a spinor
    $\psi\in\overline{\mathcal S}$, asymptotic to $\psi^\infty$ with $\psi-\psi^\infty\in H^1(\Omega)$, solving
    $\slashed D\psi=-\tfrac12\tr_g(k)e_0\psi$, such that
    \begin{align*}
        \mathcal{H}_{hyp}(V)
        =4\int_\Omega \Big(
            \big|\nabla_i\psi+\frac12 k_{ij}e_j e_0\psi\big|^2
            +\frac12\big\langle \psi,(\mu+J e_0)\psi\big\rangle
        \Big)\ge0,
    \end{align*}
    where $V=|\psi^\infty|^2_b$ and the hyperboloidal mass functional $\mathcal{H}_{hyp}(V)$ is defined in \cite[p. 307]{HirschJangZhang2025}.
\end{theorem}

By appropriate rescalings ($g_{ij}\to\lambda^2 g_{ij}$, $k_{ij}\to \lambda k_{ij}$) an analogous result holds for initial data sets $(\Omega, g,k)$ asymptotic to $\mathbb H^{n+1}_{\sigma_2}$ instead of the standard hyperbolic space $\mathbb H^{n+1}=\mathbb H^{n+1}_{n(n+1)}$, cf. \cite[p. 235]{Shi-Wang-Wei}.
See \cite{CJL, CH2003} and \cite[Appendix A]{HirschJangZhang2025} for a detailed discussion and for further references on the hyperboloidal PMT.
We emphasize that this result should not be confused with the \emph{asymptotically AdS}
PMT, in which the metric $g$ also asymptotes to the hyperbolic metric
but $k$ asymptotes to zero \cite{ChruscielMaertenTod, hirsch2025causal}.

\medskip

In our situation, the initial data set has significantly more restrictive asymptotics:
we have $k=c_2g$ at infinity, and the Bartnik extension metric has a very
specific form.
As pointed out in \cite[p.~232 and Section 3]{CH2003} as well as \cite[Appendix A]{BalehowskyWoolgar2012}, under these so-called
Wang asymptotics the mass \begin{align}\label{mass def}
    m=\mathcal H_{hyp}(V_0),\qquad V_0=V_0(\psi_0^\infty)
\end{align} reduces to the
mass defined in \cite{Wang2001AHmass}, for a certain choice of Killing spinor $\psi^\infty_0$.
In view of \eqref{mean curvature - mass inequality}, this in particular implies $m(r_0)\ge0$, which yields the desired total mean curvature bound.

\medskip

Finally, we explain why the corners at $M$ and $N$ do not pose difficulties.
The type of corner arising at $N$ is well understood; see, for instance,
\cite{BoniniQing2008}.
We therefore focus on the corner along $M$.
Strictly speaking, the results of \cite{Lin, KazarasKhuriLin2025} are formulated for
asymptotically flat spin manifolds.
However, their arguments are entirely local, and it is straightforward to see that
they extend to the asymptotically hyperboloidal spin setting as well.
More precisely, given a spinor $\psi$ on $N$, let $\psi_\pm$ denote its restrictions to
$\Omega_+:=\Omega_1\cup_N\Omega_2$ and $\Omega_-:=\Omega_0$, respectively.
Recall Witten’s standard divergence identity,
\begin{align*}
    &\Big|\nabla_i\psi_\pm+\frac12 k_{ij}e_j e_0\psi_\pm\Big|^2
    +\frac12\mu|\psi_\pm|^2
    +\frac12\langle J e_0\psi_\pm,\psi_\pm\rangle
    -\Big|\slashed D\psi_\pm-\frac12\tr_g(k)e_0\psi_\pm\Big|^2 \\
    &\qquad=\nabla_i\big(
        \langle \psi_\pm,\nabla_i\psi_\pm\rangle
        -\langle e_i\psi_\pm,\slashed D\psi_\pm\rangle
    \big)
    +\frac12\nabla_i\langle \psi_\pm,(k_{ij}e_j-\tr_g(k)e_i)e_0\psi_\pm\rangle.
\end{align*}
Integrating this identity by parts on both $\Omega_-$ and $\Omega_+$, and then imposing
the transmission boundary condition \cite[Equation (1.9)]{KazarasKhuriLin2025}
\begin{align}\label{transmission}
    \psi^-=\cosh\frac{\vartheta}{2}\,\psi^+
    +\sinh\frac{\vartheta}{2}\,\nu^+ e_0^+\psi^+,
\end{align}
for some angle function $\vartheta$ on $M$ and normal vectors $\nu^\pm$ to $M$, one obtains Kazaras--Khuri--Lin's version of Witten’s
mass formula \cite[Proposition~4.4]{KazarasKhuriLin2025}.
The existence theory for spinors solving $\slashed D\psi=-\tfrac12\tr_g(k)e_0\psi$ together  boundary conditions $\psi\to\psi^\infty$ at infinity and \eqref{transmission} follows from a straightforward
combination of the arguments in \cite{KazarasKhuriLin2025} and
\cite{Wang2001AHmass, CJL, CH2003}.
See also \cite{BaerBallmann2012} for a more general discussion of such transmission boundary conditions.

\medskip

We note that the hyperboloidal positive mass theorem can in fact be reduced to the
asymptotically flat case, as shown in \cite{ChruscielDelay2019}.
Thus, the results of \cite{KazarasKhuriLin2025} can also be applied directly.
This finishes the proof of \cref{thm:mainA}.

\medskip

The hyperboloidal positive mass theorem also holds for non-spin manifolds, see \cite{ACG2008, Sakovich2021, Lundberg2023, CD2019}.
However, the non-smooth case treated in \cite{KazarasKhuriLin2025} does not yet have an
analogue in the non-spin setting.
Resolving this issue would extend the statement of \cref{thm:mainA} to the non-spin case.

\medskip 

We close this section with some remarks on the asymptotics of $\Lambda(M, g_M, \sigma, \kappa)$ for fixed $(M, g_M)$. 
In \cite{Baer2026UpperBound} it is shown that $\Lambda(M, g_M, \sigma, \kappa)$ can be chosen to depend linearly on $\sqrt{ |\sigma|}$ and $|\kappa|$.

\medskip 

Our constant $ \Lambda(M,g_M, \sigma_0, \kappa)$, in case $M \approx S^n$, can be explicitly estimated via 
\begin{align}\label{eq:mean-curvature-in-hyperbolic-space}
    H(N=S^n_{r_0}\subseteq \mathbb H^{n+1}_{\sigma_2})= n\sqrt{-\frac{\sigma_2}{n(n+1)}}\coth\left(r_0\sqrt{-\frac{\sigma_2}{n(n+1)}}\right),
\end{align}
and
\begin{align}\label{eq:sigma_2-for-hyperbolic-spheres}
    \sigma_2=-\frac{n+1}{n}\cosh(1)^2\max\left\{16\kappa^2, \frac{n}{n+1}|\sigma_0|, (\cosh 1)^{-1}\frac{n}{n+1}|\sigma_1|\right\}.
\end{align}
For $-\sigma_0\gg1$ or $-\kappa \gg1$, we therefore have $H(N\subseteq \mathbb H^{n+1}_{\sigma_2})\approx\max\{\sqrt{-\sigma_0},-\kappa\}$, which gives the same asymptotics as Bär's proof.

\medskip 

With the methods developed in the succeeding sections, we can show that the same asymptotics hold for all closed spin manifolds.
To justify this, let $(\Omega_0, g_0)$ be a spin fill-in of $(M,g_M)$ such that $\scal_{g_0} \geq \sigma_0$ and $H_{g_{0}} \geq \kappa$ for some $\sigma_0, \kappa < 0$. 
Let $h_0 \colon M \to \bbR_{>0}$ be the (smoothened) function in \eqref{def:h}. 
According to the proof of \cite[Lemma 3.1 and Theorem 4.1]{Shi-Wang-Wei}, there exists a Riemannian metric $g_1$ on $M\times[0,1]$ such that the cobordism $(\Omega_1=M\times[0,1],g_1)$ connects $(M ,g_M)$ to $(N:= M, 2 g_M)$ satisfying properties \ref{um}, \ref{dois} and \ref{tres} above.
This results in a spin fill-in $(\Omega = \Omega_0 \cup_M \Omega_1, g_{\Omega})$ with \emph{positive} mean curvature along $(N,2g_M)$, with a  crease at $M$ and with scalar curvature equal to $\sigma_1 = \sigma_1(M,g_M)$. 

\medskip

Since $N$ is spin, the sphere $S^n$ can be obtained from $N$ by a sequence of $\ell$ oriented surgeries in codimension $\geq 3$ for some $\ell\ge0$, see the beginning of the proof of \cref{thm:mainB} on page \pageref{proof:mainB}.
An inductive application of \cref{thm:surgery-step} (which still holds for creased fill-ins) shows that there exists a smooth metric $\widehat g_{S^n}$ on $S^n$ and a constant $C=C(M,g_M) \le 0$, both of which only depend on $(M, g_M)$, such that the fill-in $(\Omega, g_{\Omega})$ of $(N,g_N)$ can be extended by smooth handle attachments to a creased fill-in $(\widetilde \Omega, g_{\widetilde \Omega})$ of $(S^n, \widehat g_{S^n})$ satisfying
\begin{itemize}
    \item $g_{\widetilde\Omega} = g_\Omega = g_0$ on $\Omega_0$,
    \item $\scal(\widetilde\Omega,g_{\widetilde{\Omega}})\ge \min\{\sigma_0, C-\ell\}$,
    \item $\int_N H \big(N \subseteq \Omega_1) \le 2^{\ell}\int_{S^n}H(S^n\subseteq \widetilde{\Omega})$.
\end{itemize}
By \eqref{total mean curvature monotonicity} and \eqref{mean curvature - mass inequality}, we thus obtain
\begin{align*}
    \int_N H \big(M \subseteq \Omega_0) \leq 2^{\ell}\cdot 3\cdot \int_{S^n_{r_0}} H(S^n_{r_0}\subseteq\bbH^{n+1}_{\sigma_2}) 
\end{align*}
for some $r_0=r_0(M,g_M)>0$ and for $\sigma_2$ given by
\[
    \sigma_2 = -\frac{n+1}{n}\cosh(1)^2\max\left\{16\kappa^2, -\sigma_0\frac{n}{n+1}, -C+\ell\right\},
\]
similar to \eqref{eq:sigma_2-for-hyperbolic-spheres}.
Hence, there exists a $C_0(M,g_M)>0$, independent of $\kappa$ and $\sigma_0$, \eqref{eq:mean-curvature-in-hyperbolic-space} gives
\begin{align}\label{constant dependency general}
  \Lambda(M,g_M,\sigma_0,\kappa)  \le C_0(M,g_M) \cdot \big(\max\{\sqrt{-\sigma_0},-\kappa\}+1\big).
\end{align}

\section{Fill-ins of families of metrics}\label{sec:useful-observation}

Let $(M,g_M)$ be a closed oriented Riemannian manifold.
In this section we derive a uniform upper bound for the total mean curvatures of fill-ins of $M$ for a certain family of metrics on $M$.
The proof is based on arguments in \cite{Shi-Wang-Wei}.

\begin{definition} 
Let  $s \in \mathbb{R}$ such that $\scal_g > s$.
We define  $\mathscr{R}_{g_M, s}(M)$ as the set of Riemannian metrics $g$ on $M$ with the following property: There exists a smooth path $\zeta_{g} \colon [0,1] \to C^{\infty}({\rm Sym} (T^*M \otimes T^* M))$ of Riemannian metrics starting at $g$ and ending at $g_M$ such that, for all $t\in [0,1]$, we have

\begin{enumerate}
   \item \label{eins} $\zeta'_{g}(t) > 0$ pointwise on $M$, in the sense of symmetric $(0,2)$-tensor fields,
   \item \label{zwei} $\scal_{\zeta_g(t)} > s$ pointwise on $M$.
\end{enumerate}
\end{definition}

\begin{proposition} \label{prop:criterion-hssw} For any $\sigma, s \in \mathbb{R}$  and  $g \in \mathscr{R}_{g_M, s}(M)$ the following holds: Let $(\Omega_0, g_{\Omega_0})$ be a fill-in of $(M , g)$ with $\scal_{g_\Omega} \geq  \sigma$ and $H_{g_\Omega} \geq 0$. 
Then there exists a smooth Riemannian metric $g_{ \Omega_1}$ on 
\[
  \Omega_1 := \Omega_0 \bigcup_{\partial \Omega \approx M} \big(M \times [0,1]\big)
 \]
with the following properties: 
\begin{itemize}
 \item $g_{\Omega_1} = g_{\Omega_0}$ outside an arbitrarily small neighborhood of $M \times [0,1] \subset \Omega_1$, 
 \item $g_{\Omega_1}$ induces the metric $g_M$ on $\partial \Omega_1 = M \times \{1\}$,
 \item $H_{g_{ \Omega_1}} > 0$ along $M \times \{1\}$, 
 \item $\scal_{ \Omega_1} \geq \min\{\sigma, s\} - 1$, 
 \item $
   \int_{\partial \Omega_0} H_{g_{\Omega_0}} {\rm dvol}_{g} < \int_{\partial  \Omega_1}H_{g_{\Omega_1}} {\rm dvol}_{g_M}$. 
\end{itemize}
\end{proposition}

\begin{proof}
Let $g \in \mathscr{R}_{g_M, s}(M)$.
Let $\zeta_{g} \colon [0,1] \to C^{\infty}({\rm Sym} (T^*  M \otimes T^* M))$ be a smooth path from $g$ to $g_M$ that satisfies properties \ref{eins} and \ref{zwei}.
Consider the smooth metric $\overline g$ on $M \times [0,1]$ defined by
\[
    \overline g = dt^2 + \zeta_g(t) . 
\]
Let $\overline A_t$ be the second fundamental form and $\overline H_t = \tr_{\zeta_g(t)} \overline A_t$ be the mean curvature of $\Sigma_t := M \times \{t\} \subset (M \times [0,1], \overline g)$ with respect to the unit normal field $\frac{\partial}{\partial t}$.
By  \ref{eins}, we have $ \overline A_t > 0$, hence $\overline H_t > 0$ and $\overline H_t^2 - \| \overline A_t \|^2 > 0$ for $t \in [0,1]$

\medskip

Let $(\Omega, g_\Omega)$ be a  fill-in of $(M, g)$ with $\scal_{g_\Omega} \geq  \sigma$ and $H_{g_\Omega} \geq 0$. 
Without loss of generality, we  can assume that $\partial \Omega = M$.
Applying \cite[Proposition 3.8]{BaerHanke2}, we can increase the mean curvature $H_{g_\Omega}$ slightly, while replacing $\sigma$ by $\sigma-\tfrac{1}{2}$. 
Hence, we can assume that $H_{g_\Omega} > 0$.

\medskip

On $M \times [0,1]$, consider the Bartnik-type parabolic PDE
\begin{align*}
 \overline H_t \frac{\partial u}{\partial t} & = u^2 \Delta_{\zeta_g(t)} u + \frac{1}{2} (s - \scal_{\zeta_g(t)}) u^3 + \frac{1}{2} ( \scal_{ \zeta_g(t)} - \scal_{\overline g}) u \\
 u( \cdot, 0) & = \frac{\overline H_0}{H_{g_\Omega}} > 0 . 
\end{align*}
Since, by \ref{zwei}, we have $s - \scal_{\zeta_g(t)} < 0$ for $t \in [0,1]$, this equation has a positive solution $u \colon M \times [0,1] \to \bbR_{>0}$.
Consider the Riemannian metric $\widetilde g = u^2 dt^2 + \zeta_{g}(t)$ on $M \times [0,1]$.

\medskip

Let $H_t = u^{-1} \overline H_t > 0$ be the mean curvature on $\Sigma_t$ induced by $\widetilde g$ with respect to the unit normal field $u^{-1} \frac{\partial}{\partial t}$.
Note that $H_0 = H_{g_{\Omega}}$.
By the computation in the proof of \cite[Theorem 4.1]{Shi-Wang-Wei} and using that $\overline H_t^2 - \|\overline A_t\|^2 >  0$, we have
\[
  \frac{\partial}{\partial t} \int_{\Sigma_t} H_t \, \rm{d vol}_{\zeta_{g}(t)} > 0 .
\]
In particular,
\begin{equation} \label{crucial}
     \int_{\Sigma_1} H_1 {\rm d vol}_{\zeta_g(1)} > \int_{\Sigma_0} H_0 {\rm d vol}_{\zeta_g(0)} = \int_{\partial \Omega_0} H_{g_{\Omega_0}}  {\rm d vol}_{g} .
\end{equation}
Furthermore, we have $\scal_{\widetilde g} = s$.

\medskip

Put 
\[
    \Omega_1 := \Omega_0 \bigcup_{\partial \Omega  = M \times \{0\}} M  \times [0,1]   . 
\]
By \cite[Theorem 4.11]{BaerHanke2} and by the choice of $u(\cdot, 0)$, the mean convex singularity at the gluing region $\partial \Omega_0 = M \times \{0\}$ can be smoothened, that is, there exists a smooth Riemannian metric $g_{\Omega_1}$ on $ \Omega_1$ which is equal to $\widetilde g$ near $M \times \{1\}$ and whose scalar curvature is bounded below by $\min \{ s , \sigma\} -1$.

\medskip

The manifold $( \Omega_1, g_{\Omega_1})$ is a  fill-in of $(M, g_M)$.
Furthermore, by \eqref{crucial}, we get
\[
   \int_{\partial \Omega_0} H_{g_{\Omega_0}} {\rm dvol}_{g} < \int_{\Sigma_1} H_1 {\rm d vol}_{\zeta_g(1)} = \int_{\partial  \Omega_1} H_{g_{ \Omega_1}} {\rm dvol}_{g_M}. \qedhere
\]
\end{proof}

\section{Surgery bending in Euclidean space}\label{sec:euclidean-handles}
In this section we construct smooth hypersurfaces $\overline\Sigma_\varrho$ of size $\varrho$ in Euclidean space which we will later use to define surgery bendings with quantitative control. 
Let $(x_1, \ldots, x_k)$ be the standard coordinates on $\bbR^k$ and let $t$ be the standard coordinate on $[0,1]$.
Equip $B_1^{k}(0)\times[0,1]\subset\mathbb R^{\,k+1}$ with the flat metric $g_{\mathbb R^{k+1}}=\dx_1^2+\dots+\dx_{k}^2 + \dt^2$.
Let $r := \sqrt{x_1^2 + \ldots + x_k^2} \colon \bbR^k \to [0, \infty)$ denote the radial coordinate.

\begin{construction}(Bending hypersurface) \label{Construction: Euclidean handle}
    Let $\varrho\in(0,\frac1{10}]$ be fixed and consider the function 
    \[\xi_\varrho(r)\coloneqq -\frac{4}{3}\varrho \left(\sqrt{r}-\varrho^{2}\right)^{3/2}-4\varrho^3\sqrt{\sqrt{r}-\varrho^2}\]
    from \cref{Lemma: construction of function}.
    Put 
    \[
      a_\varrho\coloneqq -\xi_\varrho(2\sqrt{\varrho})  > 0. 
    \]
    
    \medskip

    Fix a cutoff-function $\eta\colon\bbR\to[0,1]$ such that $\eta\equiv1$ on $(-\infty,0]$, $\eta\equiv0$ on $[1,\infty)$, and $-2\le \eta'\le0$, $|\eta''|\le10$.
    Define  $\eta_\varrho(r)=\eta\left(\tfrac{r-\sqrt \varrho}{\sqrt \varrho}\right)$ and set
    \[f_\varrho(r)\coloneqq\eta_\varrho(r)\bigl(\xi_\varrho(r)+a_\varrho\bigr),
    \]
    see \cref{fig:f_rho}.
    Let $G_\varrho \subset \bbR^k \times \bbR$ be the graph of the function 
    \[B_1^k(0)\setminus B_{\varrho^4}^k(0)\to\bbR,\qquad(x_1,\dots,x_k)\mapsto f_\varrho (r) .\]
    \begin{figure}[ht]
    \begin{tikzpicture}
        % \grid
        \node at (0,0) {\includegraphics[width=0.7\textwidth, trim = 5 100 1000 120, clip]{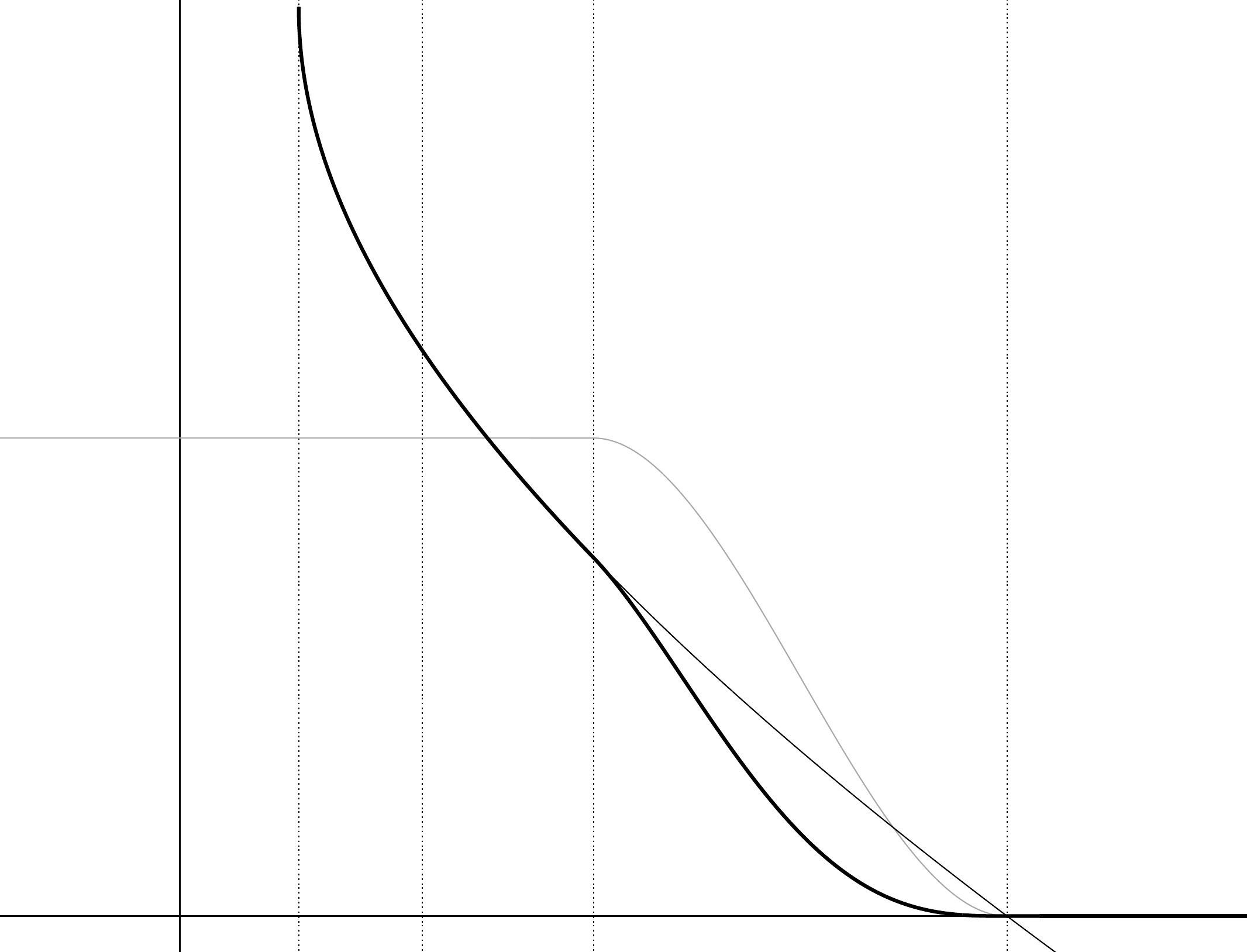}};
        \node at (-2.2,-3.6) {$\varrho^4$};
        \node at (-1.4,-3.6) {$\varrho^2$};
        \node at (-0.2,-3.6) {$\sqrt{\varrho}$};
        \node at (2.6,-3.6) {$2\sqrt{\varrho}$};
        \node at (-1,1) {${f}_\varrho$};
        \node at (2.3,-2.4) {$\xi_\varrho$};
        \node at (1,0) {$\eta_\varrho$};
    \end{tikzpicture}
    \caption{The function ${f}_\varrho$ constructed as the interpolation of $\xi_\rho$ and the constant function using the cutoff-function $\eta_\varrho$.}\label{fig:f_rho}
    \end{figure}
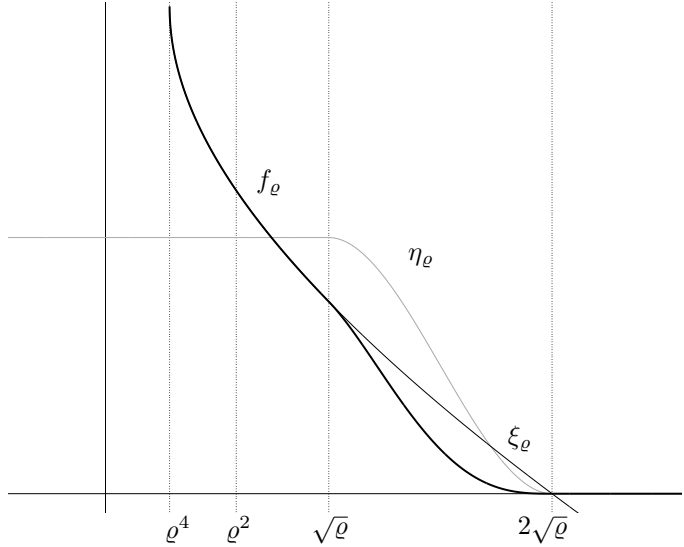

    Finally, put
\[
   \overline \Sigma_{\varrho} := G_\varrho \; \bigcup \; \{ r = \varrho^4 \} \times [a_\varrho, 1] \subset B^k_1(\varrho) \times [0,1].
\]
  Since the inverse of the function $f_{\varrho}|_{[\varrho^4, \varrho^2]}$ is the smooth function $[f_{\varrho}(\varrho^2), a_{\varrho}] \ni t \mapsto  \xi_{\rho}^{-1}(t- a_{\varrho})$ with vanishing derivative at $t = a_{\varrho}$ (see \cref{Lemma: construction of function}), $\overline \Sigma_{\varrho}$ is a $C^1$-hypersurface that is smooth everywhere except $\{ t=a_\varrho\}$.
  This $C^1$-singularity  will be smoothened out at the end of this section.
\end{construction}

\begin{figure}[ht]
    \begin{tikzpicture}
        % \grid
        \node at (0,0) {\includegraphics[width=.95\textwidth]{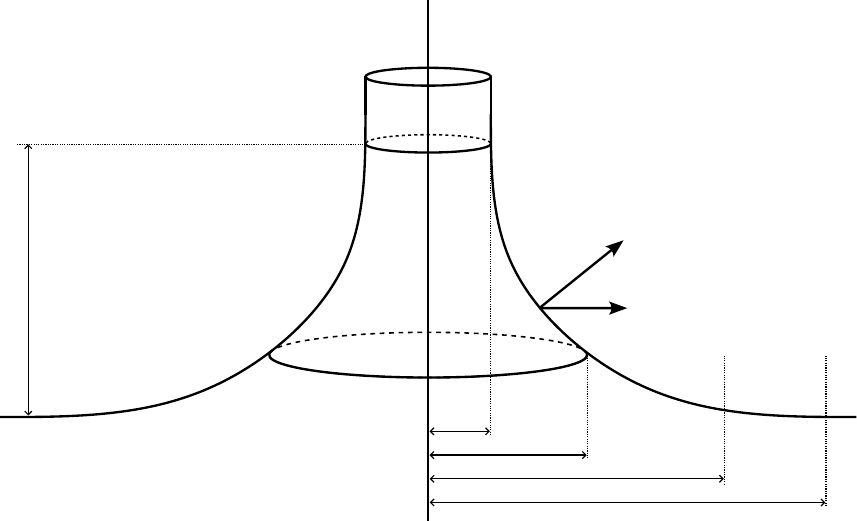}};
        \node at (0.5,-2.2) {$\varrho^4$};
        \node at (1.7,-2.53) {$\varrho^2$};
        \node at (3.35,-2.86) {$\sqrt{\varrho}$};
        \node at (5.05,-3.2) {$2\sqrt{\varrho}$};
        \node at (-5.2,0) {$a_\varrho$};
        \node at (0.7,3.5) {$t$-axis};
        \node at (3,0.6) {$\overline{\nu}$};
        \node at (3.5,-0.8) {$\overline{\nu}^\intercal$};
    \end{tikzpicture}
    \caption{The hypersurface $\overline\Sigma_\varrho$ as the graph of a spherically symmetric function with a cylinder attached.}\label{fig:handle-in-one-fiber}
\end{figure}
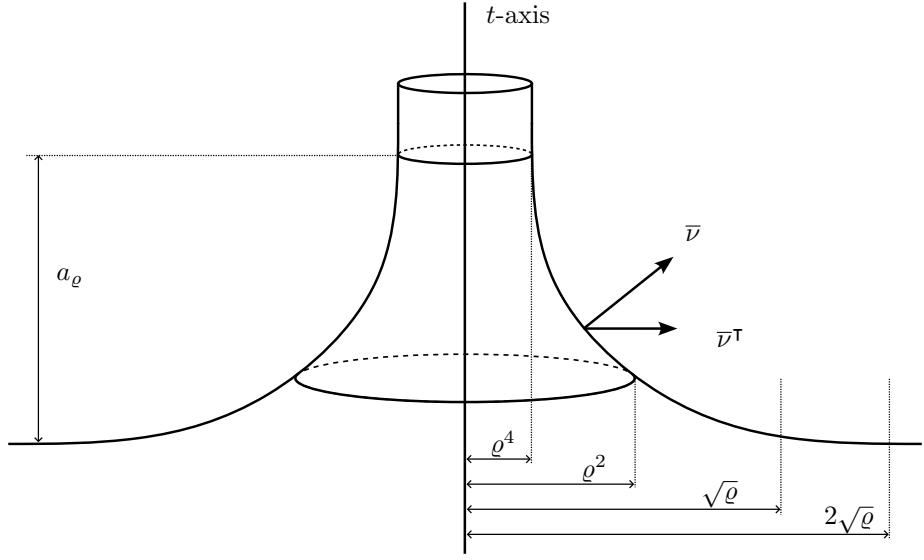

\begin{remark}
    Let us motivate the choice of the function $\xi_\varrho$.
    The graph of the functions $ \ell_\varrho(r)=-2\sqrt{2\varrho}\sqrt{r-2\varrho}$ gives rise to spherically symmetric surfaces satisfying
    $$-2\lambda_1=\lambda_2=\dots=\lambda_k$$
    for their principal curvatures, cf. \cref{Lemma: spherical symmetry principal curvatures}.
    In $n=3$ these are the (spatial) Schwarzschild metrics of mass $m=\varrho$.
    In this case the scalar curvature vanishes which does not allow us to absorb the error terms appearing in the handle construction.
    Therefore, we solve the ODE
    $$-4\lambda_1=\lambda_2=\dots=\lambda_k$$  
    which leads to a spherically symmetric surface with strictly positive scalar curvature.
    The corresponding graph function is given precisely by $\xi_\varrho$.
\end{remark}

Let $\overline\nu_\varrho$ be the unit normal vector field of $\overline\Sigma_\varrho$ and denote by $\overline\nu_\varrho^{\intercal}$ the projection of $\overline\nu_\varrho$ onto the tangent space of the slices $B_1^k(0)\times\{t\}$, i.e.,
\begin{align}\label{eq:definition-of-normal-in-Euclidean}
    \overline\nu^\intercal_\varrho=\overline\nu_\varrho-\langle \overline\nu_\varrho,\partial_t\rangle\partial_t.
\end{align}
Note that $\langle \overline \nu_{\varrho},\partial_t\rangle\ge0$ since $\overline\Sigma_\varrho$ is a graph.
We have the following properties.

\begin{proposition}\label{Prop: Euclidean Handle}
    The hypersurfaces $\overline\Sigma_\varrho$ from \cref{Construction: Euclidean handle} satisfy the following properties.
    \begin{enumerate}
        \item $\overline\Sigma_\varrho$ is $\mathrm{O}(k)$-invariant and satisfies
        \begin{align*}
            \overline\Sigma_\varrho\cap (B_{1}^k(0)\times[2\varrho,1])&=\partial B_{\varrho^4}^k(0)\times[2\varrho,1],\\
            \overline\Sigma_\varrho\cap \Bigl(\bigl(B_{1}^k(0)\setminus B_{2\sqrt \varrho}^k(0)\bigr)\times[0,1]\Bigr)&=\bigl(B_{1}^k(0)\setminus B_{2\sqrt \varrho}^k(0)\bigr)\times\{0\}.
        \end{align*}
        \item \label{deux} For $r\in [\varrho^4,\sqrt \varrho]$,  we have 
        \begin{align*}
            H(\overline\Sigma_\varrho)>\frac12|A(\overline\Sigma_\varrho)|\quad\text{if }k\ge2\\
            \scal(\overline\Sigma_\varrho)>\frac13|A(\overline\Sigma_\varrho)|^{2}\quad\text{if }k\ge3.
        \end{align*}
        \item \label{trois} For $r\in[\varrho^4,\varrho^2]$, $k\ge3$, we have 
        \[\min\left\{H(\overline\Sigma_\varrho)^2, \scal(\overline\Sigma_\varrho)\right\}>\tfrac{1}{2} \varrho^{-3}.\]
    \end{enumerate}
    For the final two properties let $\overline C\coloneqq 18k+ {432}$.
    \begin{enumerate}
        \setcounter{enumi}{3}
        \item \label{final1} For $r\ge \varrho^{2}$, we have $|\overline\nu_\varrho^{\intercal}| < \overline C\,\sqrt \varrho$.
        \item \label{final2} For $r\ge \sqrt{\varrho}$, we have $|A(\overline\Sigma_\varrho)| < \overline C\,\varrho^{3/8}$ and $\scal(\overline\Sigma_\varrho) > -\overline C\,\varrho^{3/4}$.
    \end{enumerate}
\end{proposition}

\begin{proof}
    The first assertion follows directly from the construction, the only thing to verify is that $a_\varrho<2\varrho$ for $\varrho\le \frac1{10}$:
    \[
    a_\varrho = -\xi_\varrho(2\sqrt{\varrho}) = \frac{4}{3}\varrho\underbrace{\left(\sqrt{2\sqrt \varrho}-\varrho^{2}\right)^{3/2}}_{<1} + 4\varrho^{3}\underbrace{\sqrt{\sqrt{2\sqrt \varrho}-\varrho^{2}}}_{<1} < \frac43\varrho + \frac4{100}\varrho < 2\varrho.
    \]

    Concerning \emph{(ii)}, we note that on this inner annulus the cutoff function $\eta_\varrho$ equals $1$, so $f_\varrho=\xi_\varrho+a_\varrho$ and therefore
    $f_\varrho'=\xi_\varrho'$ as well as $f_\varrho''=\xi_\varrho''$.
    By \cref{Lemma: construction of function}, we have
    \[
    4\,\xi_\varrho''=-\frac1r\,\xi_\varrho'(1+\xi_\varrho'^2).
    \]
    Combined with \cref{Lemma: spherical symmetry principal curvatures}, gives 
    \[
    \lambda_2=\dots=\lambda_k=-4\,\lambda_1.
    \]
    On the graph of $f_\varrho$ on $[\varrho^4,\sqrt{\varrho}]$.
    Therefore, we get $\lambda_1 = -(16k-15)^{-1/2}|A|$ and thus
    \begin{align*}
        H(\overline\Sigma_\varrho) &= \lambda_1 - 4(k-1)\lambda_1 = \frac{4k-5}{\sqrt{16k-15}}|A|\\
            &> \frac12|A| &&\text{if }k\ge2\\
        \scal(\overline\Sigma_\varrho) &= 16(k-1)(k-2)\lambda_1^2 -8(k-1)\lambda_1^2\\
            &= 8(k-1)(2k-5)\lambda_1^2 \\
            &= \frac{8\left(k-1\right)\left(2k-5\right)}{16k-15}|A|^{2}\\
            &> \frac13 |A|^2  &&\text{if } k\ge3.
    \end{align*}
    On the other hand, the cylinder of radius $\varrho^4$ has principal curvatures: $\lambda_1=0$ and $\lambda_2=\dots=\lambda_k=\varrho^{-4}\not=0$ and therefore
    \begin{align}\label{Eq: curvature of cylinder}
        \begin{aligned}
            H(\overline\Sigma_\varrho) &=  \sqrt{k-1}\,|A| > \frac12|A| && \text{if }k\ge2\\
            \scal(\overline\Sigma_\varrho) &=  (k-2)|A|^2 > \frac13 |A|^2  &&\text{if } k\ge3.
        \end{aligned}
    \end{align}

    Next, we show \emph{(iii)}. On the cylindrical part, we have $|A| = \sqrt{k-1}\varrho^{-4}$, so the desired inequality holds by \eqref{Eq: curvature of cylinder}. 
    On the \enquote{graphical part}, we observe that both $\xi_\varrho'$ and the function $x\mapsto \frac{-x}{\sqrt{1+x^2}}$ are monotone decreasing and thus, for $r\le\varrho^2$:
    \begin{align*}
        \lambda_2(r) &= -\frac{\xi_\varrho'(r)}{r\sqrt{1+\xi_\varrho'(r)^2}} \ge -\frac{\xi_\varrho'(\varrho^2)}{\varrho^2\sqrt{1+\xi_\varrho'(\varrho^2)^2}}\\
            &= \frac{1}{\varrho^2\sqrt{1+\frac{\varrho^2}{\varrho-\varrho^2}}} \frac{\varrho}{\sqrt{\varrho-\varrho^2}} = \frac{\sqrt{\varrho-\varrho^2}}{\varrho\sqrt{\varrho}}\frac{1}{\sqrt{\varrho-\varrho^2}} = \frac1{\varrho^{3/2}}.
    \end{align*}
    Hence, we obtain
    \begin{align*}
        H ={}& -\frac14\lambda_2 + (k-1)\lambda_2 = (k-\tfrac54)\lambda_2\\
             \ge{}& (k-\tfrac54)\varrho^{-3/2} \\
        \scal ={}& (k-1)(k-2)\lambda_2^2 - \frac12(k-1)\lambda_2^2\\
            \ge{}& (k-1)(k-\tfrac52)\varrho^{-3} 
    \end{align*}
    and if $k\ge3$, then $H> \varrho^{-3/2}$ and $\scal\ge \varrho^{-3}$, which proves \emph{(iii)}.

    \medskip

    Next, we show the estimate from \emph{(iv)} for $r\le \sqrt{\varrho}$. Since $f_\varrho'=\xi_\varrho'$ on $[\varrho^2,\sqrt \varrho]$ and thus $f_\varrho'$ is monotone increasing, we have 
    \[
    0\ge f_\varrho'(r)\ge \xi_\varrho'(\varrho^2)
    =-\frac{\varrho}{\sqrt{\varrho-\varrho^2}}\ge -2\sqrt\varrho,
    \]
    by \cref{Lemma: construction of function}.
    Hence, $|f_\varrho'(r)|\le 2\sqrt{\varrho}$ for $r\in [\varrho^2,\sqrt{\varrho}]$.
    Moreover, 
    \[
    |\overline\nu_\varrho^\intercal|=\frac{|\nabla f_\varrho|}{\sqrt{1+|\nabla f_\varrho|^2}}\le |\nabla f_\varrho| = |f_\varrho'| \le 2\sqrt{\varrho},
    \]
    which proves the claim for $\varrho^2\le r\le \sqrt{\varrho}$.

    \medskip

    Next, let us prove \emph{(iv)} and \emph{(v)} for $r\ge \sqrt{\varrho}$.
    Before we can do so, we need some estimates for the derivatives of $f_\varrho$ on $[\sqrt \varrho,2\sqrt \varrho]$.
    To this end, recall that
    \[
    \xi_\varrho'(r)=-\frac{\varrho}{\sqrt{\sqrt r-\varrho^{2}}},\qquad
    4\xi_\varrho''(r)=-\frac1r\xi_\varrho'(r)\left(1+\xi_\varrho'(r)^2\right),
    \]
    by \cref{Lemma: construction of function}.
    Furthermore, since $\varrho$ was assumed to be smaller than $\tfrac{1}{10}$, we have
    \[\varrho^{1/4}-\varrho^{2} = \frac12\varrho^{1/4} + \varrho^{1/4}\left(\frac12-\varrho^{3/4} \right)\ge \frac12 \varrho^{1/4},
    \]
    and hence
    \begin{align*}
        |\xi_\varrho'(\sqrt \varrho)| ={}& \frac{\varrho}{\sqrt{\varrho^{1/4}-\varrho^{2}}}\le \sqrt{2}\,\varrho^{7/8},\\
        |\xi_\varrho''(\sqrt \varrho)| \le{}& \frac{1}{4\sqrt \varrho}\,|\xi_\varrho'(\sqrt \varrho)|\left(1+\xi_\varrho'(\sqrt \varrho)^2\right)\le \frac{\sqrt{2}}{4}\,\varrho^{3/8}\left(1+2\varrho^{14/8}\right) \le \varrho^{3/8}
    \end{align*}
    Moreover, $\xi_\varrho$ is strictly decreasing and thus, for $r\in[\sqrt \varrho, 2\sqrt \varrho]$, the fundamental theorem of calculus yields
    \begin{align*}
        0\le{}& \xi_\varrho(r)+a_\varrho
        =-\int_{r}^{2\sqrt \varrho}\xi_\varrho'(s)ds \\
        \le{}& (2\sqrt \varrho -r)|\xi_\varrho'(r)|
        \le \sqrt \varrho\,|\xi_\varrho'(\sqrt \varrho)|
        \le \sqrt2\,\varrho^{11/8}.
    \end{align*}
    Therefore, since $\xi_\varrho'$ is negative and strictly increasing we obtain for $r\in[\sqrt \varrho, 2\sqrt \varrho]$,
    \begin{align*}
        |f_\varrho'(r)| \le{}& |\eta(r)\xi_\varrho'(r)+\eta_\varrho'(r)(a_\varrho+\xi_\varrho(r)) |\\
             \le{}& -\xi_\varrho'(\sqrt \varrho)+\frac2{\sqrt \varrho}\,(a_\varrho+\xi_\varrho(\sqrt \varrho))\\
             \le{}& \sqrt{2}\,\varrho^{7/8} + \frac{2\sqrt{2}}{\sqrt \varrho}\,\varrho^{11/8} < 3\sqrt{2}\, \varrho^{7/8}.
    \end{align*}
    By the differential equation $4\xi_\varrho''=-\frac1r\xi_\varrho'\left(1+\xi_\varrho'^2\right)$ and the monotonicity of $\xi_\varrho'$, we observe that $\xi_\varrho''$ is positive and strictly decreasing and therefore, 
    \begin{align*}
        |f_\varrho''(r)|\le{}& |\eta_\varrho(r)\xi_\varrho''(r)|+|2\eta_\varrho'(r)\xi_\varrho'(r)|+|\eta_\varrho''(r)(a_\varrho+\xi_\varrho(r))\\
            \le{}&\xi_\varrho''(\sqrt \varrho)-\frac4{\sqrt \varrho}\,\xi_\varrho'(\sqrt \varrho)+\frac {10}\varrho\,(a_\varrho+\xi_\varrho(\sqrt \varrho))\\
            \le{}& \varrho^{3/8} + 4\sqrt{2}\,\varrho^{3/8} + 10\sqrt{2}\,\varrho^{3/8} < 15\sqrt{2}\, \varrho^{3/8}
    \end{align*}
    for $r\ge \sqrt \varrho$.

    \medskip

    Now, let us turn to proving \emph{(iv)} and \emph{(v)} for $r\ge \sqrt \varrho$.
    We observe that as before, we have $|\overline\nu_\varrho^\intercal|\le |f_\varrho'|$ and therefore
    \[
    |\overline\nu_\varrho^\intercal|\le 3\sqrt{2}\, \varrho^{7/8} = 3\sqrt2\,\varrho^{3/8} \sqrt \varrho < 3\sqrt2\sqrt{\varrho}.
    \]
    Furthermore, by \cref{Lemma: spherical symmetry principal curvatures}, the principal curvatures of the graph of $f_\varrho$ are given by 
    \begin{align*}
        |\lambda_1|={}&\frac{|f_\varrho''|}{(1+f_\varrho'^2)^{3/2}} \le 15\sqrt{2}\, \varrho^{3/8}\\
        \lambda_2=\dots=\lambda_k={}&-\frac{f_\varrho'}{r\sqrt{1+f_\varrho'^2}} {\le} \frac{3\sqrt{2}\,\varrho^{7/8}}{\sqrt \varrho} = 3\sqrt{2}\, \varrho^{3/8}
    \end{align*}
    for $r\in[\sqrt \varrho,2\sqrt \varrho]$.
    In particular, we obtain
    \begin{equation}\label{eq:v-for-A}
        |A(\overline\Sigma_\varrho)|(r)  = \sqrt{\lambda_1^{2}+(k-1)\lambda_2^{2}} \le 3\sqrt{2(k+24)}\,\varrho^{3/8}.
    \end{equation}
    The scalar curvature estimate can be derived from the Gauss equation
    \[
    \scal(\overline\Sigma_\varrho) = H(\overline\Sigma_\varrho)^2 - |A(\overline{\Sigma}_\varrho)|^2\ge- |A(\overline{\Sigma}_\varrho)|^2, 
    \]
    which together with \eqref{eq:v-for-A} yields the desired inequality and hence concludes the proof of \emph{(v)}.
    Lastly, we point out that the analysis still holds after smoothing out the $\overline \Sigma_\varrho$ at the gluing region $t=a_\varrho$, see for instance \cite{BaerHanke1}.
\end{proof}

We close our analysis of the Euclidean handle with the following monotonicity result.
\begin{proposition}\label{lem:monotonicity-of-Euclidean-handles}
    For every sufficiently %$R\in (0,\tfrac{2}{\sqrt{10}}]$
    small $R \in (0, \tfrac{2}{\sqrt{10}}]$ and $\varrho_1\in(0,R^2/4]$, there exists a smooth path of diffeomorphisms $\Phi_{\varrho_1,\varrho}:\overline\Sigma_{\varrho_1}\to\overline\Sigma_\varrho$, $\varrho\in(0,\varrho_1]$, such that $\Phi_{\varrho_1,\varrho_1}=\id_{\overline\Sigma_{\varrho_1}}$ and for the metrics $g_{\varrho,\delta}$ on $\overline\Sigma_\varrho$ induced by the Euclidean metric $\delta+\dt^2$, the assignment
    \[\varrho\mapsto e^{\tfrac12\varrho}\cdot\Phi_{\varrho_1,\varrho}^\ast(g_{\varrho,\delta})\] 
    is strictly monotonically increasing.
\end{proposition}
We remark that the exponent $\tfrac12$ can be replaced by an arbitrarily small positive constant.

\medskip

Before we can prove \cref{lem:monotonicity-of-Euclidean-handles}, we need some preparation.
Since the Euclidean handle $\overline\Sigma_\varrho\subseteq B_{1}^k(0)\times[0,1]$ is rotationally symmetric, it can be obtained by rotating a curve $\overline\gamma_\varrho$.
Hence, we may write the metric $g_{\overline\Sigma_\varrho}$ on $\overline \Sigma_\varrho$ induced by the Euclidean metric $\delta+\dt^2$ in warped product form
\begin{equation} \label{warpedprod}
g_{\varrho,\delta}=d\tau^2+r_\varrho(\tau)^2g_{S^{k-1}},
\end{equation}
where $\tau$ is the arc-length along the generating curve $\overline\gamma_\varrho$ and $r_\varrho(\tau)$ is the radius-function.
The parametrization is such that $r_\varrho$ is monotonically non-increasing.
For $R\in(0,\tfrac{2}{\sqrt{10}}]$ and all $\varrho<R^2/4$, we denote by $d_\varrho$ the length of the curve $\overline\gamma_\varrho$ until radius $R$. 
Since $\Sigma_\varrho$ ends in a straight cylinder, the radial function satisfies $r_\varrho|_{[d_\varrho-1+a_\varrho,d_\varrho]} = \varrho^4$.
We have the following estimate on $d_\varrho$ and its $\varrho$-derivative.

\begin{lemma}\label{Lemma length estimate for handles}
    There exists  $R \in (0, \tfrac{2}{\sqrt{10}}]$, depending only on $n$, such that for every $\varrho<R^2/4$, we have 
    \begin{align*}
        | \partial_\varrho d_\varrho | < \varrho^{1/4} \qquad\text{and}\qquad 1<d_\varrho<1+R.
    \end{align*}
\end{lemma}

For the following proof, we use the notation $A\lesssim B$ to mean that there exists a positive constant $C$ just depending on $n$ such that $A\le C\cdot B$.

\begin{proof}
    Recall the function $f_\varrho(r)\coloneqq\eta_\varrho(r)\bigl(\xi_\varrho(r)+a_\varrho\bigr)$.
    The length $d_\varrho$ of $\overline\gamma_\varrho$ is given
    \begin{align*}
        d_\varrho=\int_{\varrho^4}^{4\varrho^4}\sqrt{1+f_\varrho'(r)^2}dr + \int_{4\varrho^4}^{R}\sqrt{1+f_\varrho'(r)^2}dr+(1-a_\varrho),
    \end{align*}
    and we thus need to compute
    \begin{align*}
        \partial_\varrho d_\varrho={}&\underbrace{\partial_\varrho\int_{\varrho^4}^{4\varrho^4}\sqrt{1+f_\varrho'(r)^2}dr}_{\eqqcolon A_0}  
        \underbrace{-16\varrho^3\sqrt{1+f_\varrho'(4\varrho^4)^2}}_{\eqqcolon A_1} + \underbrace{\int_{4\varrho^4}^{R}\frac{f_\varrho'(r)\partial_\varrho f_\varrho'(r)}{\sqrt{1+f_\varrho'(r)^2}} dr}_{\eqqcolon A_2}  -\partial_\varrho a_\varrho
    \end{align*}
    We estimate $A_0$, $A_1$ and $A_2$ and $\partial_\varrho a_\varrho$ individually, starting with $A_0$.
    We note that an antiderivative of $\sqrt{1+f'_\varrho(r)^2}$ is given by
    \begin{align*}
        r^{1/4}\sqrt{\sqrt{r}-\varrho^{2}}\left(\sqrt{r}+\tfrac{3}{2}\varrho^{2}\right)
        +\frac{3}{2}\varrho^{4}\,\operatorname{arsinh}\left(\frac{\sqrt{\sqrt{r}-\varrho^{2}}}{\varrho}\right).
    \end{align*}
    Since $\sqrt{\varrho^4}-\varrho^{2}=0$ and $\sqrt{\sqrt{4\varrho^4}-\varrho^{2}}=\varrho$, we get 
    \[
| A_0 | = \left|  \partial_\varrho\int_{\varrho^4}^{4\varrho^4}\sqrt{1+f_\varrho'(r)^2}dr\right|\lesssim \varrho^3.
    \]
    Next, we turn to $A_1$.
    Since
    \begin{align*}
 f_\varrho'(4\varrho^4)=\xi_\varrho'(4\varrho^4)=-\frac{\varrho}{\sqrt{\sqrt{4\varrho^4}-\varrho^{2}}}=-1,
    \end{align*}
    we obtain $| A_1| =| -16\sqrt 2 \varrho^3| \lesssim \varrho^3$.

    \medskip

    Next, we compute
     \begin{align*}
         \begin{split}
             \partial_\varrho a_\varrho={}&\partial_\varrho\left(\frac{4}{3}\varrho\Big(\sqrt2 \varrho^{1/4}-\varrho^{2}\Big)^{3/2} + 4\varrho^{3}\sqrt{\sqrt2 \varrho^{1/4}-\varrho^{2}}\right)\\
             ={}&\frac43 \Big(\sqrt2 \varrho^{1/4}-\varrho^{2}\Big)^{3/2}+2\varrho\Big(\sqrt2 \varrho^{1/4}-\varrho^{2}\Big)^{1/2}\Big( \frac{\sqrt 2}{4} \varrho^{-3/4}-2\varrho\Big)\\
             &+12\varrho^2\sqrt{\sqrt2 \varrho^{1/4}-\varrho^{2}}+2\varrho^3\Big(\sqrt2 \varrho^{1/4}-\varrho^{2}\Big)^{-1/2}\Big(  \frac{\sqrt 2}{4} \varrho^{-3/4}-2\varrho  \Big)\\
             \le{}&11\varrho^{3/8} + 4\varrho^{1+1/8-3/4} + 24\varrho^{2+1/8} + 2\varrho^{3-1/8-3/4}.
         \end{split}
     \end{align*}

Now we turn to $A_2$. We have
\begin{align*}
   | A_2 | \le 
     \int_{4\varrho^4}^{R}|f_\varrho'(r)||\partial_\varrho f_\varrho'(r)| dr.
\end{align*}
Note that $f_\varrho=\xi_\varrho$ for $r\le \sqrt\varrho$.
We obtain
\begin{align*}
    |f_\varrho'(r)|\cdot|\partial_\varrho f_\varrho'(r)|
   \le& 1\cdot 2\varrho^{-1} &&\text{for  $r\in[4\varrho^4,\varrho^2]$},\\
     |f_\varrho'(r)|\cdot|\partial_\varrho f_\varrho'(r)|
     \le& 2\sqrt\varrho \cdot 2\varrho^{-1/2} &&\text{for $r\in[\varrho^2,\varrho]$},\\
    |f_\varrho'(r)|\cdot|\partial_\varrho f_\varrho'(r)|
     \le& 2\varrho^{3/4} \cdot 2\varrho^{-1/4} &&\text{for $r\in[\varrho,\sqrt\varrho]$}.
\end{align*}
    Moreover, we have for $r\in[\sqrt\varrho,2\sqrt\varrho]$:
    \begin{align*}
        &|f_\varrho'(r)|\cdot|\partial_\varrho f_\varrho'(r)| \\
        &\quad\le \Big( |\eta_\varrho' || \xi_\varrho+a_\varrho|+|\eta_\varrho|| \xi_\varrho'| \Big)\\
        &\qquad\qquad \cdot
        \Big( |\partial_\varrho\eta_\varrho' || \xi_\varrho+a_\varrho|+|\partial_\varrho \eta_\varrho|| \xi_\varrho'| 
        + |\eta_\varrho' || \partial_\varrho\xi_\varrho+\partial_\varrho a_\varrho|+|\eta_\varrho|| \partial_\varrho\xi_\varrho'| \Big)\\
        &\quad\lesssim\Big ( \varrho^{-1/2} \cdot \varrho^{5/4}+1\cdot \varrho^{3/4} \Big)\\
        &\qquad\qquad\cdot\Big(\varrho^{-2} \cdot \varrho^{5/4} + \varrho^{-3/2}\cdot \varrho^{3/4}
        +\varrho^{-1/2}\cdot ( 1+1) + 1\cdot \varrho^{-1/8} \Big)\\
        &\quad\lesssim \varrho^{3/4}\cdot\varrho^{-3/4},
    \end{align*}
    where we used $\partial_\varrho a_\varrho\le 1$ and that on $[\sqrt\varrho, 2\sqrt\varrho]$ by the fundamental theorem of calculus
    \begin{align*}
        |f_\varrho|\le \sqrt\varrho| \xi'(\sqrt\varrho)|\le 2\sqrt \varrho \cdot \varrho^{3/4}=2\varrho^{5/4}.
    \end{align*}

    Finally, for $r\ge2\varrho$, $f_\varrho\equiv0$ and therefore, putting the pieces together and integrating, we obtain
    \[
        | A_2| \lesssim \sqrt{\varrho} . 
    \] 
    Hence, 
    \[ 
      | \partial_\varrho d_\varrho| \lesssim \sqrt{\varrho}. 
    \]
    Since the curve $\overline\gamma_\varrho$ is shorter than the one consisting of straight ones parallel to the coordinate-axes and longer than the straight line connecting its end points, we get 
    \[1\le \sqrt{1+R^2}<d_\varrho< 1+R\qedhere\]
\end{proof}

\begin{proof}[Proof of \cref{lem:monotonicity-of-Euclidean-handles}]
    Fix the bump function $\beta(x)=\eta(4x)\eta(-4x)$, where $\eta$ is the cut-off function introduced earlier satisfying $-2\le \eta'\le0$ and $|\eta''|<10$, and we note that $\operatorname{supp}\beta \subseteq[-\frac14,\frac14]$.
    Next, we define a smooth family or reparameterizations $\phi_{\varrho_1,\varrho}:[0,d_{\varrho_1}]\to[0,d_\varrho]$ for $\varrho\in(0,\varrho_1]$ by
    \begin{align}\label{reparametrization family}
        \phi_{\varrho_1,\varrho}(\tau)=\tau-q_{\varrho_1,\varrho}\int_0^\tau\beta\left(s-\frac12\right)ds
    \end{align}
    where $q_{\varrho_1,\varrho}=(d_{\varrho_1}-d_{\varrho}) \left(\int_{\mathbb R}\beta(s)ds\right)^{-1}$.
    We now define a diffeomorphism
    \[
    \overline \Phi_{\varrho_1,\varrho}:\overline\Sigma_{\varrho_1} \longrightarrow \overline\Sigma_\varrho,\quad (\tau, x)\mapsto (\phi_{\varrho_1,\varrho}(\tau), x)
    \]
    where $x$ denotes the spherical coordinates in the warped product. 
    Note that 
    \[\overline\Phi_{\varrho_1,\varrho}^*g_{\varrho,\delta}(\tau,x) = (\partial_\tau\phi_{\varrho_1,\varrho}(\tau))^2\dtau^2 + r_\varrho(\phi_{\varrho_1,\varrho}(\tau))^2g_{S^{k-1}}.\]
    A direct inspection of the warping functions $r_\varrho$ gives for any $\varrho\ge \widetilde \varrho$
    \begin{align*}
        r_{\varrho}\left(\phi_{\varrho_1, \varrho}(\tau)\right) \le r_{\widetilde\varrho}\left(\phi_{\varrho_1,\widetilde \varrho}(\tau)\right),\qquad \tau\in[0,d_{\varrho_1}].
    \end{align*}
This implies 
    \begin{align*}
        \partial_\varrho  r_\varrho(\phi_{\varrho_1,\varrho}(\tau))
        =&(\partial_\varrho r_\varrho)(\phi_{\varrho_1,\varrho}(\tau))+r_\varrho'(\phi_{\varrho_1,\varrho}(\tau))\partial_\varrho \phi_{\varrho_1,\varrho}(\tau)\ge 0
    \end{align*}
    where we used that $r_\varrho'$ vanishes on the support of $\partial_\varrho \phi_{\varrho_1,\varrho}(\tau)$.
    Next, we observe that
    \begin{align*}
        \partial_\tau \phi(\tau)=1-(d_{\varrho_1} - d_{\varrho})\beta\left(\tau-\frac12\right)
    \end{align*}
    and therefore 
    \begin{align*}
        \partial_\varrho\partial_\tau\phi(\tau) = \beta\left(\tau-\frac12\right) \partial_\varrho d_\varrho - (d_{\varrho_1} - d_\varrho)\beta'(\tau-\frac12)>-\frac12\partial_\tau\phi(\tau)
    \end{align*}
    {where we used that $(d_{\varrho_1} - d_{\varrho})\ll1$ and that $\beta'$ is bounded.}
     Consequently, we obtain for $\varrho\in[\varrho_0,\varrho_1]$
    \begin{align*}
        \partial_\varrho \overline\Phi_{\varrho_1,\varrho}^*g_{\varrho,\delta}(\tau,x) \ge-\frac12 \overline\Phi_{\varrho_1,\varrho}^*g_{\varrho,\delta}(\tau,x)
    \end{align*}
    which finishes the proof.
\end{proof}

To conclude the construction of $\overline \Sigma$, we will  now smoothen the $C^1$-singularity of $\overline \Sigma_{\varrho}$ at $\{ t = a_{\varrho}\}$ while preserving  \cref{Prop: Euclidean Handle} and \cref{Lemma length estimate for handles}, hence also \cref{lem:monotonicity-of-Euclidean-handles}.
Fix some $R \in (0, \tfrac{2}{\sqrt{10}}]$ so that \cref{Lemma length estimate for handles} holds for this $R$.
Consider the smooth manifold 
\[
   V = \{ (\varrho , t) \mid \varrho \in (0, R^2 /4) ,t \in (f_{\varrho}(\varrho^2) , a_{\varrho}] \} \subset (0, R^2/4) \times (0,1]
\]
with non-compact boundary $\partial V = \{ ( \varrho, a_{\varrho}) \mid \varrho \in (0, R^2/4)\} \subset V$.

\medskip

Put $f_0 \colon V \to \bbR$, $f_0(\varrho, t) = \xi_{\rho}^{-1}(t- a_{\varrho})$. 
Note that for fixed $\varrho$, the function $f_0(\varrho, -) \colon (f_{\varrho}(\varrho^2), a_{\varrho}] \to \bbR$ is the inverse of the function $[\varrho^4, \varrho^2) \to \bbR$, $r \mapsto f_\varrho(r)$.

\medskip

Let $\omega \colon V \to \bbR_{>0}$ be a smooth function. 
For each $\varrho \in (0, R^2 /4)$, we obtain the surface of revolution
\[
  S(\omega, \varrho) := \{ r = \omega(t) \mid t \in (f_{\varrho}(\varrho^2), a_{\varrho}] \} \subset \bbR^n \times (f_{\varrho}(\varrho^2), a_{\varrho}] . 
\]
Let $J^2V \to V$ be the bundle of $2$-jets of strictly positive real valued smooth functions on $V$.
We consider the open partial differential relation $\mathfrak{R}  \subset J^2 V$ determined by the following pointwise conditions for $\omega$, working with the numbers  $|\partial_{\varrho} d_\varrho| < \varrho^{1/4}$, from \cref{Lemma length estimate for handles}:
\begin{enumerate}
    \item $H(S(\omega, \varrho))>\frac12|A(S(\omega, \varrho))|$ if $k\ge2$, 
    \item $\scal(S(\omega, \varrho))>\frac13|A(S(\omega, \varrho))|^{2}$ if $k\ge 3$,
    \item $\min\left\{H(S(\omega, \varrho))^2, \scal(S(\omega, \varrho))\right\}>\tfrac{1}{2} \varrho^{-3}$
    \item \label{der_length} $\left| \partial_\varrho\left( \sqrt{ 1 + \omega'(\varrho, -)} - \sqrt{ 1 + f'_0(\varrho,-)} \right)\right|< \varrho^{1/4} - | \partial_\varrho d_\varrho |$. 
\end{enumerate}
By  \cref{Prop: Euclidean Handle}, $f_0$ solves $\mathfrak{R}$ over $V$.
Consider the smooth path $F \colon [0,1] \to C^{\infty}(V, \bbR_{>0})$, 
\[
  F(\tau)(\varrho, t) := (1-\tau) f_0(\varrho, t) + \tau \varrho^{4}.
\]
We have $j^1F(\tau)|_{\partial V} = j^1 f_0|_{\partial V}$ for all $t$.
Carrying out the computation in the proof of part \ref{deux} and \ref{trois} of \cref{Prop: Euclidean Handle} with $\lambda_2 = \ldots = \lambda_k = \varrho^4$ and $\lambda_1 \in [-\tfrac{1}{4} \varrho^4, 0]$, one sees that $F$ solves $\mathfrak{R}$ along $\partial V$ for all $\tau$ (for $\omega = F(\tau)$, the left-hand side of \ref{der_length} is equal to $0$ along $\partial V$).
By continuity, there exists an open neighborhood $\partial V \subset U \subset V$ such that $F$
solves $\mathfrak{R}$ over $U$ for all $t$.

\medskip

By the local flexibility lemma  \cite[Theorem 1.2]{BaerHanke1} applied to the closed subset $\partial V \subset V$, there exists an open neighborhood $\partial V \subset U_0 \subset U$ satisfying $\overline U_0 \subset V$ and a smooth path $f \colon [0,1] \to C^{\infty}(V , \bbR_{>0})$ such that $f(0) = f_0$, $f(\tau) = F(\tau)$ on $U_0$, $f(\tau) = f_0$ on $V \setminus U$ and $f(\tau)$ solves $\mathfrak{R}$ over $V$ for all $\tau$.

\medskip

For $\varrho \in (0,R^2/4)$, define
\begin{itemize}
 \item the smooth function $\psi_{\varrho} \colon (f_{\varrho}(\varrho^2) , a_{\varrho}] \to \bbR_{>0}$, $\psi_{\varrho}(t) = f(1)(\varrho, t)$.
 It is equal to $\varrho^4$ near $a_{\varrho}$ and equal to $f_0(\varrho,-)$ near $f_{\varrho}(\varrho^2)$;
 \item the surface of revolution
\[
  S(\varrho) := \{ r = \psi_{\varrho}(t) \mid t \in (f_{\varrho}(\varrho^2) , a_{\varrho}] \} \subset B_1^k(0) \times (f_{\varrho}(\varrho^2), a_{\varrho}];
\]
 \item the smooth hypersurface
\[
   \widetilde \Sigma_{\varrho} := \{ (r, t) \in \overline \Sigma_{\varrho} \mid t \notin (f_{\varrho}(\varrho^2), a_{\varrho}] \} \; \bigcup \; S(\varrho) \subset B_1^k(0) \times [0,1].
\]
\end{itemize}
The hypersurface $\widetilde \Sigma_{\rho}$ satisfies the properties in \cref{Prop: Euclidean Handle} by the choice of $\mathfrak{R}$. 
Note that parts \ref{final1} and \ref{final2} of \cref{Prop: Euclidean Handle} are satisfied since the relevant parts of $\overline \Sigma_\varrho$ are not affected by the smoothing.

\medskip

Let $\widetilde d_{\rho}$ be defined for $\widetilde \Sigma_{\varrho}$ as before \cref{Lemma length estimate for handles}.
Using that $\overline{\Sigma}_{\varrho} = \widetilde{\Sigma}_{\varrho}$ on $\{ t \notin (f_{\varrho}(\varrho^2), a_{\varrho}]\}$, and that $\psi'_{\varrho}(t)  = f_0'(\varrho, -)(t)$ for $t = f_{\varrho}(\varrho^2)$ and for $t = a_{\varrho}$ (where both sides of the equation  are $0$), we obtain
\begin{align*}
    | \partial_{\varrho}  \widetilde d_{\varrho}|  - | \partial_{\varrho}d_{\varrho} | & \leq \Big| \partial_{\varrho} ( \widetilde d_{\varrho} - d_{\varrho})\Big| \\
    & = \Big|  \partial_{\varrho} \int_{f_{\varrho}(\varrho^2)}^{a_{\varrho}} \left( \sqrt{ 1 + (\psi'_{\varrho})^2} -  \sqrt{ 1 + f_0'(\varrho, -)^2}\right) dt \Big| \\
    & = \Big| \int_{f_{\varrho}(\varrho^2)}^{a_{\varrho}}
    \partial_{\varrho}  \left( \sqrt{ 1 + (\psi'_{\varrho})^2} - \sqrt{ 1 + f_0'(\varrho, -)^2}\right) dt \Big| \\
    & \le \big| a_{\varrho} - f_{\varrho}(\varrho^2)\big| \max_{(f_{\varrho}(\varrho^2), a_{\varrho}]} \Big| \partial_{\varrho} \left( \sqrt{ 1 + (\psi'_{\varrho})^2} - \sqrt{ 1 + f_0'(\varrho, -)^2}\right) \Big| \\
    & < 1 \cdot \left(  \varrho^{1/4} - | \partial_{\varrho}d_{\varrho}|\right) .
\end{align*} 
Adding $|\partial_{\varrho} d_{\varrho}|$ to both sides of this inequality proves the first inequality from \cref{Lemma length estimate for handles}. 
The inequality $1 < \widetilde d_{\varrho}$ is clear, while the inequality $\widetilde d_{\varrho} < 1 + R$ is ensured if $\|f(1) - f_0\|_{C^{1}(V)}$ is sufficiently small, a condition that can be included in the relation $\mathfrak{R}$. 

\medskip

Finally, we replace $\overline \Sigma_{\varrho}$ with the smooth bending handle $\widetilde \Sigma_{\varrho}$. 
By construction,  \cref{Prop: Euclidean Handle} and \cref{Lemma length estimate for handles}, hold for the new $\overline \Sigma_{\varrho}$, which we will use from now on.

\section{Surgery bending in arbitrary backgrounds}\label{sec:handles}
In this section, we will construct the general surgery bending.
To this end, let $\mathcal{S}$  be a compact submanifold of codimension $k$ in an $n$-dimensional Riemannian manifold $(M, g_M)$.
We denote its normal bundle by $\calN_{\mathcal{S}}$.
Recall the constant $R\in(0,\tfrac{2}{\sqrt{10}}]$ from the previous section, which we may choose smaller if necessary, so that the normal exponential map of $\calN_\mathcal{S}$ along $\mathcal{S}$ is a diffeomorphism on the sub-bundle of vectors with length less than or equal to $R$ to the tubular neighborhood $B_{R}(\mathcal{S})$ of radius $R$ around $\mathcal{S}$.
With this identification, we consider $B_R(\mathcal{S})$ as a fibre bundle over $\mathcal{S}$ with fibre $B_r^k(0)$.

\medskip

We will construct a smooth submanifold $\Sigma_\varrho\subset M\times[0,1]$ describing a surgery bending in $\calN_\mathcal{S}$ which for $M = \bbR^k$ with the Euclidean metric and $\mathcal{S}$ a point recovers \cref{Construction: Euclidean handle}.
First, we choose Fermi coordinates around $\mathcal{S}$, that is, we have an identification $U\times B^k_R(0)\cong B_{R}(\mathcal{S})|_{U}$, where $U\subseteq \mathcal{S}$ is a trivialising open subset for $\calN_\mathcal{S}$.
We define for $\varrho\in(0,R^2/4)$
\[\Sigma_\varrho|_{U}\coloneqq U\times \overline\Sigma_\varrho\]
for the smooth submanifold $\overline\Sigma_\varrho\subset B^k_R(0)\times[0,1]$ as in \cref{Construction: Euclidean handle}.
Since $\mathcal{N}_{\mathcal{S}}$ has structure group $\mathrm{O}(k)$, using the $\mathrm{O}(k)$-invariance of $\overline\Sigma_\varrho$ (cf. \cref{Prop: Euclidean Handle}, (i)), the local definitions can be glued together into a global definition of $\Sigma_\varrho\subset B_R(\calS) \times[0,1]$.
Moreover, since $2\sqrt{\varrho}<R$, the hypersurface $\Sigma_\varrho$ can be extended by $(M\setminus B_R(\mathcal{S}))\times\{0\}$ to a smooth hypersurface  $\Sigma_\varrho \subset M\times[0,1]$.

\medskip

The main result of this section is the following theorem, which states the existence of a background metric on $M\times [0,1]$ that extends $g_M$ on $M\times\{0\}$ such that for $\varrho$ small enough, the scalar curvature of $\Sigma_\varrho$ can be uniformly bounded in terms of the scalar curvature of the original metric $g_M$ and, for $\varrho$ even smaller the mean curvature of $\Sigma_\varrho$ is positive.

\medskip

\begin{theorem}\label{theorem:handle-properties}
    Let $k\ge 3$ and $\alpha\in(0,\tfrac13]$. Then, there exists a constant $0<\varrho_1<\min(R^2/4, \tfrac{\alpha}{3})$ such that for all $\kappa, \varrho\in(0,\varrho_1]$ and for the metric
    \[g_{\kappa}=u^{\frac{4}{\,n-2\,}}(g_M+dt^2),\qquad u(x,t) = u_{\kappa}(x,t) = 1 + t \tilde\kappa \eta\left(\alpha^{-1}t-1\right)\]
    on $M\times[0,1]$, where $\tilde\kappa\coloneqq\tfrac{n-2}{2n}\kappa$, the scalar curvature of the induced metric $g_{\kappa,\varrho}$ on $\Sigma_\varrho$ satisfies
    \[\scal(\Sigma_\varrho, g_{\varrho, \kappa})> \scal(M, g_M)\circ\pr_M-\varrho_1^{1/4},\]
    where $\pr_M\colon M\times[0,1]\to M$ is the projection.
    Furthermore, there exists a constant $\varrho_0\in(0,\varrho_1)$ such that the mean curvature of $\Sigma_\varrho$ with respect to $g_{\kappa}$ satisfies
    \[H(\Sigma_\varrho,g_\kappa) > (1-\alpha)\kappa - \varrho^{1/4}\]
    for all $0<\varrho\le \varrho_0$.
\end{theorem}

\begin{remark}\leavevmode
    \begin{enumerate}
        \item We emphasize that $\varrho_1$ is independent of $\kappa$, while $\varrho_0$ will depend on $\kappa$.
        \item The mean curvature of the slice $M\times\{0\}\subset M\times[0,1]$ with respect to the ambient metric $g_{\kappa}$ and the exterior normal is given by 
        \[H(M\times\{0\},g_{\kappa}) = - \frac{2n}{n-2} \tilde\kappa = -\kappa.
        \]
        \item 
        If $k=2$ and $n\ge3$, then the second statement of \cref{theorem:handle-properties} still holds true, that is, the mean curvature of $\Sigma_\varrho$ with respect to $g_{\kappa}$ is still positive. However, the scalar curvature estimate might be incorrect.
    \end{enumerate}
\end{remark}

\begin{remark}\label{remark error terms}
Essentially, the content of \cref{theorem:handle-properties} is that the handle construction of \cref{Prop: Euclidean Handle} extends to arbitrary Riemannian manifolds and to surgeries of any codimension $k\ge3$.
This leads to a number of error terms.
They can be grouped into three different categories:
\begin{enumerate}
    \item Error terms which approach zero as $\varrho \to 0$. These are not problematic.
    \item Error terms of the form $\eps|\overline A_\varrho|$, $\eps\ll 1$, where we note that the second fundamental form $|A_\varrho|$ blows up as $\varrho\to0$. Using the pinching estimates $H(\overline\Sigma_\varrho)>\tfrac12|A(\overline\Sigma_\varrho)|$ and $\scal(\overline\Sigma_\varrho)>\tfrac13|A(\overline\Sigma_\varrho)|^{2}$ we can ensure that they can be absorbed.    
    \item Error terms which involve $|\overline \nu^\intercal_\varrho|$. They can be controlled by the dichotomy $\min\{H(\overline\Sigma_\varrho)^2, \scal(\overline\Sigma_\varrho)\}>\tfrac{1}{2} \varrho^{-3}$ for $r\in[\varrho^4,\varrho^2]$ and $|\overline\nu_\varrho^{\intercal}| < \overline C\,\sqrt \varrho$ for $r\ge \varrho^{2}$.
\end{enumerate}
\end{remark}

\medskip

Before diving into the proof of \cref{theorem:handle-properties}, let us introduce its main driving force, which are estimates for scalar and mean curvature, as well as the appropriate conformal transformation formulae. 

\medskip

For the following lemma, we use the notation $A\lesssim B$ to mean that there exists a positive constant $C$ independent of $\varrho$ such that $A\le C\cdot B$.

\begin{lemma}\label{lem:curvature-estimates-black-box}
    There exists a constant $\eps>0$ such that for every $0 < \varrho\le \eps$ and for the metric $g_\varrho$ on $\Sigma_\varrho$ induced by $g+\dt^2$, the following inequality holds:
    \begin{align}\label{eq:scalar-curvature-black-box}
        \begin{split}
            &\scal(\Sigma_\varrho,g_\varrho)- \scal(M, g_M)\circ\pr_M - \scal(\overline\Sigma_\varrho,g_{\varrho,\delta})\\
            &\qquad\qquad \gtrsim-\varrho\cdot\bigl(1+|A(\overline\Sigma_\varrho,\delta+\dt^2)|\bigr)^2 - \|\overline\nu_\varrho^\intercal\|\cdot(1+|A(\overline\Sigma_\varrho,\delta+\dt^2)|)
        \end{split}
    \end{align}
    Here, $\|\overline\nu_\varrho^\intercal\|$ denotes the norm of the tangential part of the normal vector of the hypersurface $\overline\Sigma_\varrho\subset\bbR^{k+1}$ constructed in \cref{Construction: Euclidean handle}  with respect to the Euclidean metric (see also \eqref{eq:definition-of-normal-in-Euclidean}), and $\scal(\overline\Sigma_\varrho,g_{\varrho,\delta})$ denotes the scalar curvature of this hypersurface with respect to the metric $g_{\varrho,\delta}$ induced by the Euclidean ambient metric $\delta+\dt^2$, composed with the projection onto the fiber of the tubular neighbourhood $D(\calS)$.
    Furthermore, we have the following inequality concerning the mean curvature 
    \begin{align}\label{eq:mean-curvature-black-box}
        |H(\Sigma_\varrho,g+\dt^2) - H(\overline\Sigma_\varrho,\delta+\dt^2)| \lesssim  \|\overline\nu_\varrho^\intercal\| + \varrho\cdot (1+ |A(\overline\Sigma_\varrho,\delta+\dt^2)|),
    \end{align}
    and we can compare the $\partial_t$-part of the normal vector $\nu_\varrho$ of $\Sigma_\varrho$ with respect to $g+\dt^2$ and the normal vector $\overline\nu_\varrho$ of $\overline\Sigma_\varrho\subset\bbR^{k+1}$ taken with respect to the Euclidean metric $\delta+\dt^2$ as follows:
    \begin{align}\label{eq:t-part-of-normal-vector}
        |\scpr{\nu_\varrho,\partial_t}_{g+\dt^2} - \scpr{\overline\nu_\varrho,\partial_t}_{\delta+\dt^2}|\lesssim \varrho.
    \end{align}
\end{lemma}

\begin{lemma}\label{lem:conformal-change}
    Let $u$, $g_{\kappa}$ and $g_{\varrho,\kappa}$ be as defined in \cref{theorem:handle-properties}. Then we have the following equalities comparing the scalar and mean curvature before and after conformal transformation by $u^{4/(n-2)}$:
    \begin{align}
        \begin{split}\label{eq:conformal-change-for-scalar-curvature}
            \scal(\Sigma_\varrho,g_{\varrho,\kappa}) ={}& u^{-\frac{n+2}{n-2}}\Bigl(
                \scal(\Sigma_\varrho, g_\varrho) - \frac{4(n-1)}{n-2}(1-\scpr{\nu_\varrho,\partial_t}^2)\partial^2_t u\\ 
                &\qquad\qquad+ \frac{4(n-1)}{n-2}H(\Sigma_\varrho,g+\dt^2)\scpr{\nu_\varrho,\partial_t} \partial_tu
            \Bigr)\\
        \end{split}\\
        \begin{split}\label{eq:conformal-change-for-mean-curvature}
            H(\Sigma_\varrho,g_{\kappa}) = u^{-\frac{2}{n-2}} \left(H(\Sigma_\varrho,g+\dt^2) + \frac{2n}{n-2}\scpr{\nu_\varrho,\partial_t}u^{-1}\partial_tu\right)
        \end{split}
    \end{align}
    where the normal vector $\nu_\varrho$ of $\Sigma_\varrho$ and the inner product with $\partial_t$ are taken with respect to the metric $g+\dt^2$.
\end{lemma}

\cref{lem:conformal-change} is well-known, and we will postpone the proof of \cref{lem:curvature-estimates-black-box} for now and instead turn to the proof of \cref{theorem:handle-properties}. 

\begin{proof}[Proof of \cref{theorem:handle-properties}]    
    In order to estimate the scalar curvature and mean curvature of $g_{\varrho,\kappa}$, we first note that at points with radial coordinate $r\ge2\sqrt\varrho$ we have $\Sigma_\varrho=M$ and $g_\varrho=g$.
    So, the scalar curvature estimate is trivially true here, whereas the estimate on the mean curvature follows from \cref{lem:conformal-change} using $\partial_tu=\tilde\kappa$, $u=1=\scpr{\nu_\varrho,\partial_t}$ and $H(M\times\{0\}, g+\dt^2)=0$. 
    Hence, it suffices to consider $r\le 2\sqrt\varrho$.
    We will distinguish $3$ cases: $r\in[\varrho^4,\varrho^2]$, $r\in[\varrho^2,\sqrt{\varrho}]$ and $r\in[\sqrt\varrho,2\sqrt{\varrho}]$.
    In every case we will first estimate the scalar and mean curvature with respect to the ambient metric $g+\dt^2$ using \cref{lem:curvature-estimates-black-box} before investigating the effect of conformal change on scalar curvature employing \cref{lem:conformal-change}.
    At the very end of the proof, we will prove positivity of the mean curvature after conformal change after passing to a possibly smaller $\varrho_0$, depending on $\kappa$.
    
    Let us start by considering $r\in[\varrho^4,\varrho^2]$, where \cref{Prop: Euclidean Handle} yields
    \begin{align*}
        \scal(\overline\Sigma_\varrho,g_{\varrho,\delta})> \frac12\varrho^{-3}\quad\text{and}\quad \scal(\overline{\Sigma}_\varrho,g_{\varrho,\delta})>\frac13|A(\overline\Sigma_\varrho,\delta+\dt^2)|^2,
    \end{align*}
    where $g_{\varrho,\delta}$ is the metric on $\overline\Sigma_\varrho\subset\bbR^{k+1}$ induced by the ambient Euclidean metric $\delta$.
    This also yields $\scal(\overline\Sigma_\varrho)\ge \frac16\varrho^{-3/2}|A(\overline\Sigma_\varrho,\delta+\dt^2)|$.
    Hence, 
    \[\scal(\overline\Sigma_\varrho,g_{\varrho,\delta})> \frac1{20} \left(\varrho^{-3} + |A(\overline\Sigma_\varrho,\delta+\dt^2)|^2 + \varrho^{-3/2}|A(\overline\Sigma_\varrho,\delta+\dt^2)|\right).\]
    Furthermore, we obtain from \eqref{eq:scalar-curvature-black-box} from \cref{lem:curvature-estimates-black-box} that there exist constants $C_0>0$, independent of $\kappa$, such that:
    \begin{align}\label{eq:curvature-black-box-with-constant}
        \begin{split}
            &\scal(\Sigma_\varrho,g_\varrho) - \scal(M, g_M)\circ\pr_M\\
            &\quad >{} \scal(\overline\Sigma_\varrho,g_{\varrho,\delta}) - C_0\Bigl(\varrho\cdot(1+|A(\overline\Sigma_\varrho,\delta+\dt^2)|)^2 + \|\overline\nu_\varrho^\intercal\|(1+|A(\overline\Sigma_\varrho,\delta+\dt^2)|)\Bigr),
        \end{split}
    \end{align}
    where $\pr_M\colon \Sigma_\varrho\subset M\times[0,1]\to M$ denotes the projection.
    We will omit \enquote{$\circ\pr_M$} for the remainder of this proof to ease the notation.
    Since $\|\overline\nu_\varrho^\intercal\|\le1$, and we thus obtain
    \begin{align*}
        \begin{split}
            &\scal(\Sigma_\varrho,g_\varrho) - \scal(M, g_M) \\
                &\qquad> \frac1{20}\left(\varrho^{-3} + |A(\overline\Sigma_\varrho,\delta+\dt^2)|^2 + \varrho^{-3/2}|A(\overline\Sigma_\varrho,\delta+\dt^2)|\right)\\
                &\qquad\qquad - C_0\Bigl(\varrho\cdot(1+|A(\overline\Sigma_\varrho,\delta+\dt^2)|)^2 + (1+|A(\overline\Sigma_\varrho,\delta+\dt^2)|)\Bigr)\\
                &\qquad\ge \varrho^{-3}\left(\frac1{20} - C_0\varrho^3\right) 
                + |A(\overline\Sigma_\varrho,\delta+\dt^2)|\left(\frac1{20}\varrho^{-3/2}-C_0\right)\\
                &\qquad\qquad+ |A(\overline\Sigma_\varrho,\delta+\dt^2)|^2\left(\frac1{20}-C_0\varrho\right)
        \end{split}
    \end{align*}
    Hence, for $c_0\coloneqq \tfrac{1}{40}$ and $\varrho_1\le \tfrac{1}{40C_0}$, we have
    \begin{align}\label{eq:scalar-curvature-estimate-region1}
        \scal(\Sigma_\varrho,g_\varrho) >\scal(M, g_M) + c_0\cdot \varrho^{-3}.
    \end{align}
    For the mean curvature, we first recall that by \eqref{eq:mean-curvature-black-box} from \cref{lem:curvature-estimates-black-box} there exists a constant $C_1$ such that
    \begin{align}\label{eq:mean-curvature-black-box-with-constant}
        H(\Sigma_\varrho,g+\dt^2) - H(\overline\Sigma_\varrho,\delta+\dt^2)\ge -C_1\left(\|\overline{\nu}_\varrho^\intercal\| + \varrho(1+|A(\overline\Sigma,\delta+\dt^2)|)\right).
    \end{align}
    Furthermore, we have 
    \[H(\overline\Sigma_\varrho,\delta+\dt^2) > \frac12 |A(\overline\Sigma_\varrho,\delta+\dt^2)|\quad\text{and}\quad H(\overline\Sigma_\varrho,\delta+\dt^2) > \frac12\varrho^{-3/2},\]
    by \cref{Prop: Euclidean Handle}.
    Therefore, $H(\overline\Sigma_\varrho,\delta+\dt^2)> \tfrac{1}{4}(|A(\overline\Sigma_\varrho,\delta+\dt^2)| + \varrho^{-3/2})$ and \eqref{eq:mean-curvature-black-box-with-constant} becomes:
    \begin{align*}
        H(\Sigma_\varrho,g+\dt^2) >{}& \frac14\left(\varrho^{-3/2}+ |A(\overline\Sigma_\varrho,\delta+\dt^2)|\right)\\
            & - C_1\left(1 + \varrho\cdot (1+ |A(\overline\Sigma_\varrho,\delta+\dt^2)|)\right)\\
            ={}& \varrho^{-3/2}\left(\frac14-C_1(\varrho^{3/2} + \varrho^{5/2})\right) + |A(\overline\Sigma_\varrho,\delta+\dt^2)|\left(\frac14-C_1\varrho\right)
    \end{align*}
    Hence, for $c_1\coloneqq\tfrac{1}{8}$ and $\varrho_1\le \tfrac{1}{16C_1}$, we get
    \begin{align}\label{eq:mean-curvature-region-1}
        H(\Sigma_\varrho,g+\dt^2)> c_1\cdot \varrho^{-3/2}
    \end{align}

    To close the first case of the proof for the scalar curvature estimate, let us consider the conformal transformation formula \eqref{eq:conformal-change-for-scalar-curvature} from \cref{lem:conformal-change}:
    \begin{align*}
        \begin{split}
            \scal(\Sigma_\varrho,g_{\varrho,\kappa}) ={}& u^{-\frac{n+2}{n-2}}\Bigl(
                \scal(\Sigma_\varrho, g_\varrho) - \frac{4(n-1)}{n-2}(1-\scpr{\nu_\varrho,\partial_t}^2)\partial^2_t u\\ 
                &\qquad\qquad+ \frac{4(n-1)}{n-2}H(\Sigma_\varrho,g+\dt^2)\scpr{\nu_\varrho,\partial_t} \partial_tu
            \Bigr)
        \end{split}
    \end{align*}

    If $t\ge2\varrho$, then $\Sigma_\varrho$ is given by the straight cylinder of radius $\varrho^{4}$, so $\scpr{\nu_\varrho,\partial_t}=0$.
    Furthermore, $u\le 2$ and 
    \[\partial^2_t u= \alpha^{-1}\tilde\kappa\eta'(\alpha^{-1}t-1) + \alpha^{-2}t\tilde\kappa\eta''(\alpha^{-1}t-1)\lesssim \alpha^{-2} \varrho_1\lesssim \varrho_1^{-1},\]
    by the choice of $\eta$ and $\kappa$.
    Therefore, we may assume $\partial^2_t u\le C_1\varrho_1$ and the above formula gives the following estimate
    \begin{align*}
        \scal(\Sigma_\varrho,g_{\varrho,\kappa}) >{}& 2^{-\frac{n+2}{n-2}}\Bigl(
            \scal(M, g_M) + \varrho^{-3}(c_1 - \frac{4(n-1)}{n-2}C_0\varrho_1^{2}) 
        \Bigr).
    \end{align*}
    We choose $\varrho_1$ such that $c_1-\tfrac{4(n-1)}{n-2}C_0\varrho_1^2\le c_1/2$ and $c_1\varrho_1^{-3/2}\ge 66\max|\scal(M,g_M)|$, so that we get $\scal(\Sigma_\varrho,g_{\varrho,\kappa})> |\scal(M,g_M)|$ as desired.
    
    \medskip

    On the other hand, if $t\le2\varrho\le\alpha$, then $u=1+t\tilde\kappa$, $\partial_t u=\tilde\kappa$ and $\partial_t^2u = 0$. As $H(\Sigma_\varrho,g+\dt^2)$ is positive, so is the entire final summand, and we get
    \begin{align*}
        \scal(\Sigma_\varrho,g_{\varrho,\kappa}) >{}& 2^{-\frac{n+2}{n-2}}\underbrace{\Bigl(
            \scal(M, g_M) + c_1\varrho^{-3}
        \Bigr)}_{\ge 32\max|\scal(M, g_M)|} \ge |\scal(M, g_M)|,
    \end{align*}
    which finishes the first case.

    \medskip

    Next, we investigate the region $r\in[\varrho^2,\sqrt{\varrho}]$, where \cref{Prop: Euclidean Handle} gives the following estimates:
    \begin{align*}
        \scal(\overline\Sigma_\varrho)> \frac13|A(\overline\Sigma_\varrho,\delta+\dt^2)|^2\quad\text{and}\quad \|\nu_\varrho^\intercal\|< \overline C\sqrt{\varrho}
    \end{align*}
    for $\overline C = 18k+432$ the constant from \cref{Prop: Euclidean Handle}.
    We set $C_2\coloneqq 4C_0\cdot \overline C$, and observe that \eqref{eq:curvature-black-box-with-constant} yields
    \begin{align*}
        &\scal(\Sigma_\varrho,g_\varrho) - \scal(M, g_M)\\
        &\qquad> \frac13|A(\overline\Sigma_\varrho,\delta+\dt^2)|^2 \\
        &\qquad\qquad-\frac{C_2}{4}\left(\varrho\cdot(1+|A(\overline\Sigma_\varrho,\delta+\dt^2)|)^2 - \varrho^{1/2}\cdot(1+|A(\overline\Sigma_\varrho,\delta+\dt^2)|)\right)\\
        &\qquad\ge  -\frac{C_2}{4}\varrho^{1/2}-\frac{C_2}{2}\varrho^{1/2}|A(\overline\Sigma_\varrho,\delta+\dt^2)| + \left(\frac13-\frac{C_2}{4}\varrho\right)|A(\overline\Sigma_\varrho,\delta+\dt^2)|^2
    \end{align*}
    If $|A(\overline\Sigma_\varrho,\delta+\dt^2)|\le 1$, then this becomes $\scal(\Sigma_\varrho,g_\varrho) -\scal(M, g_M)>-C_2\varrho^{1/2}$.
    On the other hand, if $|A(\overline\Sigma_\varrho,\delta+\dt^2)|\ge 1$, then $|A(\overline\Sigma_\varrho,\delta+\dt^2)|\le |A(\overline\Sigma_\varrho,\delta+\dt^2)|^2$ and we obtain 
    \begin{align*}
        \begin{split}
            &\scal(\Sigma_\varrho,g_\varrho) - \scal(M, g_M) >  -\frac{C_2}{4}\varrho^{1/2} + \left(\frac13-\frac{C_2}{2}\varrho^{1/2}-\frac{C_2}{4}\varrho\right)|A(\overline\Sigma_\varrho,\delta+\dt^2)|^2
        \end{split}
    \end{align*}
    and thus, for $\varrho_1\le (\frac{1}{3C_2})^2$, we have 
    \begin{align}\label{eq:scalar-curvature-estimate-region2}
        \scal(\Sigma_\varrho,g_\varrho) > \scal(M, g_M) - C_2\varrho^{1/2}
    \end{align}
    Concerning the mean curvature, \cref{Prop: Euclidean Handle} again yields
    \[H(\overline\Sigma_\varrho,\delta+\dt^2)>\frac12|A(\overline\Sigma_\varrho,\delta+\dt^2)|,\]
    and we can hence estimate
    \begin{align*}
        \begin{split}
            H(\Sigma_\varrho,g+\dt^2) >{}& \frac12|A(\overline\Sigma_\varrho,\delta+\dt^2)| - C_1\cdot\overline C\left(\varrho^{1/2} + \varrho\cdot (1+ |A(\overline\Sigma_\varrho,\delta+\dt^2)|)\right)\\
                \ge{}& -C_3\varrho^{1/2} + |A(\overline\Sigma_\varrho,\delta+\dt^2)|\left(\frac{1}{2} - C_3\varrho\right)
        \end{split}
    \end{align*}
    where $C_3\coloneqq 2C_1\cdot\overline C$.
    If $\varrho_1\le \tfrac{1}{4C_3}$, we get
    \begin{align}\label{eq:mean-curvature-in-the-middle}
        H(\Sigma_\varrho,g+\dt^2)>-C_3\varrho^{1/2} 
    \end{align}
    By employing the conformal change formula from \cref{lem:conformal-change}, we can use $u=1+t\tilde\kappa$ to obtain
    \begin{align*}
        \scal(\Sigma_\varrho,g_{\varrho,\kappa}) ={}& u^{-\frac{n+2}{n-2}}\Bigl(\scal(\Sigma_\varrho, g_\varrho) + \frac{4(n-1)}{n-2}H(\Sigma_\varrho,g+\dt^2)\scpr{\nu_\varrho,\partial_t} \tilde\kappa\Bigr)\\
            >{}&u^{-\frac{n+2}{n-2}}\Bigl(\scal(M, g_M) - C_3\varrho^{1/2}\Bigr)\\
            \ge{}&u^{-\frac{n+2}{n-2}}\left(\scal(M, g_M) - \frac12\varrho^{1/4}\right),
    \end{align*}
    if $\varrho_1^{1/4}C_3\le \tfrac{1}{2}$.
    Since $u=1+t\tilde\kappa\le1+\kappa\le 1+\varrho_1$, we have $1\ge u^{-1}\ge u^{-(n+2)/(n-2)} \ge u^{{-5}}\ge 1-5t\tilde\kappa\ge 1-5\varrho_1$.\footnote{This follows from the inequality $(1+x)^{-n} \ge 1-nx$ for  $n\in\bbN$.}
    If $\scal(M, g_M) -\frac12\varrho^{1/4}\le 0$, then 
    \[\scal(\Sigma_\varrho,g_{\varrho,\kappa}) >\scal(M, g_M) -\frac12\varrho^{1/4} \ge\scal(M, g_M) -\frac12\varrho_1^{1/4},\]
    so it suffices to consider the case that $\scal(M, g_M) -\frac12\varrho^{1/4}\ge 0$, where we can argue similarly:
    \begin{align*}
        \scal(\Sigma_\varrho,g_{\varrho,\kappa})
            >{}&(1-5\varrho_1)\left(\scal(M, g_M) -\frac12\varrho^{1/4}\right)\\
            >{}&\scal(M, g_M) - \varrho_1^{1/4}\Bigl(
                5\cdot\scal(M, g_M)\varrho_1^{3/4} + \frac12
            \Bigr),
    \end{align*}
    and, if $\varrho_1$ is chosen such that $5\max\scal(M, g_M)\varrho_1^{3/4} < \frac12$, we obtain 
    \[\scal(\Sigma_\varrho,g_{\varrho,\kappa})>\scal(M, g_M)-\varrho_1^{1/4}.\]

    \medskip

    Finally, we consider $r\in[\sqrt{\varrho},2\sqrt\varrho]$.
    By \cref{Prop: Euclidean Handle}, we have
    \begin{align*}
        \|\nu_\varrho^\intercal\|\lesssim \varrho^{1/2},\qquad|A(\overline\Sigma_\varrho,\delta+\dt^2)|\lesssim\varrho^{3/8}\quad\text{and}\quad \scal(\overline\Sigma_\varrho,g_\varrho)\gtrsim-\varrho^{3/4}.
    \end{align*}
    By consulting \eqref{eq:scalar-curvature-black-box} from \cref{lem:curvature-estimates-black-box} one final time, we deduce that 
    \begin{align}\label{eq:scalar-curvature-estimate-region3}
        \begin{split}
            &\scal(\Sigma_\varrho,g_\varrho)-\scal(M, g_M)\\
            &\qquad\gtrsim -\varrho^{3/4} - \varrho\cdot(1+\varrho^{3/8})^2 - \varrho^{1/2}\cdot(1+\varrho^{3/8}) \gtrsim -\varrho^{1/2}
        \end{split}
    \end{align}
    Furthermore, using $H(\overline\Sigma_\varrho,\delta+\dt^2)\ge-n|A(\overline\Sigma_\varrho,\delta+\dt^2)|\gtrsim\varrho^{3/8}$ we obtain
    \begin{align}\label{eq:mean-curvature-in-the-end}
        \begin{split}
            H(\Sigma_\varrho,g+\dt^2) \gtrsim{}& H(\overline\Sigma_\varrho,\delta+\dt^2)- \|\nu_\varrho^\intercal\|-\varrho(1+ |A(\overline\Sigma_\varrho,\delta+\dt^2)|)\\
                \gtrsim{}& -\varrho^{3/8} -  \varrho^{1/2}-\varrho-\varrho^{11/8}\gtrsim -\varrho^{3/8}
        \end{split}
    \end{align}
    Thus there exists a constant $C_4>0$ such that 
    \[\scal(\Sigma_\varrho,g_\varrho)>\scal(M, g_M)-C_4\varrho^{1/2}\quad\text{and}\quad H(\Sigma_\varrho,g+\dt^2)>-C_4\varrho^{3/8}.\]
    An application of \cref{lem:conformal-change} together with the observation that $\partial_tu = \tilde\kappa\le 1$, $\partial_t^2 u=0$ and $\tfrac{4(n-1)}{(n-2)}\le8$ now gives:
    \begin{align*}
        \scal(\Sigma_\varrho,g_{\varrho,\kappa})>{}& u^{-\frac{n+2}{n-2}}\left(\scal(M, g_M) - \varrho^{1/4}C_4\cdot (\varrho^{1/4} + 8\varrho^{1/8})\right)\\
        \ge{}& u^{-\frac{n+2}{n-2}}\left(\scal(M, g_M) - \frac12\varrho^{1/4}\right),
    \end{align*}
    provided $\varrho_1$ satisfies $C_4\cdot (\varrho^{1/4} + 8\varrho^{1/8})<\tfrac{1}{2}$.
    Using $\varrho^{3/8}\le\varrho^{1/4}$, the same argument as in the second case, below \eqref{eq:mean-curvature-in-the-middle} shows that 
    \[\scal(\Sigma_\varrho,g_{\varrho,\kappa})>\scal(M, g_M)-\varrho_1^{1/4}.\]

    \medskip

    It remains to show positivity of the mean curvature with respect to the conformally transformed metric after passing to a smaller radius bound $\varrho_0$ which may depend on $\kappa$.
    Recall the conformal transformation formula:
    \[H(\Sigma_\varrho,g_{\kappa}) = u^{-\frac{2}{n-2}} \left(H(\Sigma_\varrho,g+\dt^2) + \frac{2n}{n-2}\scpr{\nu_\varrho,\partial_t}u^{-1}\partial_tu\right)\]
    We will again distinguish radial cases.

    \medskip

    If $r\in[\varrho^4,\varrho^2]$, we have already seen that 
    \[H(\Sigma_\varrho, g+\dt^2) > c_1\varrho^{-3/2},\]
    see \eqref{eq:mean-curvature-region-1}
    Since $\partial_t u = \tilde\kappa\eta(\alpha^{-1}t-1) + \alpha^{-1}t\tilde\kappa\eta'(\alpha^{-1}t-1) \ge -3\alpha^{-1}$, $u^{-1}\le1$, $\tilde\kappa\le1$ and $\scpr{\nu_\varrho,\partial_t}\le1$, we thus obtain for $C_5\coloneqq\tfrac{6n\alpha^{-1}}{(n-2)}$:
    \begin{align}\label{eq:mean-curvature-after-trafo-region-1}
        H(\Sigma_\varrho,g_{\kappa}) > u^{-\frac{2}{n-2}} \left(c_1\varrho^{-3/2} - C_5\right).
    \end{align}
    If $\varrho_0$ is chosen such that $c_1\varrho_0^{-3/2}-C_5\ge2^{2/(n-2)}\ge u^{2/(n-2)}\kappa$, then 
    \[H(\Sigma_\varrho,g_{\kappa})>\kappa>(1-\alpha)\kappa-\varrho^{1/4}\]
    as desired.

    \medskip

    The cases $r\in[\varrho^2,\sqrt\varrho]$ and $r\in[\sqrt\varrho,2\sqrt\varrho]$ can be treated simultaneously.
    We first observe, that \eqref{eq:mean-curvature-in-the-middle} and \eqref{eq:mean-curvature-in-the-end} both give
    \begin{align}\label{eq:sth1}
        H(\Sigma_\varrho, g+\dt^2) > -\varrho^{1/4},
    \end{align}
    provided $\varrho_0$ is small enough.
    Furthermore, combining $\scpr{\nu_\varrho,\partial_t}_{g+\dt^2}-\scpr{\overline\nu_\varrho,\partial_t}_{\delta+\dt^2}\gtrsim-\varrho$ from \cref{lem:curvature-estimates-black-box} and $\scpr{\overline\nu_\varrho,\partial_t}_{\delta+\dt^2}-1\gtrsim-\varrho^{1/2}$ from \cref{Prop: Euclidean Handle}, we find a constant $C_6>0$ such that we have
    \begin{align}\label{eq:sth2}
        \scpr{\nu_\varrho,\partial_t}_{g+\dt^2} \ge 1  - C_6\varrho^{1/2}\ge 1-\frac12\varrho^{1/4}
    \end{align}
    if $\varrho_0\le C_6^{-4}$.
    From \eqref{eq:sth1} and \eqref{eq:sth2}, together with $\partial_tu = \tilde\kappa=\tfrac{n-2}{2n}\kappa$ and $u^{-1}\ge 1-\varrho_1$, we obtain
    \begin{align*}
        H(\Sigma_\varrho,g+\dt^2) + \frac{2n}{n-2}\scpr{\nu_\varrho,\partial_t}u^{-1}\partial_tu &>-\frac12\varrho^{1/4}+(1-\varrho^{1/4})(1-\varrho_1)\kappa\\
            &=(1-\varrho_1)\kappa - \frac12\varrho^{1/4} - \varrho^{1/4}(1-\varrho_1)\kappa\\
            &\ge (1-\varrho_1)\kappa -\varrho^{1/4}
    \end{align*}
    Note, that $u^{-\frac{2}{n-2}}\ge 1-2\varrho_1$, and we choose $\varrho_0\le \tfrac{1}{2}(1-\varrho_1)^4\kappa^4$ small enough, so that $(1-\varrho_1)\kappa -\varrho^{1/4}>0$ as well. The transformation formula \eqref{eq:conformal-change-for-mean-curvature} then gives
    \begin{align}\label{eq:mean-curvature-after-trafo-region-2}
        \begin{split}
            H(\Sigma_\varrho,g_{\kappa}) 
            >{}& (1-2\varrho_1) \Bigl((1-\varrho_1)\kappa -\varrho^{1/4}\Bigr)\\
            ={}& (1-3\varrho_1)\kappa - \varrho^{1/4}\ge (1-\alpha)\kappa - \varrho^{1/4},
        \end{split}
    \end{align}
    which finishes the proof.
\end{proof}

Let us now dive into the proof of \cref{lem:curvature-estimates-black-box}. 

\begin{proof}[Proof of \cref{lem:curvature-estimates-black-box}]
    By the Gauss equation we have
    \begin{align}\label{eq:scalar-curvature-first-formula}
        \begin{split}
            &\scal(\Sigma_\varrho,g_{\varrho}) - \underbrace{\scal(M\times[0,1],g+\dt^2)}_{=\scal(M, g_M)\circ\pr_M}\\
                &\qquad={}H(\Sigma_\varrho,g+\dt^2)^2-|A(\Sigma_\varrho,g+\dt^2)|^2
                -2\Ric(M\times[0,1],g+\dt^2)(\nu_\varrho,\nu_\varrho),
        \end{split}
    \end{align}
    where $\nu_\varrho$ is the upwards unit normal vector of $\Sigma_\varrho$ with respect to the ambient metric $g+\dt^2$ and $\pr_M\colon M\times[0,1]\to M$ is the projection.
    Again, we will omit \enquote{$\circ\pr_M$} from now on.
    In order to estimate the summands from the right-hand side of  \eqref{eq:scalar-curvature-first-formula}, consider the symmetric $2$-tensor $\gamma$ on $\mathcal{S}$ defined as the first order Taylor expansion of $g$ in normal direction, that is  
    \begin{equation}\label{eq: definition gamma}
        \gamma_{\alpha\beta}=g_{\alpha\beta}|_\mathcal{S}+\sum_{i=1}^kx^ig_{\alpha\beta,i}|_\mathcal{S}
    \end{equation}
    which we extend canonically to $B_R^k$.
    Since the tensor $\gamma$ precisely describes the $1$-jet of the metric $g$ along $\mathcal{S}$, there exists an $\eps>0$ such that the restriction of $\gamma$ to the disk-bundle of radius $\eps$ is positive definite and hence a Riemannian metric. 
    If $r$ denotes the radial component of the fiber of the tubular neighbourhood $D(\mathcal{S})$ in $M$, we have
    \[|g-\gamma|_{\gamma} \le C\cdot r^2,\]
    since $g_{\alpha\beta}(x)=g_{\alpha\beta }(p)+x^i g_{\alpha\beta,i}(p)+O(r^2)$ by the Taylor remainder estimate.
    The advantage of working with the background metric $\gamma+\dt^2$ is that the expressions for $H$ and $|A|$ can be readily compared to the corresponding expressions for Euclidean handles as described in \cref{sec:euclidean-handles}.

    \medskip

    The proof now consists of investigating the difference in second fundamental form, mean curvature and normal vector when passing from the background metric $g+\dt^2$ to $\gamma+\dt^2$ and afterwards relating those terms for $\gamma+\dt^2$ to the corresponding ones for the Euclidean background metric $\delta+\dt^2$ in the fibers.
    Let us abbreviate 
    \begin{align*}
        A_g\coloneqq{}& A(\Sigma_\varrho, g+\dt^2) && A_\gamma \coloneqq A(\Sigma_\varrho, \gamma+\dt^2) && A_\delta \coloneqq A(\overline\Sigma_\varrho, \delta+\dt^2)\\
        H_g\coloneqq{}& H(\Sigma_\varrho, g+\dt^2) && H_\gamma \coloneqq H(\Sigma_\varrho, \gamma+\dt^2) && H_\delta \coloneqq H(\overline\Sigma_\varrho, \delta+\dt^2)
    \end{align*}
    Furthermore, we denote by $\nu_g$, $\nu_\gamma$ and $\nu_\delta$ the respective upward unit normal vector fields of $\Sigma_\varrho$ or $\overline\Sigma_\varrho$.
    We have the following comparison formulae
    \begin{align*}
        \|\nu_g - \nu_\gamma\|_{\gamma+\dt^2} \lesssim{}& r^2,\qquad &|A_g-A_\gamma|_{\gamma+\dt^2}\lesssim{}& (1+|A_\gamma|_{\gamma+\dt^2})\cdot r^2,
    \end{align*}
    where the latter identity follows from $ A_g(e_i,e_j)=-\langle \nabla^{g+dt^2}_{e_i}e_j,\nu_g\rangle_{ g}$ and 
    $|\Gamma^{m}_{ij}|_g\lesssim r^2$ for the Christoffel symbols $\Gamma_{ij}^m=\tfrac12g^{ml}(\nabla^\gamma_i g_{lj}+\nabla^\gamma_jg_{li}+\nabla^\gamma_lg_{ij})$.

    \medskip

    Since the difference of $\Sigma_\varrho$ and $M\times\{0\}$ is supported on $\{r\le 2\sqrt{\varrho}\}$, we thus obtain 
    \begin{align*}
        |H_g-H_\gamma| ={} |\tr(A_g)-\tr(A_\gamma)|\lesssim{} &\varrho\cdot(1+|A_\gamma|_{\gamma+\dt^2})\\
        \bigl||A_g|_{g+\dt^2}-|A_\gamma|_{\gamma+\dt^2}\bigr|\lesssim{}& \varrho\cdot(1+|A_\gamma|_{\gamma+\dt^2})
    \end{align*}
    and therefore 
    \begin{align*}
        \bigl||A_g|_{g+\dt^2}^2 &- |A_\gamma|_{\gamma+\dt^2}^2\bigr| \\
            \le{}& \bigl||A_g|_{g+\dt^2} - |A_\gamma|_{\gamma+\dt^2}\bigr|^2 + 2|A_\gamma|_{\gamma+\dt^2}\bigl||A_g|_{g+\dt^2} - |A_\gamma|_{\gamma+\dt^2}\bigr|\\
            \lesssim{}& \varrho^2\cdot(1+|A_\gamma|_{\gamma+\dt^2})^2 + 2|A_\gamma|_{\gamma+\dt^2}(1+|A_\gamma|_{\gamma+\dt^2})\cdot \varrho\\
            \lesssim{}& \varrho\cdot  (1+|A_\gamma|_{\gamma+\dt^2})\cdot \left(\varrho(1+|A_\gamma|_{\gamma+\dt^2}) + 2|A_\gamma|_{\gamma+\dt^2}\right)\\
            \lesssim{}& \varrho\cdot (1+|A_\gamma|_{\gamma+\dt^2})^2\\
        |H_g^2 - H_\gamma^2|
            \le{}& |H_g-H_\gamma|^2 + 2|H_\gamma|\cdot |H_g-H_\gamma|\\
            \lesssim{}& \varrho^2\cdot(1+|A_\gamma|_{\gamma+\dt^2})^2 + n|A_\gamma|_{\gamma+\dt^2}(1+|A_\gamma|_{\gamma+\dt^2})\cdot \varrho\\
            \lesssim{}& \varrho\cdot(1+|A_\gamma|_{\gamma+\dt^2})\bigl((1+|A_\gamma|_{\gamma+\dt^2})\cdot \varrho + n|A_\gamma|_{\gamma+\dt^2}\bigr)\\
            \lesssim{}& \varrho\cdot(1+|A_\gamma|_{\gamma+\dt^2})^2.
    \end{align*}
    Putting these together, we obtain
    \begin{align*}
        \bigl|(H_g^2-|A_g|_{g+\dt^2}^2)-H_\gamma^2-|A_\gamma|_{\gamma+\dt^2}^2\bigr| &\lesssim \varrho\cdot (1+|A_\gamma|_{\gamma+\dt^2})^2
    \end{align*}
    If we write $\nu_g^\intercal \coloneqq \nu_g - \scpr{\nu_g,\partial_t}_{g+\dt^2}\partial_t$ and $\nu_\gamma^\intercal \coloneqq \nu_\gamma - \scpr{\nu_\gamma,\partial_t}_{\gamma+\dt^2}\partial_t$, then we observe 
    \begin{align*}
        &\bigl|\scpr{\nu_g,\partial_t}_{g+\dt^2} - \scpr{\nu_\gamma,\partial_t}_{\gamma+\dt^2}\bigr|
            \\&\qquad\le{}\underbrace{\bigl|\scpr{\nu_g,\partial_t}_{g+\dt^2} - \scpr{\nu_g,\partial_t}_{\gamma+\dt^2}\bigr|}_{\lesssim \|g-\gamma\|_\gamma\lesssim\varrho} + \underbrace{\bigl|\scpr{\nu_g,\partial_t}_{\gamma+\dt^2} - \scpr{\nu_\gamma,\partial_t}_{\gamma+\dt^2}\bigr|}_{\lesssim \|\nu_g-\nu_\gamma\|_{\gamma+\dt^2}\lesssim\varrho} \lesssim{} \varrho\\
        &\left|\|\nu_g^\intercal\|_g - \|\nu_\gamma^\intercal\|_{\gamma}\right|
            \\&\qquad\le{} \bigl|\|\nu_g^\intercal\|_{g+\dt^2} - \|\nu_\gamma^\intercal\|_{g+\dt^2}\bigr| + \bigl|\|\nu_\gamma^\intercal\|_{g+\dt^2} - \|\nu_\gamma^\intercal\|_{\gamma+\dt^2}\bigr|\\
            &\qquad\le{} \|\nu_g^\intercal-\nu_\gamma^\intercal\|_{g+\dt^2} + \bigl|\|\nu_\gamma^\intercal\|_{g+\dt^2} - \|\nu_\gamma^\intercal\|_{\gamma+\dt^2}\bigr|\\
            &\qquad\le{} \|\nu_g-\nu_\gamma\|_{g+\dt^2} + \underbrace{\bigl|\scpr{\nu_g,\partial_t}_{g+\dt^2} - \scpr{\nu_\gamma,\partial_t}_{\gamma+\dt^2}\bigr|}_{\lesssim\varrho} + \underbrace{\bigl|\|\nu_\gamma^\intercal\|_{g+\dt^2} - \|\nu_\gamma^\intercal\|_{\gamma+\dt^2}\bigr|}_{\le\|g-\gamma\|_{\gamma+\dt^2}\lesssim{}\varrho}\\
            &\qquad\lesssim{}\varrho
    \end{align*}
    Regarding the final summand in \eqref{eq:scalar-curvature-first-formula}, we note that $\Ric(M\times[0,1],g+\dt^2)(\partial_t,\partial_t)=0$ and therefore, we obtain
    \begin{align*}
        \Ric(M\times[0,1],g+\dt^2)(\nu_g,\nu_g) ={}& \|\nu_g^\intercal\|_g^2\cdot \Ric(M, g_M)\left(\frac{\nu_g^\intercal}{\|\nu_g^\intercal\|_g},\frac{\nu_g^\intercal}{\|\nu_g^\intercal\|_g}\right)\\
         \lesssim{}&\|\nu_g^\intercal\|_g^2 \lesssim \|\nu_\gamma^\intercal\|_\gamma+\varrho
    \end{align*}
    We note that the constant in the above inequality depends on the maximum of the Ricci-curvature of $g$ and hence, in particular on $g$, but is independent of $\varrho$.

    \medskip

    As mentioned above, the advantage of working with the background metric $\gamma+\dt^2$ is that the expressions for $H$ and $|A|$ can be readily compared to the corresponding expressions for Euclidean handles from \cref{sec:euclidean-handles}.
    To make this precise, we first observe that the projection
    \[
    \pi:B_R(\mathcal{S})\times[0,1]\longrightarrow B_{R}^k(0)\times[0,1]\subseteq\mathbb R^k\times[0,1]
    \]
    satisfies $D\pi(\nu_\gamma)=\nu_\delta$.
    Furthermore, the restriction of $\pi$ to the fibers is a radial isometry, so 
    \[
    \|\nu_\gamma^\intercal\|_{\gamma+\dt^2} = \|\nu_\delta^\intercal\|_{\delta+\dt^2}\quad \text{and}\quad \scpr{\nu_\gamma,\partial_t}_{\gamma+\dt^2} = \scpr{\nu_\delta,\partial_t}_{\delta+\dt^2}
    \]
    In particular, we immediately obtain $|\scpr{\nu_g,\partial_t}_{g+\dt^2} - \scpr{\nu_\gamma,\partial_t}_{\gamma+\dt^2}|\lesssim{}\varrho$, proving the final estimate from \cref{lem:curvature-estimates-black-box}.

    \medskip

    If we fix a point $x\in\Sigma_\varrho$ and choose an adapted $\gamma$-orthonormal frame $\{e_i,e_a\}$ of $\Sigma_\varrho$ at $x$, where $\{e_i\}$ are normal to $\mathcal{S}$ and $\{e_a\}$ tangent to $\mathcal{S}$, we see that $\{D\pi(e_i)\}$ form normal coordinates at $\pi(x)\in\overline\Sigma_\varrho$.
    By definition of the second fundamental form,
    \begin{align*}
        A_\gamma(e_i,e_j) ={}& -\langle\nabla^{\gamma+\dt^2}_{e_i}e_j,\nu_\gamma\rangle =  -\langle\nabla^{\delta+\dt^2}_{e_i}e_j,\nu_\delta\rangle = A_\delta(e_i,e_j).
    \end{align*}
    Next, we note that $\scpr{\nabla^{\gamma+\dt^2}_{e_a}e_j,e_i}$ and $\scpr{\nabla^{\gamma+\dt^2}_{e_a}e_b,e_i}$ are components of the second fundamental form of the fibers of the tubular neighbourhood, hence they are independent of $\varrho$, and we get $\scpr{\nabla_{e_a}e_j,e_i},\scpr{\nabla_{e_a}e_b,e_i}\lesssim1$.
    Moreover, both $\nabla_{e_a}e_j$ and $\nabla_{e_a}e_b$ are orthogonal to $\partial_t$, so we obtain
    \begin{align*}
        A_\gamma(e_a,e_j) ={}& -\langle\nabla_{e_a}e_j,\nu_\gamma\rangle = -\langle\nabla_{e_a}e_j,\nu_\gamma^\intercal\rangle\\
            ={}&-\sum_{i}\langle\nabla_{e_a}e_j,e_i\rangle \langle \nu^\intercal_\gamma,e_i\rangle\lesssim \|\nu^\intercal_\gamma\| = \|\nu^\intercal_\delta\|
    \end{align*}
    and $A_\gamma(e_a,e_b)\lesssim\|\nu_\delta^\intercal\|$, analogously.
    Thus, we have obtained
    \begin{align*}
        |A_\gamma(e_i,e_a)|\lesssim \|\nu_\delta^\intercal\|\quad
        |A_\gamma(e_a,e_b)|\lesssim \|\nu_\delta^\intercal\|\quad\text{and}\quad
        A_\gamma(e_i,e_j)=A_\delta(e_i,e_j).
    \end{align*}
    Using these, we immediately get 
    \begin{align*}
        \bigl||A_\gamma|_{\gamma+\dt^2}^2-|A_\delta|_{\delta+\dt^2}^2\bigr|\lesssim{}& \|\nu_\delta^\intercal\|^2\le \|\nu_\delta^\intercal\|\\
        |H_\gamma-H_\delta| \lesssim{}&  \|\nu_\delta^\intercal\|
    \end{align*}
    which, together with $|H_g-H_\gamma| \lesssim{} \varrho\cdot (1+|A_\gamma|_{\gamma+\dt^2})$ combines to
    \begin{align*}
        |H_g- H_\delta| \lesssim  \|\nu_\varrho^\intercal\|+\varrho\cdot (1+|A_\gamma|_{\gamma+\dt^2}),
    \end{align*}
    proving the desired mean curvature estimate.
    
    \medskip
    
    It remains to verify the scalar curvature estimate.
    We have
    \begin{align*}
        |H_\gamma^2 - H_\delta^2|
            ={}& |2H_\delta\cdot(H_\gamma - H_\delta) + (H_\gamma - H_\delta)^2|\\
            \lesssim{}& |A_\delta|_{\delta+\dt^2}\cdot \|\nu_\delta^\intercal\| + \|\nu_\delta^\intercal\| ^2\\
            \lesssim{}& \|\nu_\delta^\intercal\|(1+|A_\delta|_{\delta+\dt^2}) 
    \end{align*}
    where we used ${|H_\delta|\le n|A_\delta|}\lesssim|A_\delta|$.
    Combining these, we get 
    \begin{align*}
        (H_\gamma^2-|A_\gamma|_{\gamma+\dt^2}^2) - (H_\delta^2-|A_\delta|_{\delta+\dt^2}^2) \gtrsim{}& -\|\nu_\delta^\intercal\|(1+|A_\delta|_{\delta+\dt^2}).
    \end{align*}
    Furthermore, there exists a constant $C$ such that
    \begin{align*}
        |A_\gamma|_{\gamma+\dt^2} \le{}& \sqrt{|A_\delta|_{\delta+\dt^2}^2 + C \|\nu_\delta^\intercal\|^2}\\
            \le{}& \sqrt{|A_\delta|_{\delta+\dt^2}^2 + 2\sqrt C\|\nu_\delta^\intercal\|\cdot |A_\delta|_{\delta+\dt^2} + C \|\nu_\delta^\intercal\|^2}\\
            \le{}& |A_\delta|_{\delta+\dt^2} + C\|\nu_\delta ^\intercal\|
    \end{align*}
    and therefore $|A_\gamma|_{\gamma+\dt^2} -  |A_\delta|_{\delta+\dt^2}\lesssim \|\nu_\varrho^\intercal\|\le 1$ and $1+|A_\gamma|_{\gamma+\dt^2}\lesssim 1+|A_\delta|_{\delta+\dt^2}$.
    Note that $H_\delta^2-|A_\delta|_{\delta+\dt^2}^2$ is precisely the scalar curvature of $(\Sigma_\varrho,g_{\varrho,\delta})$, where $g_{\varrho,\delta}$ is induced by the Euclidean metric $\delta+\dt^2$.
    Thus, we obtain
    \begin{align*}
        &(H_g^2-|A_g|_{g+\dt^2}^2) - \scal(\overline\Sigma_\varrho,g_{\varrho,\delta}) \\
        &\quad= (H_g^2-|A_g|_{g+\dt^2}^2)-(H_\gamma^2-|A_\gamma|_{\gamma+\dt^2}^2) + (H_\gamma^2-|A_\gamma|_{\gamma+\dt^2}^2)- (H_\delta^2-|A_\delta|_{\delta+\dt^2}^2)\\
        &\quad\gtrsim -\varrho\cdot(1+|A_\gamma|_{\gamma+\dt^2})^2 - \|\nu_\delta^\intercal\|\cdot(1+|A_\gamma|_{\gamma+\dt^2})\\
        &\quad\gtrsim -\varrho\cdot(1+|A_\delta|_{\delta+\dt^2})^2 - \|\nu_\delta^\intercal\|\cdot(1+|A_\delta|_{\delta+\dt^2})
    \end{align*}    
    Putting all the pieces together, we therefore obtain:
    \begin{align*}
        &\scal(\Sigma_\varrho,g_\varrho) -\scal(M, g_M) - \scal(\overline\Sigma_\varrho,g_{\varrho,\delta})\\
        &\qquad= H^2_g-|A_g|^2 - \scal(\overline\Sigma_\varrho,g_{\varrho,\delta}) -2\underbrace{\Ric(M\times[0,1],g+\dt^2)(\nu_\varrho,\nu_\varrho)}_{\lesssim\|\nu_\delta^\intercal\|+\varrho} \\
        &\qquad\gtrsim -\varrho\cdot(1+|A_\delta|_{\delta+\dt^2})^2 - \|\nu_\delta^\intercal\|\cdot(1+|A_\delta|_{\delta+\dt^2})  -\|\nu_\delta^\intercal\|+\varrho\\
        &\qquad\gtrsim -\varrho\cdot(1+|A_\delta|_{\delta+\dt^2})^2 - \|\nu_\delta^\intercal\|\cdot(1+|A_\delta|_{\delta+\dt^2})
    \end{align*}
    which finishes the proof.
\end{proof}

Next, we compare $\Sigma_\varrho$ for different values of $\varrho$.

\begin{proposition}\label{prop:monotonicity-of-arbitrary-handles}
    For every $\kappa\in (0,\varrho_1]$ there exists a smooth path of diffeomorphisms $\Phi_{\varrho}:\Sigma_{\varrho_1}\to\Sigma_\varrho$, $\varrho_0 \leq \varrho \leq \varrho_1$,  such that $\Phi_{\varrho_1}=\id_{\Sigma_{\varrho_1}}$ and for $g_{\kappa,\varrho}$ the metric on $\Sigma_\varrho$ induced by $g_{\kappa}$ (see \cref{theorem:handle-properties}) the assignment
    \[\varrho\mapsto e^\varrho\cdot\Phi_{\varrho}^\ast(g_{\varrho,\kappa})\] 
    is strictly monotonically increasing.
   
\end{proposition}

\begin{proof}
Recall from \eqref{reparametrization family} the family of reparametrizations
\begin{align*}
        \phi_{\varrho_1,\varrho}(\tau)=\tau-q_{\varrho_1,\varrho}\int_0^\tau\beta\left(s-\frac12\right)ds
    \end{align*}
    and their lift to $\overline\Sigma_\varrho $ and $\Sigma_\varrho$ which, by slight abuse of notation, we both denote by $ \Phi_{\varrho_1,\varrho}$.
First, we show that the proof in \cref{lem:monotonicity-of-Euclidean-handles} still goes through in case the Euclidean background metric $\delta+\dt^2$ is replaced with $g+dt^2$ restricted to the fiber of the tubular neighbourhood $B_R(\calS)\times[0,1]$.
We denote this restriction of $g_M+dt^2$ by $\hat g$.
As in the proof of \cref{lem:monotonicity-of-Euclidean-handles} we have
\begin{align*}
   \Phi_{\varrho_1,\varrho}^\ast g_{\overline\Sigma_\varrho,\hat g}(x,\tau)=
   \hat g(x,\phi_{\varrho_1,\varrho}(\tau))\bigl(\partial_\tau\phi_{\varrho_1,\varrho}(\tau),\partial_\tau\phi_{\varrho_1,\varrho}(\tau)\bigr)\dtau^2 + g_{S^{k-1}_{r_\varrho(\phi_{\varrho_1,\varrho}(\tau))}, \hat g}.
\end{align*}
Here $g_{S^{k-1}_r, \hat g} $ is the metric induced by $\hat g$ on a sphere of radius $r$ where the radius is measured with respect to $\delta$.
For $\varrho_1$ sufficiently small, the metric on the spheres $S^k_r$ is increasing as long as the radius $r\le \varrho_1$ increases.
Consequently, 
\begin{align*}
    \partial_\varrho g_{S^{k-1}_{r_\varrho(\phi_{\varrho_1,\varrho}(\tau))}, \hat g} \ge0.
\end{align*}
Moreover, we have $|\partial_\varrho\phi_{\varrho_1,\varrho}(\tau)|\lesssim \varrho^{1/4}\ll 1$.
Since the derivatives of $\hat g$ are bounded, this yields as in the proof of \cref{lem:monotonicity-of-Euclidean-handles}
\begin{align*}
    &\partial_\varrho \Bigl( \hat g(x,\phi_{\varrho_1,\varrho}(\tau))\bigl(\partial_\tau\phi_{\varrho_1,\varrho}(\tau),\partial_\tau\phi_{\varrho_1,\varrho}(\tau)\bigr)\dtau^2\Bigr)
    \\&\qquad\ge
    -\frac12  \hat g(x,\phi_{\varrho_1,\varrho}(\tau))(\partial_\tau\phi_{\varrho_1,\varrho}(\tau),\partial_\tau\phi_{\varrho_1,\varrho}(\tau))\dtau^2.
\end{align*}
Consequently,
\begin{align*}
  \partial_\varrho\left[e^{\varrho/2} \left( \hat g(x,\phi_{\varrho_1,\varrho}(\tau))(\partial_\tau\phi_{\varrho_1,\varrho}(\tau),\partial_\tau\phi_{\varrho_1,\varrho}(\tau))\dtau^2 + g_{S^{k-1}_{r_\varrho(\phi_{\varrho_1,\varrho}(\tau))}, \hat g} \right)\right]\ge0
\end{align*}
and \cref{lem:monotonicity-of-Euclidean-handles} still holds for the background metric $\hat g$.
Next, we note that the other components of the metric $g_\varrho $ are independent of $\varrho$ except for the basepoint change caused by the reparametrization $\Phi$.
Since $|\partial_\varrho\phi_{\varrho_1,\varrho}(\tau)|\lesssim \varrho^{1/4}$, and since the geometry of both $M$ and $\calS$ are bounded, the corresponding $\varrho$-derivatives of these remaining components are negligible small.
Moreover, recall that the conformal factor $ u_{\kappa}(x,t) = 1 + t \tilde\kappa \eta\left(\alpha^{-1}t-1\right)\le 1+\varrho_1$ is fixed and again that $|\partial_\varrho\phi_{\varrho_1,\varrho}(\tau)|\ll 1$.
Hence, the result follows.
Finally, we remark that the above estimates still hold in view of the smoothing construction performed at the end of \cref{sec:euclidean-handles}.
\end{proof}

Finally, we deform the metric $e^{\varrho_1}\cdot g_{\kappa,\varrho_1}$ to a metric which is independent of $\kappa$ by linear interpolation of conformal factors:
\begin{proposition}\label{prop:final-path-of-metrics}
    For $\varrho_1>0$ and every $\kappa\in(0,\varrho_1]$, define the path $g_s$, $s\in[0,1]$ of metrics on $\Sigma_{\varrho_1}$ to be induced by the family 
    \[s\mapsto e^{\varrho_1}\bigl((1-s)\cdot u + 2s\bigr)^{\frac{4}{n-2}}\cdot g_{\varrho_1}\]
    of metrics on $M\times [0,1]$, where $u=(1+t\tilde\kappa\eta(\alpha^{-1}t-1))$ as in \cref{theorem:handle-properties}.
    Then, $g_s$ is strictly monotonically increasing in $s$ and satisfies
    \[\scal(\Sigma_{\varrho_1},g_s)\ge -(|\scal(M, g_M)\circ\pr_M| + \varrho_1^{1/4}).\]
\end{proposition}

\begin{proof}
    Since $u_s=(1-s)\cdot u + 2s$ is the linear path and $u\le2$, the monotonicity of $g_s$ follows immediately.
    Concerning the scalar curvature estimate, we recall the conformal transformation formula for scalar curvature we have
   \begin{align*}
        \scal(\Sigma_{\varrho_1},e^{\varrho_1} u_s^{\frac4{n-2}}\cdot g_{\varrho_1})
            ={}& e^{-\varrho_1}u_s^{-\frac{n+2}{n-2}}\left(\scal(\Sigma_{\varrho_1},g_{\varrho_1})u_s - \frac{4(n-1)}{n-2}(\Delta_{g_{\varrho_1}}u_s)\right)\\
            ={}& e^{-\varrho_1}u_s^{-\frac{n+2}{n-2}}\Bigl((1-s)\underbrace{\left(\scal(\Sigma_{\varrho_1},g_{\varrho_1}) u - \frac{4(n-1)}{n-2}(\Delta_{g_{\varrho_1}}u)\right)}_{= u^{(n+2)/(n-2)}\scal\left(\Sigma_{\varrho_1}, g_{\varrho_1,\kappa}\right)}\\
                &\qquad\qquad\qquad\qquad\qquad+ 2s\cdot\scal(\Sigma_{\varrho_1},g_{\varrho_1})\Bigr).
    \end{align*}
    Note that 
    \[\scal(\Sigma_{\varrho_1},g_{\varrho_1,\kappa})\ge\scal(M, g_M)\circ\pr_M-\varrho_1^{1/4}\ge -|\scal(M, g_M)\circ\pr_M| - \varrho_1^{1/4}.\]
    Furthermore, the proof of \cref{theorem:handle-properties} reveals that the same inequalities hold for $\scal(\Sigma_{\varrho_1},g_{\varrho_1})$ as well, see \eqref{eq:scalar-curvature-estimate-region1}, \eqref{eq:scalar-curvature-estimate-region2} and \eqref{eq:scalar-curvature-estimate-region3}. 
    Thus, we obtain
    \begin{align*}
        \scal(\Sigma_{\varrho_1},&e^{\varrho_1}\cdot u_s^{\frac4{n-2}}\cdot g_{\varrho_1})\\
            \ge{}& - e^{-\varrho_1}u_s^{-\frac{n+2}{n-2}}\underbrace{\left( (1-s)\cdot u^{\frac{n+2}{n-2}} + s\cdot 2 \right)}_{\le ((1-s)\cdot u + s\cdot 2)^{{(n+2)}/{(n-2)}}}(|\scal(M, g_M)\circ\pr_M| + \varrho_1^{1/4})\\
            \ge{}& -(|\scal(M, g_M)\circ\pr_M| + \varrho_1^{1/4})\qedhere
    \end{align*}
\end{proof}

\begin{remark}\label{rem: scalar curvature is great}
    If the scalar curvature of $g_M$ is positive, we can arrange for $\scal(g_s)$ to be positive in \cref{prop:final-path-of-metrics}: If $\varrho_1$ is chosen such that $\scal(M, g_M)\circ\pr_M-\varrho_1^{1/4}>0$, then $\scal(\Sigma_{\varrho_1},g_{\varrho_1})$ and $\scal(\Sigma_{\varrho_1},g_{\varrho_1,\kappa})$ can be bounded by $\scal(M, g_M)\circ\pr_M-\varrho_1^{1/4}>0$. 
    Since $(1-s)\cdot u^{\frac{n+2}{n-2}} + s\cdot 2\ge u_s$, we obtain 
    \[\scal(\Sigma_{\varrho_1},g_s)\ge e^{-\varrho_1} u_s^{-\frac{4}{n-2}} (\scal(M, g_M)\circ\pr_M-\varrho^{1/4})>0.\]
\end{remark}

\begin{remark}
    \cref{theorem:handle-properties} has an analogue for initial data sets $(M,g_M,k)$, if we replace the scalar curvature by the dominant energy scalar $\mu-|J|$, where
    \[\mu=\frac12(\scal_{g_\Omega}+\tr(k)^2-|k|^2),\qquad J=\operatorname{div}(k-\tr(k)g).\]
    That is, if we extend $k$ trivially to $M\times[0,1]$ and multiply it by a cutoff-function supported on $[0,2a_\varrho]$, the dominant energy scalar of the hypersurface $\Sigma_\varrho$ only decreases by $\varrho^{1/4}$.
    If one goes through the proof, one encounters additional error terms for $\tfrac{1}{2}(\tr_{g_M}(k)^2-|k|^2)$, $J$ and $\tr_{g_\varrho}(k)$, which have exactly the same type as the error terms already appearing for $\scal$ and $H$, see \cref{remark error terms}.
\end{remark}

\section{Non-spin fill-ins with non-negative mean curvature}\label{sec:handle attachement}

In this section we will prove \cref{thm:mainB} by induction on the number of surgeries required to transform $M$ into a sphere.
The following is the main result needed for the induction step.

\begin{proposition}\label{thm:surgery-step}
    Let $(M,g_M)$ be a closed oriented Riemannian  manifold. 
    Let  $\widetilde{M}$ be a closed oriented manifold such that $\widetilde M$ is obtained from $M$ by performing an oriented surgery in codimension at least $3$.
    Then there exists 
\begin{itemize}
    \item  a Riemannian metric $g_{\widetilde M}$ on $\widetilde M$, 
    \item a smooth bordism $T_1$ from $M$ to $\widetilde M$,
    \item a constant $C \le 0$,
    \end{itemize}
    with the following property: 
    Let $(\Omega,g_{\Omega})$ be a fill-in of $(M, g_M)$ with $\scal_{g_{\Omega}} \geq \sigma$ and $H_{g_{\Omega}} \geq 0$.
    Then there exists a smooth Riemannian metric $g_{\Omega_1}$ on 
    \[
      \Omega_1 = \Omega \bigcup_{\partial \Omega \approx M} T_1
    \]
    such that the following holds: 
\begin{itemize}
 \item $g_{\Omega_1} = g_{\Omega}$ outside an arbitrarily small neighborhood of $T_1 \subset \Omega_1$, 
 \item $g_{\Omega_1}$ induces the metric $g_{\widetilde M}$ on $\partial \Omega_1 = \widetilde M \times \{1\}$,
 \item $H_{g_{ \Omega_1}} > 0$ along $\widetilde M \times \{1\}$, 
 \item $\scal_{g_{\Omega_1}} \geq \min\{\sigma, C\}-1$, 
 \item $
   \int_{\partial \Omega} H_{g_\Omega} {\rm dvol}_{g_{\partial \Omega}} < 2 \int_{\partial \Omega_1}H_{g_{ \Omega_1}} {\rm dvol}_{g_{\partial \Omega_1}}$. 
\end{itemize}
\end{proposition}

\begin{proof}
    Let $\psi\colon S^{n-k}\times D^k\embeds M$, $k \geq 3$, be the surgery datum for passing from $M$ to $\widetilde M$. 
    Let $\mathcal{S} := \psi(S^{n-k} \times 0) \subset M$ be the corresponding surgery sphere.
    The trace of $\psi$ is defined as
    \[
    T = T(\psi)\coloneqq (M\times[0,1])\cup_{\psi} (D^{n-k+1}\times D^k) . 
    \]
    It yields an oriented bordism from $M$ to  $\widetilde M$. 
    For simplicity, we will suppress the map $\psi$ from now on in our notation.

    \medskip

    Choose a metric $h$ on $T$ that restricts to $g_M+\dt^2$ on $M\times[0,1]\subset T$.
    The metric $h$ stays fixed for the remainder of the proof.
    Choose $R \in (0,  \frac{2}{\sqrt{10}}]$ smaller than the normal injectivity radius of $h$ along the core of the surgery
    \[
        \calS \times [0,1]  \;  \bigcup \; D^{n-k+1}\times\{0\} \subset T.
    \]
    In particular, we assume that $T$ contains the tubular $R$-neighborhood of its core.

    \medskip
 
    Let us choose $\alpha\in(0,\tfrac{1}{3}]$ and let $\varrho_1\in(0,\min(R^2/4, \alpha/2))$ be the first constant obtained in \cref{theorem:handle-properties} for  $\mathcal{S} \subset (M, g_M)$.
    For $0 < \varrho \leq \varrho_1$, let $B_{\varrho}(D^{n-k+1} \times 0) \subset T$ be the $\varrho$-tubular neighborhood of $D^{n-k+1} \times 0$ in $T$ and let $S_{\varrho}(D^{n-k+1} \times 0) \subset B_{\varrho}(D^{n-k+1} \times 0)$ be the normal sphere bundle of radius $\varrho$ around $D^{n-k+1} \times 0$. 
    Let $h_{S_{\varrho}}$ be the metric on $S_{\varrho}(D^{n-k+1} \times 0)$ induced by $h$.

    \medskip

    By decreasing $\alpha$ and hence $\varrho_1$, we may assume that for every $0 < \varrho\le \varrho_1$, the metric $h_{S_\varrho}$ has positive scalar curvature and $S_{\varrho} (D^{n-k+1} \times 0) \subset B_{\varrho}(D^{n-k+1} \times 0)$ has positive mean curvature.
    For $0 < \varrho \leq \varrho_1$, let 
    \begin{equation} \label{eqn:rad_dil}
        \Psi_{\varrho_1, \varrho} \colon S_{\varrho^4_1}(D^{n-k+1} \times 0) \to S_{\varrho^4}(D^{n-k+1} \times 0)
    \end{equation}
    be the diffeomorphism induced by radial dilatation in the normal bundle of $\mathcal{S}$. 
    By further decreasing $\alpha$, we can assume that the smooth path $(0, \varrho_1] \ni \varrho \mapsto \Psi_{\varrho_1, \varrho}^*(h_{\varrho^4})$ satisfies
    \[
        \partial_{\varrho} \Psi_{\varrho_1, \varrho}^*(h_{\varrho^4}) \geq -   \Psi_{\varrho_1, \varrho}^*(h_{\varrho^4}).
    \]
    This implies that the smooth path
    \begin{equation} \label{eqn:dil_in_core}
        \varrho \mapsto e^{ \varrho} \cdot \Psi_{\varrho_1, \varrho}^*(h_{\varrho^4})
    \end{equation}
    is strictly monotonically increasing.

    \medskip

    Let $\sigma \in \bbR$ and let $(\Omega,g_\Omega)$ be a  fill-in of $(M,g_M)$ with $\scal(g_\Omega)\ge \sigma$ and $H_{g_{\Omega}}\ge 0$ on $\partial \Omega$.
    By a small deformation of $g_{\Omega}$ near $\partial \Omega \subset \Omega$, leaving the induced metric on $\partial \Omega$ unchanged and replacing $\sigma$ with $\sigma - \tfrac{1}{2}$, we can assume that $H_{g_{\Omega}} > 0$ along $\partial \Omega$, see \cite[Proposition 3.8]{BaerHanke2}.
    By compactness of $\partial \Omega$, there exists a constant $\kappa = \kappa(\Omega,g_\Omega) > 0$ such that $H_{g_{\Omega}} \ge \kappa$.

\medskip

  We choose $\varrho_0=\varrho_0(\kappa)$ as the second constant from \cref{theorem:handle-properties}.
    By further decreasing $\varrho_0$ we may assume that $(1-\alpha)\kappa>\varrho_0^{1/4}$.
    By choosing $\varrho_0$ even smaller, we may assume that 
    \begin{equation}\label{eq:total-mean-curvature-on-3rho0}
        \int_{\partial\Omega} H_{g_\Omega}\dvol_{g_{\partial\Omega}}\le 2\int_{\partial\Omega\setminus B_{4\sqrt{\varrho_0}}(\mathcal{S})} H_{g_\Omega}\dvol_{g_{\partial\Omega}}.
    \end{equation}
    Note that $\varrho_1$ does not depend on the fill-in $(\Omega, g_{\Omega})$, but $\varrho_0$ does.

    \medskip

    For the given  $\kappa$, consider the metric $g_\kappa$ on $M\times[0,1]$ defined in \cref{theorem:handle-properties},
    \[
    g_{\kappa}=u^{\frac{4}{\,n-2\,}}(g_M+dt^2),\qquad u(x,t) = u_{\kappa}(x,t) = 1 + t \widetilde \kappa \eta\left(\alpha^{-1}t-1\right).
    \]
    Recall that $H_{g_\kappa} = - \kappa$ along $M \times \{0\}$ with respect to the exterior normal.
    By \cref{theorem:handle-properties} and our choice of $\varrho_0$, we have
    \[H(\Sigma_{\varrho_0}\cap (M\times[0,1]),g_\kappa) >(1-\alpha)\kappa - \varrho_0^{1/4}>0. \]
    Since $g_\kappa$ restricted to $M\times[\alpha, 1]$ equals $g_M+\dt^2$, the union of $g_{\kappa}$ and $h|_{T\setminus (M\times[0,1])}$ defines a smooth metric $g_T$ on $T$. 
    Here, $h$ is the fixed metric on  $T$ chosen before.
    Furthermore, as $g_\kappa \stackrel{\kappa \to 0}{\longrightarrow} g_M+\dt^2$ in the $C^\infty$-topology,  $\scal(T,g_T)$ is bounded from below by a constant $C(M, g_M, h)$, which is independent of $\kappa$.

    \medskip
    
    For $0 < \varrho\le \varrho_1$ we now define the hypersurface $\Sigma_\varrho\subset T$ as follows: The restriction $\Sigma_\varrho\cap (M\times[0,1])$ is given by the manifold $\Sigma_{\varrho} \subset M \times [0,1]$  constructed in the beginning of \cref{sec:handles}.
    Since $a_{\varrho} < 2\varrho \leq \alpha
    $, we see that  $\Sigma_\varrho\cap \left(M\times[\alpha,1]\right)  = S_{\varrho^4}(\mathcal{S}) \times [\alpha,1]$, where $S_{\varrho^4}(\mathcal{S}) \subset M$ is the normal sphere bundle of radius $\varrho^4$ with respect to the metric $g_M$.
    We extend this by the normal sphere bundle $S_{\varrho^4}(D^{n-k+1} \times 0)$ around the upper part $D^{n-k+1}\times\{0\} \subset T$.
    Let $\Sigma_\varrho$ be  the resulting smooth manifold  and denote the induced Riemannian metric by $g_{\varrho,\kappa}$, see \cref{fig:sigma_rho_construction}.
    Since $g_T$ is given by the union of $g_\kappa$ and $h$, our choice of $\varrho_1$ and \cref{theorem:handle-properties} ensures that the scalar curvature of $(\Sigma_\varrho,g_{\varrho,\kappa})$ is bounded below by 
    \begin{equation} \label{eqn:scalbound}
    \scal(M,g_M)-\varrho_{1}^{1/4}
    \end{equation}
    for all $0 < \varrho\leq \varrho_1$ and that the mean curvature of $\Sigma_{\varrho_0}$ is positive.
    
    \begin{figure}[ht]
        \scalebox{0.8}{
            \begin{tikzpicture}
                \node at (0,0) {\includegraphics[width=.8\textwidth]{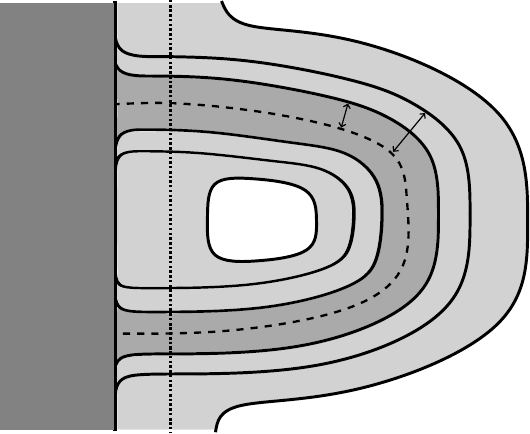}};
                % \grid
                \node at (-4,2) {$\Omega$};
                \node(2) at (-4,4.5) {$M$};
                \draw[-stealth, bend right=30] (2) to (-3, 3.8);
                \node (0) at (1,4.5) {${\widetilde M}$};
                \draw[-stealth, bend left=15] (0) to (1.3, 3.6);
                \draw[-stealth, bend right=20] (0) to (0, 0.8);
                \node at (1.8,1.85) {\small$\varrho_0^4$};
                \node (1) at (2.2,4) {$\Sigma_{\varrho_0}$};
                \draw[-stealth, bend left=15] (1) to (2.1, 2.2);
                \draw[-stealth, bend right=20] (1) to (1.1, 1.4);
                \node at (3.1,1.7) {\small$\varrho^4$};
                \node (2) at (4.5,3.3) {$\Sigma_{\varrho}$};
                \draw[-stealth, bend right=20] (2) to (1.6, 0.7);
                \draw[-stealth, bend left=20] (2) to (3.6, 1.7);
                \node at (-2.3,4.4) {$\overbrace{\qquad\quad\ }^{\ }$};
                \node at (-2.3,4.8) {$M\times[0,1]$};
                \node at (-1.1,3.6) {$T$};
                \node at (-1,2.4) {$T_0$};
                
                \node at (-4,-4.5) {$g_\Omega$};
                \node at (-1.8,-4.5) {$g_T$};
                \node at (-1.2,-3.5) {$h$};
                \node at (-2.3,-3.5) {$g_\kappa$};
    
                \node (3) at (3,-4.2) {$\begin{matrix}
                    D^{n-k+1}\times\{0\}\\
                    \text{Core of the handle }T
                \end{matrix}$};
                \draw[-stealth, bend right=20] (3) to (2.3, -1.5);
            \end{tikzpicture}
        }
        \caption{Schematic depiction of the quantitative surgery construction.}\label{fig:sigma_rho_construction}
    \end{figure}

    \medskip

    By a smooth deformation of $g_\Omega$ we can assure that for some $\delta >0$, the restriction of $g_\Omega$ to the tubular neighborhood $B_{2\sqrt{\varrho_0}}(\mathcal{S}) \times (-\delta, 0] \subset \Omega$ of $B_{2\sqrt{\varrho_0}}(\mathcal{S}) \subset \partial \Omega$ can be smoothly extended by the metric $g_\kappa$ onto $B_{2\sqrt{\varrho_0}}(\mathcal{S})\times(-\delta,1]$. 
    This can be achieved without altering the induced metric $g_{\varrho, \kappa}$ on $\Sigma_{\varrho}$, $0 < \varrho \leq \varrho_1$ or the mean curvature of $g_{\Omega}$ on $M \setminus B_{4\sqrt{\varrho_0}}(\mathcal{S})$, while preserving $H_{g_\Omega} \geq \kappa$. 
    This deformation is described in detail in \cref{sec:fixed lower H bound}, below \eqref{eq:pigeonhole}, where $U_i\subset V_i\subset Z_i$ correspond to $B_{2\sqrt{\varrho_0}}(\mathcal{S})\subset B_{3\sqrt{\varrho_0}}(\mathcal{S})\subset B_{4\sqrt{\varrho_0}}(\mathcal{S})$. 
    
    \medskip

    The manifold $\Sigma_{\varrho_0}$ thus bounds an oriented manifold of the form $\Omega_0 = \Omega\cup T_{0}$, which is diffeomorphic to $\Omega\cup T$ and $g_T$ restricts to a smooth metric $g_{\Omega_0}$ on $\Omega_0$. 
    Hence, we obtain a  fill-in $(\Omega_{0}, g_{\Omega_0})$ of $\Sigma_{\varrho_0}$.
    Since on $M \setminus B_{4\sqrt{\varrho_0}}$ the metrics $g_M$ and $g_{\varrho_0, \kappa}$ agree, since $\Sigma_{\varrho_0} \subset (\Omega_0, g_{\Omega_0})$ has positive mean curvature and using \eqref{eq:total-mean-curvature-on-3rho0}, we have 
     \begin{equation} \label{eqn:upper_bound_cool}
    \int_{\partial\Omega} H_{g_\Omega}\dvol_{g_{\partial\Omega}}\le 2\int_{\partial \Omega_0} H_{g_{\Omega_0}}\dvol_{g_{\partial\Omega_0}}.
    \end{equation}
    
    In the next step, the right hand side will be bounded above by an application of \cref{prop:criterion-hssw}.
     Let us abbreviate $\Sigma\coloneqq \Sigma_{\varrho_1}$ and consider the metric 
    \[g_\Sigma\coloneqq e^{\varrho_1} \cdot 2^{4/(n-2)}\iota^*\big((g_M+\dt^2)\cup h\big),\]
    where $\iota\colon\Sigma\embeds T$ denotes the inclusion.
    Note, that $g_\Sigma$ only depends on $(M,g_M)$ and $\psi$ and that $\Sigma \approx \widetilde M$.
    
    \medskip 
    
    We claim that $g_{\varrho_0,\kappa}\in\mathscr{R}_{g_\Sigma,s}(\Sigma)$ (see \cref{prop:criterion-hssw}) for 
    \[s=-\max(|\scal(M,g_M)|+\varrho^{1/4})\ge -\max(|\scal(M,g_M)| -1 . 
    \]
    There exists a smooth path of diffeomorphisms $\Phi_{\varrho}:\Sigma_{\varrho_1}\to\Sigma_\varrho$, $\varrho_0 \leq \varrho \leq \varrho_1$, such that $\Phi_{\varrho_1}=\id$, $\Phi_{\varrho}|_{M\setminus B_{2\sqrt{\varrho_1}}}=\id$ and such that the path $\varrho\mapsto e^\varrho\Phi_{\varrho}^*g_{\varrho,\kappa}$ is strictly increasing.
    On $\Sigma_{\varrho_1}\cap (M\times[0,1])$ these diffeomorphisms are constructed in \cref{prop:monotonicity-of-arbitrary-handles}. 
    Since the restrictions of these diffeomorphisms to $\Sigma_{\varrho_1}\cap (M\times[\tfrac{3}{4},1])$ dilate the normal sphere bundle around $\calS \subset M$ from radius $\varrho_1^4$ to $\varrho^4$ (this follows form the choice of $\beta$ in the proof of \cref{prop:monotonicity-of-arbitrary-handles}), we can take  the union with the dilatations $\Psi_{\varrho_1, \varrho}$ from \eqref{eqn:rad_dil} to define  $\Phi_{\varrho_1, \varrho}$ on $\Sigma_\varrho$. 
    By construction, we get $\Phi_{\varrho_1}=\id$, and the path $\varrho\mapsto e^\varrho\Phi_{\varrho}^*g_{\varrho,\kappa}$ is strictly  increasing by \cref{prop:monotonicity-of-arbitrary-handles} and the choice of $\alpha$, see \eqref{eqn:dil_in_core}.
    This path ends at the metric $e^{\varrho_1} g_{\varrho_1, \kappa}$ on $\Sigma$.
    Finally, we 
    \begin{itemize}
        \item rescale the conformal factor  in $e^{\varrho_1} \cdot u^{\tfrac{4}{n-2}} (g_M + dt^2)$, which induces the metric $e^{\varrho_1} g_{\varrho_1 , \kappa}$ on $\Sigma  \cap (M \times [0,1])$, along the path in \cref{prop:final-path-of-metrics}, 
        \item rescale  the conformal factor  in $e^{\varrho_1} \cdot h$, which induces the metric $e^{\varrho_1} g_{\varrho_1 , \kappa}$ on $\Sigma  \setminus (M \times [0,1])$, along the path $s \mapsto e^{\varrho_1}(1+s)^{\tfrac{4}{n-2}}$. 
    \end{itemize}
    This results in a strictly increasing smooth path of metrics on $\Sigma$ starting at $e^{\varrho_1} g_{\varrho_1, \kappa}$ and ending at $g_{\Sigma}$, while maintaining the uniform lower curvature bound $s=-\max(|\scal(M,g_M)|+\varrho^{1/4})$.
     This proves our claim that $g_{\varrho_0,\kappa}\in\mathscr{R}_{g_\Sigma,s}(\Sigma)$.
   Let $g_{\Omega_1}$ be the smooth Riemannian metric on $\Omega_1 = \Omega_0 \cup_{\widetilde M} (\widetilde M \times [0,1])$
   constructed in  \cref{prop:criterion-hssw}.
   We have 
   \[
    \scal_{g_{\Omega_1}} \geq \min \{\sigma, s, C(M,g_M,h)\} - 1.
   \]
Put $T_1 = T_0\cup (\widetilde M\times[0,1])$ and $C := \min \{0, s, C(M, g_M, h)\}$. 
   We obtain
\[
    \int_{\partial \Omega_0} H_{g_{\Omega_0}}\dvol_{g_{\partial \Omega_0}} < \int_{\partial \Omega_1} H_{g_{\Omega_1}} \dvol_{g_{\partial \Omega_1}}.
\]
Hence the claim follows from \eqref{eqn:upper_bound_cool}.
\end{proof}

\begin{proof}[Proof \cref{thm:mainB}] \label{proof:mainB}
Let $(M,g_M)$ and $\sigma \in \bbR$ be as in \cref{thm:mainB}.
Put $n := \dim M$.
As before, we can assume $n \geq 5$.
Since $M$ carries only finitely many spin structures, we can work with a fixed spin structure on $M$.

\medskip 

Since each fill-in of $M$ induces a fill-in of the spin manifold $M \sqcup \overline M$, where $\overline M$ carries the reverse spin structure, we can furthermore assume that the spin manifold $M$ is spin null-bordant. 
Let $B$ be a compact spin manifold with $\partial B = M$.

\medskip

Removing a small open disc from the interior of $B$, we obtain a compact spin bordism $W$ from $M$ to $S^n$.
Annihilating $\pi_0(W)$, $\pi_1(W)$ and $\pi_2(W)$ by spin surgeries in the interior of $S$ (using $n \geq 5$ and that $W$ is spin), we may assume that the inclusion $S^n \hookrightarrow \partial W \subset W$ is $2$-connected. 
Since $n \geq 5$, this implies that $(W,M)$ has a relative handle decomposition with handles of codimension at least $3$.
Hence, the sphere $S^n$ is obtained from $M$ by oriented surgeries of codimension at least $3$.

\medskip
    
The proof of \cref{thm:mainB} is by induction on the number $\ell$ of surgeries required. 
The case $\ell=0$, that is $M$ is diffeomorphic to a sphere, is proven in \cite{Shi-Wang-Wei}.
For $\ell\ge 1$, we note that after performing the first oriented surgery on $M$, we obtain a spin manifold $\widetilde{M}$  which can be turned into the sphere by $\ell-1$ oriented surgeries of codimension at least $3$.
Hence, by the inductive assumption, for every Riemannian metric $g_{\widetilde{M}}$ on $\widetilde M$, \cref{thm:mainB} holds for $(\widetilde{M},g_{\widetilde{M}})$.
Applying \cref{thm:surgery-step}, \cref{thm:mainB} holds for $(M,g_M)$.
\end{proof}

\begin{remark} \label{rem:GL} The construction in \cref{thm:surgery-step} allows to impose quantitative restrictions onto the classical surgery step of Gromov--Lawson \cite{GromovLawson1980}.
More specifically, let $\varepsilon > 0$, let $\sigma \in\bbR$, and let $(M,g_M)$ be a Riemannian manifold satisfying $\scal_g > \sigma$.
Let $\widetilde M$ be obtained from $M$ by surgery along an embedding $\psi \colon S^{n-k} \times D^k \hookrightarrow M$, $k \geq 3$.
Let $\breve M := M \setminus \mathrm{im}(\psi)$.
Choosing $\alpha$ in the proof of \cref{thm:surgery-step} small enough, working with a metric $h$ on $T$ of sufficiently small volume and diameter, and working with $M \times [0,\alpha]$ instead of $M \times [0,1]$, we obtain a Riemannian metric $g_{\widetilde M}$ on $\widetilde M$ with the following properties (compare \cref{eqn:scalbound}): 
\begin{itemize}
 \item $\scal_{g_{\widetilde {M}}} > \sigma$, 
 \item the restriction of $g_{\widetilde M}$ to $\breve M$ is equal to $g_M$, 
 \item the volume and diameter of the restriction of $g_{\widetilde M}$ to $\widetilde M \setminus \breve M$ are  bounded above by $\varepsilon$.
\end{itemize}
\end{remark}

\section{A new quasi-local mass}\label{sec:quasi-local}
Our results  have  applications in Physics. 
In 1982, Penrose \cite{Penrose1982Unsolved} listed important open problems in General Relativity. 
Finding a suitable notion of quasi-local mass was listed first.

\medskip

In Newtonian gravity, the mass density function $\rho$ satisfies Poisson's equation
\[
\Delta V = 4\pi \rho,
\]
where $V$ denotes the gravitational potential.  
The total mass contained in a region $\Omega$ is then given by
\[
m(\Omega) = \int_\Omega \rho \, d\mu = \frac{1}{4\pi} \int_{\partial \Omega} \nabla_\nu V \, d\mu,
\]
where $\nabla_\nu$ denotes the outward normal derivative.  
However, in general relativity, the geometric nature of spacetime makes it impossible to express the mass as an integral of a local density depending only on the metric and its derivatives. 
To overcome this difficulty, two main approaches have been developed:

\medskip

The first is to define the \emph{total mass} for an isolated system, namely the ADM mass, associated with an asymptotically flat manifold $(N, g_N)$.  
Here, asymptotically flat initial data sets $(N,g_N,k)$ model isolated gravitational systems, such as stars or galaxies. 
The mass $m(N,g_N,k)$ is well-defined and enjoys many desirable properties such as positivity and rigidity.  
However, it has the drawback that it applies only to entire non-compact manifolds, rather than to bounded regions within them.  
For instance, if $(N,g_N,k)$ models two black holes contained in its interior, the ADM mass only measures the total mass of the system, without distinguishing the contributions of the individual black holes.  

\medskip

In order to define the mass of a finite region--known as \emph{quasi-local mass}--various proposals have been made, including the Hawking mass \cite{Hawking1968, HuiskenIlmanen}, the Bartnik mass \cite{bartnik1989new}, the Brown--York mass \cite{BrownYork1993, shi-tam-manifolds-with-boundary}, and the Wang--Yau mass \cite{wang2009isometric}.  
Each of these notions has attractive geometric features but also important limitations.
A striking common feature is that all of them are particularly well-suited for spherical topology:
The Hawking mass often underestimates the true mass and is most effectively used together with inverse mean curvature flow, whose monotonicity holds only for $M \cong S^2$.
The Brown--York and Wang--Yau masses rely on isometric embeddings into Euclidean or Minkowski space, and classical embedding results such as the Weyl theorem require positive Gauss curvature.
The Bartnik mass, perhaps the most mathematically natural proposal, has so far been explicitly computed only for spherical boundaries.

\medskip

This contrasts with the increasing importance of spacetimes with more complicated topology including Emparan--Reall's black ring \cite{EmparanReall2002BlackRing}, the Kottler metric \cite{LeeNeves2015Penrose}, and the Horowitz--Myers soliton \cite{BrendleHung2024SystolicHorowitzMyers}.

\medskip

Motivated by \cite{MantoulidisMiao2017} we define a new notion of quasi-local mass which has no such topological shortcomings.
For this purpose, let $(N,g_N)$ be a complete oriented manifold with $\scal_{g_N}\ge \sigma$, $\sigma\in\mathbb R$, and let $M\subset N$ be a mean convex hypersurface and let $g_M$ be the metric on $M$ induced by $g_N$.
We set
\begin{align*}
    \textbf{m}_{\sigma}(M\subseteq (N,g_N))=\Lambda(M,g_M, \sigma)-\int_M H_g {\rm dvol}_{g_M},
 \end{align*}
where we define 
\[\Lambda(M,g_M, \sigma)\coloneqq\inf\left\{\int_{\partial\Omega}H_{g_\Omega}\dvol_{\partial\Omega}\ \left|\ \begin{matrix}
    (\Omega,g_\Omega)\text{ fill-in of }(M,g_M)\\
    \text{ with } \scal_{g_{\Omega}}\ge\sigma,\ H_{g_\Omega}|_{\partial\Omega}\ge0
\end{matrix}
\right.\right\}\]
to be the \enquote{optimal} constant for which \cref{thm:mainB} holds.
This should be compared with Bray's inner mass \cite{Bray01} which also relies on fill-ins.
\cref{thm:mainB} immediately yields
\begin{corollary}
    Suppose that $\operatorname{dim}(N)\le 7$ and $M$ is spin.
    Then the quasi-local mass $\textbf{m}_\sigma(M\subseteq (N,g_N))$ is well-defined and $\textbf m_\sigma(M\subseteq (N,g_N))\ge0$.
\end{corollary}
Here $\tfrac12\sigma$ corresponds to the cosmological constant.
We point out that the ADM mass for asymptotically flat manifolds has been successfully extended to the setting $\sigma<0$ \cite{Wang2001AHmass}, though there is no analogue for $\sigma>0$ \cite{BrendleMarquesNeves2010MinOo}.
On the other hand $\textbf m_\sigma$ is well-defined and non-negative for any choice of $\sigma\in\mathbb R$.

\medskip

When $n=2$ and $M \cong S^2$, this recovers the notion introduced by Mantoulidis--Miao~\cite[Theorem~1.5]{MantoulidisMiao2017}.
As in their setting, we obtain:
\begin{corollary}
    Suppose that $n=2$, $M=S^2$ and $K(M)>0$ for the Gaussian curvature. Then
    \begin{align*}
        \textbf{m}_0(M)=m_{BY}(M)
    \end{align*}
    for the Brown--York mass.
    In particular, we have in this case $\textbf{m}_0(M)=0$ if and only if $\Omega\subseteq \mathbb R^3$.
\end{corollary}

\appendix
\crefalias{section}{appendix}

\section{Collection of auxiliary results}\label{sec:preliminaries}

This collection of elementary results is used to construct handles in Euclidean space around points.

\begin{lemma}\label{Lemma: construction of function}
Let $0 < \varrho < 1$ and consider on the interval $[\varrho^{4},1]$ the function
\[
\xi_\varrho (r)
=-\frac{4}{3}\varrho \left(\sqrt{r}-\varrho^{2}\right)^{3/2}
-4\varrho^{3}\sqrt{\sqrt{r}-\varrho^{2}}.
\]
Then the derivatives satisfy
\begin{align*}
\xi_\varrho'(r)&=-\frac{\varrho}{\sqrt{\sqrt{r}-\varrho^{2}}},\\[4pt]
4\xi_\varrho''(r)&=-\frac{1}{r}\,\xi_\varrho'(r)\left(1+(\xi_\varrho'(r))^{2}\right).
\end{align*}
The inverse function $\xi_\varrho^{-1}$ is given by
\begin{align*}
\xi_\varrho^{-1}(t)
=\left(\varrho^{2}+4\varrho^{2}B^{2}\right)^{2},
\qquad 
B=\sinh\left(\frac{1}{3}\operatorname{arsinh}\left(-\frac{3t}{8\varrho^{4}}\right)\right),
\end{align*}
and its derivatives are
\begin{align*}
\dot{\xi}_\varrho^{-1}(t)&=-2B,\\
\ddot{\xi}_\varrho^{-1}(t)&=\frac{1}{4\varrho^{4}(1+4B^{2})}.
\end{align*}
\end{lemma}

Here $'$ and $''$ denote the derivatives $\partial_r$ and $\partial_r\partial_r$.
Similarly, $\dot{}$ and $\ddot{}$ denote the derivatives $\partial_t$ and $\partial_t\partial_t$.

\begin{proof}
A direct computation yields the derivative identities for $\xi_\varrho$.
For the inverse, we set $r=\varrho^{4}(1+4B^{2})^{2}$ and compute $$\xi_\varrho(r)=-\frac{32}{3}\varrho^{4}B^{3}-8\varrho^{4}B.$$
Hence, the triple angle formula 
\[
\sinh(3u)=3\sinh(u)+4\sinh^{3}(u)
\]
implies the claim that $\xi_\varrho^{-1}$ inverts $\xi_\varrho$.
Differentiating $B(t)=\sinh\left(\frac{1}{3}\operatorname{arsinh}\left(-\tfrac{3t}{8\varrho^{4}}\right)\right)$ gives for $u=\left(\frac{1}{3}\operatorname{arsinh}\left(-\frac{3t}{8\varrho^{4}}\right)\right)$
\begin{align*}
  \dot B(t)=&-\frac1{8\varrho^4}\cosh(u)\frac1{\sqrt{1+\left(\tfrac{3t}{8\varrho^{4}}\right)^2}}\\
  =&-\frac1{8\varrho^4}\cosh(u)\frac1{\sqrt{1+\sinh^2(3u)}}\\
  =&-\frac1{8\varrho^4}\cosh(u)\frac1{\cosh(3u)}  .
\end{align*}
Using the triple angle formula $$\cosh(3u)=4\cosh^3u-3\cosh u=\cosh(u)(4(1+\sinh^2 u)-3),$$ we obtain
\[
\dot B(t)=-\frac{1}{8\varrho^{4}(1+4B^{2})},
\]
from which the derivative identities for $\dot{\xi}_\varrho^{-1}$ and $\ddot{\xi}_\varrho^{-1}$ follow directly.
\end{proof}

\begin{remark}
  The function $B$ can also be expressed more explicitly with the help of the formula
  \[
 2 \sinh\left(\tfrac{1}{3}\operatorname{arsinh}u\right)
  =\left(u+\sqrt{u^{2}+1}\right)^{1/3}
  -\left(\sqrt{u^{2}+1}-u\right)^{1/3}
  \]
  which follows from repeated usage of the triple angle formulas.
\end{remark}

The next result is well-known and we will omit its proof.

\begin{lemma}\label{Lemma: spherical symmetry principal curvatures}
Let $\ell:\mathbb R^{n}\to\mathbb R$ be a spherically symmetric function, i.e. $\ell(x)=\ell(r)$ where $r=|x|$. 
Then the principal curvatures of the graph of $\ell$ are
\[
\lambda_1=-\frac{\ell''}{(1+\ell'^{2})^{3/2}},
\qquad
\lambda_2=\dots=\lambda_{n}=-\frac{\ell'}{r\sqrt{1+\ell'^{2}}}
\]
where $\lambda_1$ is the principal curvature in radial direction.
\end{lemma}

\section{Example of unbounded total mean curvature}\label{sec:bad-constant}
Let $n\ge2$.
Then, for every $K>0$, there exists a metric $g_K$ on $S^n$ satisfying
\begin{itemize}
    \item $\vol(S^n,g_K)\le K^{-1}$,
    \item $\diam(S^n,g_K)\le K^{-1}$,
    \item there is a fill-in of $(S^n,g_K)$ with nonnegative scalar curvature, positive mean curvature and total mean curvature larger than $K$.
\end{itemize} 
Hence, $\Lambda(S^n,g_K,0,0)\ge K$.

\medskip

By taking connected sums with $g_K$, we can implant these examples into an arbitrary Riemannian manifold.
More precisely, let $(M,g_M)$ be a Riemannian manifold and let $(\Omega,g_\Omega)$ be a fill in with nonnegative scalar and mean curvature and 
\[\int_{\partial\Omega}H_{g_\Omega}\dvol_{g_{\partial\Omega}} = \Lambda.\]
Then, for $K>0$, there exists a metric $g_{M,K}$ on $M$ and a fill-in $(\Omega_\eps,g_{\Omega,\eps})$ of $(M,g_{M,K})$ with nonnegative scalar curvature
and positive mean curvature, which satisfies
\begin{enumerate}
    \item $g_M|_U = g_{M,K}|_U$ for a subset $U\subset M$ with $\vol(U) = \vol(M)-K^{-1}$,
    \item $\vol(M, g_{M,K}) \le \vol(M,g_M)+K^{-1}$,
    \item $\diam(M, g_{M,K}) \le {\rm diam}(M,g_M)+K^{-1}$,
    \item $\int_{\partial\Omega}H_{g_\Omega}\dvol_{g_{\partial\Omega}} = \Lambda + K$.
\end{enumerate}

To construct the metric $g_K$ on $S^n$, let $\ell\coloneqq(2c(n)K)^{2n-1}$ and $r\coloneqq(2c(n)K)^{-2}$, where $c(n)=\vol(S^n,g_\circ)$ is a dimensional constant.
By increasing $K$, we may assume that $\ell\in\bbN$ and $20r\le K^{-1}$.
We embed $\ell+1$ disjoint spheres of radius $r$ into $\bbR^{n+1}$, such that one of them is centered at the origin and the straight paths connecting a given embedded sphere with this \enquote{center sphere} does not intersect any other sphere.
We perform connected sums to connect all outer spheres to the center-sphere by straight tubes, and we may assume that any connecting tube has volume smaller than $\tfrac{1}{2\ell} K^{-1}$.
Let us denote the resulting Riemannian manifold by $(X,g_X)$.
We observe, that $X\approx S^n$.
By \cite[Theorem 3.1]{Lawson_Michelsoh}, we may assume that $X$ has positive mean curvature with respect to the euclidean metric, since $n\ge2$.
Hence, the interior of $X$ constitutes a flat fill-in $(\widetilde{\Omega},g_{\rm flat})$ with positive mean curvature. A visualisation can be found in \cref{fig:many-spheres}.

\begin{figure}[ht]
    \begin{tikzpicture}
        \node at (0,0) {\includegraphics[width=.3\textwidth]{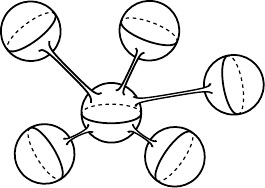}};
    \end{tikzpicture}
    \caption{The metric $g_K$ on $S^n$.}\label{fig:many-spheres}
\end{figure}

\medskip

Let us now estimate the volume and the total mean curvature of this fill-in.
For the total mean curvature, we can restrict to the respective outer hemispheres, where the mean curvature is constantly equal to $n/r$.
Hence, we obtain:
\begin{align}\label{eq:example-sphere}
    \begin{split}
        \int_{\partial\widetilde{\Omega}} H_{g_{\rm flat}}\dvol_{g_{\partial\widetilde{\Omega}}}\ge{}& \ell\int_{S^n_{r,+}} H_{g_{\rm flat}}\dvol_{g_{\partial\widetilde{\Omega}}} \ge \frac12 n c(n)\ell r^{n-1}\\
        \ge{}& \frac12 n c(n) \bigl(2c(n) K\bigr)^{2n-1}\bigl(2c(n) K\bigr)^{-2(n-1)}\ge K.
    \end{split}
\end{align}
The volume on the other hand satisfies:
\begin{align*}
    \vol(X)&\le \ell\left(\vol(S^n_r)+\frac1{2\ell} K^{-1}\right)= c(n)\bigl(2c(n) K\bigr)^{2n-1}\bigl(2c(n) K\bigr)^{-2n}+\frac12 K^{-1}\\ &\le \frac12K^{-1} + \frac12 K^{-1} = K^{-1}
\end{align*}

\medskip

Note, that while we cannot put the outer spheres onto a circle of arbitrarily small radius, we can still cut away portions of the tubes as non-embedded Riemannian manifolds, resulting in tubes of length less than $r\le \tfrac{1}{12}K^{-1}$.
The resulting manifold and its fill-in might not be isometric to a Riemannian submanifold of $(\bbR^{n+1},\delta)$, but the mean curvature will still be positive, the scalar curvature will remain flat and estimate \eqref{eq:example-sphere} is still valid, as we do not alter the outer hemispheres by this procedure.
The volume may have further decreased and the diameter of the resulting metric $g_K$ is smaller than $3\pi\cdot r+4\cdot r\le 12\cdot r\le K^{-1}$.

\printbibliography
\vspace{1em}

\end{document}